\documentclass[12pt,reqno]{amsart}%
\usepackage{amsmath, amsfonts, amssymb, amsthm, amscd, amsbsy}
\usepackage{fancyhdr}
\usepackage[usenames,dvipsnames,svgnames,x11names,hyperref]{xcolor}
\usepackage{geometry}
\usepackage{graphicx}
\usepackage{hyperref}
\usepackage{lineno}
\usepackage{amsmath}
\usepackage{amsfonts}
\usepackage{amssymb}%
\providecommand{\U}[1]{\protect\rule{.1in}{.1in}}
\hypersetup{colorlinks,breaklinks,
linkcolor={Fuchsia},
citecolor={ForestGreen},
urlcolor={NavyBlue}}
\newtheorem{theorem}{Theorem}[section]
\theoremstyle{plain}

\newtheorem{corollary}{Corollary}[section]

\newtheorem{lemma}{Lemma}[section]

\newtheorem{remark}{Remark}[section]

\numberwithin{equation}{section}
\allowdisplaybreaks
\newtheorem{theorema}{Theorem}[section]

\begin{document}
\title[Scale-dependent Poincar\'{e} inequalities and stability of HUP]{Scale-Dependent Poincar\'{e} inequalities, log-Sobolev inequality and the
stability of the Heisenberg Uncertainty Principle on the hyperbolic space}
\author{Anh Xuan Do}
\address{Anh Xuan Do: Department of Mathematics, University of Connecticut, Storrs, CT
06269, USA}
\email{anh.do@uconn.edu}
\author{Debdip Ganguly}
\address{Debdip Ganguly: Theoretical Statistics and Mathematics Unit\\
Indian Statistical Institute, Delhi Centre\\
S.J. Sansanwal Marg, New Delhi, Delhi 110016, India}
\email{debdip@isid.ac.in}
\author{Nguyen Lam}
\address{Nguyen Lam: School of Science and the Environment, Grenfell Campus, Memorial
University of Newfoundland, Corner Brook, NL A2H5G4, Canada}
\email{nlam@mun.ca}
\author{Guozhen Lu}
\address{Guozhen Lu: Department of Mathematics, University of Connecticut, Storrs, CT
06269, USA}
\email{guozhen.lu@uconn.edu}

\begin{abstract}
We establish a general scale-dependent Poincar\'{e}-Hardy type identity
involving a vector field on the hyperbolic space. By choosing suitable
parameter, potential and vector field in this identity, we can recover, as
well as derive new versions of and substantially improve several Poincar\'{e}
type, Hardy type and Poincar\'{e}-Hardy type inequalities in the literature.
We also investigate weighted Poincar\'{e} inequalities on hyperbolic space,
where the weight functions depend on a scaling parameter. This leads to a new
family of scale-dependent Poincar\'{e} inequalities with Gaussian type measure
on the hyperbolic space which is of independent interest. As a result, we
derive both scale-dependent and scale-invariant $L^{2}$-stability results for
the Heisenberg uncertainty principle in this setting. Finally, we study the
logarithmic Sobolev inequality with Gaussian measure on the hyperbolic spaces,
that is still missing in the literature.

\end{abstract}
\subjclass[2010]{}
\maketitle

\section{Introduction}

The purpose of this paper is four-fold. Firstly, we establish a general
scale-dependent Poincar\'{e}-Hardy type identity that can be used to deduce
several Hardy type inequalities, Poincar\'{e} type inequalities, and
Caffarelli-Kohn-Nirenberg type inequalities with exact remainders on the
hyperbolic space $\mathbb{H}^{N}$. Secondly, we provide some versions of the
scale-dependent Poincar\'{e} inequality with Gaussian type measures on
$\mathbb{H}^{N}$. Thirdly, by combining the Heisenberg uncertainty principle
identity on $\mathbb{H}^{N}$ and the scale-dependent Poincar\'{e} inequality
with Gaussian type measures on $\mathbb{H}^{N}$, we derive several types of
stability results for the Heisenberg uncertainty principle on $\mathbb{H}^{N}%
$. Finally, we investigate a version of the logarithmic Sobolev inequality
with Gaussian type measure, that seems to be missing in the literature still.

Throughout the paper, unless we emphasize to use a specific model of
$\mathbb{H}^{N}$, we will work with the Poincar\'{e} ball model of the
hyperbolic space $\mathbb{H}^{N}$. In particular, it is the Euclidean unit
ball $B^{N}:=\{x\in\mathbb{R}^{N}:|x|^{2}<1\}$ and equipped with the
Riemannian metric
\[
\mathrm{d}s^{2}=\left(  \frac{2}{1-|x|^{2}}\right)  ^{2}\,dx^{2}%
\]
which constitutes the ball model for the hyperbolic $N$-space, where $dx^{2}$
is the standard Euclidean metric and $|x|^{2}=\sum_{i=1}^{N}x_{i}^{2}$ is the
standard Euclidean length. By definition, the hyperbolic $N$-space is a
$N$-dimensional complete, non-compact Riemannian manifold having constant
sectional curvature equal to $-1$ and any two manifolds sharing the above
properties are isometric \cite{RAT}. In this article, all our computations
will involve only the ball model and will be denoted by $\mathbb{H}^{N}$. We
denote the inner product on the tangent space of $\mathbb{H}^{N}$ by
$\langle\cdot,\cdot\rangle_{\mathbb{H}^{N}}$ or even $\langle\cdot
,\cdot\rangle$ if there is no chance of confusion, and the volume element is
given by $dV_{\mathbb{H}}=\left(  \frac{2}{1-|x|^{2}}\right)  ^{N}dx,$ where
$dx$ denotes the Lebesgue measure on $\mathbb{R}^{N}$.

Let $\nabla_{\mathbb{H}}$ denote gradient vector field on $\mathbb{H}^{N}$. In
terms of local (global) coordinates, $\nabla_{\mathbb{H}}$ takes the form
$\nabla_{\mathbb{H}}=\left(  \frac{1-|x|^{2}}{2}\right)  ^{2}\nabla$. Here
$\nabla$ is the standard Euclidean gradient vector field. Now, for any smooth
vector field $X$ on $\mathbb{H}^{N}$, its divergence $\operatorname{div}X$ is
a smooth function on $\mathbb{H}^{N}$, and can be uniquely defined as follows:

\begin{theorema}
[The Divergence Theorem]For any smooth vector field $X$ on $\mathbb{H}^{N}$,
there exists a unique smooth function on $\mathbb{H}^{N}$, denoted by
$\operatorname{div}X$, such that the following identity holds%
\[
\int_{\mathbb{H}^{N}}\left(  \operatorname{div}X\right)  udV_{\mathbb{H}%
}=-\int_{\mathbb{H}^{N}}\left\langle X,\nabla_{\mathbb{H}}u\right\rangle
dV_{\mathbb{H}}.
\]

\end{theorema}

In the local coordinates, we can also define the divergence
$\operatorname{div}X$ by%
\[
\operatorname{div}X=\frac{1}{\sqrt{\det g}}\frac{\partial}{\partial x^{i}%
}\left(  X^{i}\sqrt{\det g}\right)
\]
where $g=\left(  g_{ij}\right)  $ is the matrix of the metric on
$\mathbb{H}^{N}$. See, for instance, \cite{Gre24, Jost}. With this, we can
define the Laplace-Beltrami operator $\Delta_{\mathbb{H}}$ by $\Delta
_{\mathbb{H}}=\operatorname{div}\circ\nabla$. In particular, in terms of local
(global) coordinates, $\Delta_{\mathbb{H}}$ takes the form:
\[
\Delta_{\mathbb{H}}=\left(  \frac{1-|x|^{2}}{2}\right)  ^{2}\Delta
+(N-2)\left(  \frac{1-|x|^{2}}{2}\right)  x\cdot\nabla,
\]
where $\nabla,\Delta$ are the standard Euclidean gradient vector field and
Laplace operator respectively, and `$\cdot$' denotes the standard inner
product in $\mathbb{R}^{N}$.

The hyperbolic distance between two points $x$ and $y$ in $\mathbb{H}^{N}$
will be denoted by $d(x,y).$ The hyperbolic distance between $x$ and the
origin can be computed explicitly
\[
\rho(x):=\,d(x,0)=\int_{0}^{|x|}\frac{2}{1-s^{2}}\,\mathrm{d}s\,=\,\log
\frac{1+|x|}{1-|x|},
\]
which implies $\left\vert x\right\vert =\tanh\frac{\rho}{2},$ and
\[
1-|x|^{2}=1-\tanh^{2}\left(  \frac{\rho(x)}{2}\right)  =\dfrac{4e^{\rho(x)}%
}{(1+e^{\rho(x)})^{2}}.
\]
Moreover, the hyperbolic distance between $x,y\in\mathbb{H}^{N}$ is given by
\[
d(x,y)=\cosh^{-1}\left(  1+\dfrac{2|x-y|^{2}}{(1-|x|^{2})(1-|y|^{2})}\right)
.
\]
As a result, a subset of $\mathbb{H}^{N}$ is a hyperbolic ball in
$\mathbb{H}^{N}$ if and only if it is a Euclidean ball in $\mathbb{R}^{N}$ and
contained in $\mathbb{H}^{N}$ possibly with a different centre and different
radius which can be explicitly computed from the formula of $d(x,y)$
\cite{RAT}. Geodesic balls in $\mathbb{H}^{N}$ of radius $r$ centred at
$x\in\mathbb{H}^{N}$ will be denoted by
\[
B_{r}(x):=\{y\in\mathbb{H}^{N}:d(x,y)<r\}.
\]

\medskip

We begin with the classical Hardy inequality in Euclidean space $\mathbb{R}%
^{N}$ which is one of the most used inequalities in analysis. It has been
studied intensively and plays crucial roles in many areas of analysis,
mathematical physics and partial differential equations. The reader is
referred to \cite{BV97, CZ13, DDT23, DLL24, DLL22, GM13, KMP07, KP03, LLZ19,
LLZ20, OK90, VZ00}, to name just a few. In particular, in an effort to
generalize the Hardy inequality to Riemannian manifolds, in \cite{C97}, Carron
demonstrated that: Let $(M,g)$ be a Riemannian manifold and let $\varrho(x)$
be a weight function satisfying the Eikonal equation $\left\vert \nabla
_{g}\varrho\right\vert =1$ and $\Delta_{g}\varrho\geq\frac{C}{\varrho}$, with
$C>0$ a positive constant, there holds
\begin{equation}
\int_{M}|\nabla_{g}u|^{2}dv_{g}\geq\left(  \dfrac{C-1}{2}\right)  ^{2}\int
_{M}\dfrac{\left\vert u\right\vert ^{2}}{\varrho^{2}}dv_{g}
\label{HardyinRiem}%
\end{equation}
for all $u\in C_{0}^{\infty}(M\setminus\varrho^{-1}\{0\})$. In a special case
when $M$ is a Cartan-Hadamard manifold of dimension $N$ which is complete,
simply-connected, and has non-positive sectional curvature everywhere, the
geodesic distance function $d(x,x_{0})$, with $x_{0}\in M$ serves as a
suitable candidate for the weight function $\varrho$. In this case, the
inequality \eqref{HardyinRiem} holds with the same best constant as in
$\mathbb{R}^{N}$. In particular, on the hyperbolic space $\mathbb{H}^{N}$,
which is one of the most important examples of Cartan-Hadamard manifolds, it
is known that
\begin{equation}
\int_{\mathbb{H}^{N}}|\nabla_{\mathbb{H}}u|^{2}dV_{\mathbb{H}}\geq\left(
\dfrac{N-2}{2}\right)  ^{2}\int_{\mathbb{H}^{N}}\dfrac{\left\vert u\right\vert
^{2}}{r^{2}}dV_{\mathbb{H}}, \label{HardyinHyper}%
\end{equation}
for all $u\in C_{0}^{\infty}(\mathbb{H}^{N}\setminus\{x_{0}\})$, with
$r:=d(x,x_{0})$ and $x_{0}\in\mathbb{H}^{N}$ is a fixed pole. For insights
into how curvature influences the improvements of the Hardy inequality, the
reader may refer to \cite{BGGP20, FLL22, KO09, KO13, RSY24, YSK14} and the
references therein.

On the hyperbolic space $\mathbb{H}^{N}$, the Poincar\'{e} inequality is
another important inequality that is closely related to the Hardy inequality.
Indeed, the hyperbolic space being a Cartan-Hadamard manifold with constant
negative sectional curvature is well-known to admit a Poincar\'{e} or $L^{2}$-
gap inequality, namely,
\begin{equation}
\int_{\mathbb{H}^{N}}\left\vert \nabla_{\mathbb{H}}u\right\vert ^{2}%
\;dV_{\mathbb{H}}\geq\dfrac{(N-1)^{2}}{4}\int_{\mathbb{H}^{N}}\left\vert
u\right\vert ^{2}\;dV_{\mathbb{H}}\quad\forall u\in C_{c}^{\infty}%
(\mathbb{H}^{N}), \label{classicPoin}%
\end{equation}
where the constant $\frac{(N-1)^{2}}{4}$ is sharp and never attained. In fact,
this inequality holds for any $L^{p}$-norm with $p>1$, and the corresponding
constant is $\left(  \frac{N-1}{p}\right)  ^{p}$, see \cite{AH19}. Focusing on
the issue of non-attainability, numerous mathematicians have worked on
refining inequality \eqref{classicPoin}. One such improvement stems from the
observation that the operator
\[
-\Delta_{p,\mathbb{H}^{N}}-\left(  \frac{N-1}{p}\right)  ^{p}=\text{div}%
\left(  \left\vert \nabla_{\mathbb{H}}\cdot\right\vert ^{p-2}\nabla
_{\mathbb{H}}\right)  -\left(  \frac{N-1}{p}\right)  ^{p}%
\]
is sub-critical on $\mathbb{H}^{N}$. This motivated us to add non-negative
terms to the right-hand side of \eqref{classicPoin} (see \cite{BAGG17, MS08,
H20}). There are also efforts to study the interpolation inequalities between
the Hardy inequality, the Poincar\'{e} inequality, and other geometric
inequalities, and their improvements on the hyperbolic space. See, for
instance, \cite{BGG17, DLP24, FLL22, KO13, LLY18, LY22a, LY19, LY22}.

\medskip

The first main goal of this paper is to set up a general Poincar\'{e}-Hardy
type identity involving a vector field on the hyperbolic space. More exactly,
we will prove that

\begin{theorem}
\label{Thm 1.1} Let $N\geq1,$ $p>1,$ $\lambda>0,$ $A\in C^{1}(\mathbb{H}%
^{N}\setminus\left\{  0\right\}  )$ and $X$ be a vector field on
$\mathbb{H}^{N}$. Then for all $u\in C_{0}^{1}(\mathbb{H}^{N}\setminus\left\{
0\right\}  )$, we have
\begin{align}
&  \lambda^{p}\int_{\mathbb{H}^{N}}A|\nabla_{\mathbb{H}}u|^{p}dV_{\mathbb{H}%
}+\dfrac{p-1}{\lambda^{p/(p-1)}}\int_{\mathbb{H}^{N}}A|uX|^{p}dV_{\mathbb{H}%
}\nonumber\\
&  =-\int_{\mathbb{H}^{N}}\operatorname{div}\left(  A|X|^{p-2}X\right)
|u|^{p}dV_{\mathbb{H}}+\int_{\mathbb{H}^{N}}A\mathcal{R}_{p}\left(  \dfrac
{uX}{\lambda^{1/(p-1)}},\lambda\nabla_{\mathbb{H}}u\right)  dV_{\mathbb{H}}.
\label{intrognridt}%
\end{align}
Here $\mathcal{R}_{p}\left(  X,Y\right)  :=|Y|^{p}+(p-1)|X|^{p}-p|X|^{p-2}%
\left\langle X,Y\right\rangle $.
\end{theorem}

$L^{p}$ identities such as \eqref{intrognridt} without the scaling factor
(that is $\lambda=1$) when $p=2$ have been established earlier in different
settings including the Euclidean spaces, hyperbolic spaces, Carnot groups and
Riemannian manifolds \cite{FLL22, FLLM23, FS08, LLZ19, LLZ20}, to name just a
few. For arbitrary $p\geq1$ without the scaling factor, it appeared initially
in \cite{DLL22} and subsequently in \cite{FLL25} in the Euclidean spaces and
then more general setting for vector fields in \cite{HY22}. Also, Theorem
\ref{Thm 1.1} with the scaling factor $\lambda$ has been set up in \cite{CFNL,
AJLL23}, for instance, in the Euclidean setting. We note that as showed in
\cite{CFNL, AJLL23}, the parameter $\lambda$ plays an important role in
studying the sharp constants and optimizers of the Caffarelli-Kohn-Nirenberg
type inequalities, which are more optimal than the Hardy type inequalities. In
the Euclidean case, to extract the scaling parameters $\lambda$ in
\eqref{intrognridt}, an important tool in the approach in \cite{AJLL23} is the
scaling/dilation arguments. Unfortunately, given that the conformal group of
$\mathbb{H}^{N}$ coincides with its isometry group, the main challenge in
proving \eqref{intrognridt} in the hyperbolic space lies in the absence of
scaling arguments. Nevertheless, our method here is to construct an identity
that incorporates the scaling parameters directly instead of using any
intermediate identity, and use it to establish \eqref{intrognridt}.

Our motivation of proving \eqref{intrognridt} is to find a method to provide
simple and straightforward proofs to the Poincar\'{e} inequality and Hardy
inequality on $\mathbb{H}^{N}$. Indeed, by choosing suitable parameter
$\lambda$, potential $A$ and vector field $X$ in \eqref{intrognridt}, we can
recover several Poincar\'{e} type, Hardy type and Poincar\'{e}-Hardy type
inequalities in the literature. Interestingly, this identity
\eqref{intrognridt} has broader applications than initially anticipated. In
this paper, we choose to present three consequences of Theorem \ref{Thm 1.1}.
Firstly, \eqref{intrognridt} provides an alternative method for proving
weighted Hardy-type inequalities on $\mathbb{H}^{N}$ (see Subsection
\ref{subsct4.1}). Moreover, if \eqref{intrognridt} is considered as an
identity in terms of $\lambda$, it yields Hardy-type identities when $\lambda$
is $1$, while optimizing $\lambda$ provides Caffarelli-Kohn-Nirenberg
identities. More precisely, we can derive the following
Caffarelli-Kohn-Nirenberg identity and inequality with Bessel pair:

\begin{theorem}
\label{M1}Let $0<R\leq\infty$, $V\geq0$ and $W$ are $C^{1}$-functions on
$(0,R)$. Assume that $(r^{N-1}V,r^{N-1}W)$ be a Bessel pair on $(0,R)$ with
its corresponding solution $\varphi>0$. Then, we have for $u\in C_{0}^{\infty
}(B_{R}\setminus\left\{  0\right\}  ):$
\begin{align*}
&  \left(  \int_{B_{R}}V(\rho(x))|\nabla_{\mathbb{H}}u|^{2}dV_{\mathbb{H}%
}\right)  ^{1/2}\left(  \int_{B_{R}}V(\rho(x))\dfrac{\left(  \varphi^{\prime
}(\rho(x))\right)  ^{2}}{\varphi^{2}(\rho(x))}\left\vert u\right\vert
^{2}dV_{\mathbb{H}}\right)  ^{1/2}\\
&  =\dfrac{1}{2}\int_{B_{R}}\left(  W(\rho(x))+V(\rho(x))\dfrac{\left(
\varphi^{\prime}(\rho(x))\right)  ^{2}}{\varphi^{2}(\rho(x))}\right)
\left\vert u\right\vert ^{2}dV_{\mathbb{H}}\\
&  -\dfrac{N-1}{2}\int_{B_{R}}V(\rho(x))\dfrac{\varphi^{\prime}(\rho
(x))}{\varphi(\rho(x))}\left(  \dfrac{\rho(x)\cosh\rho(x)-\sinh\rho(x)}%
{\rho(x)\sinh\rho(x)}\right)  |u|^{2}dV_{\mathbb{H}}\\
&  +\dfrac{1}{2}\int_{B_{R}}V(\rho(x))\left\vert \dfrac{1}{\lambda}%
u\dfrac{\varphi^{\prime}(\rho(x))}{\varphi(\rho(x))}\nabla_{\mathbb{H}}%
(\rho(x))-\lambda\nabla_{\mathbb{H}}u\right\vert ^{2}dV_{\mathbb{H}}.
\end{align*}
In particular, if $\frac{\varphi^{\prime}}{\varphi}$ is nonpositive, we deduce
a family of Caffarelli-Kohn-Nirenberg inequalities, i.e.
\begin{align*}
&  \left(  \int_{B_{R}}V(\rho(x))|\nabla_{\mathbb{H}}u|^{2}dV_{\mathbb{H}%
}\right)  ^{1/2}\left(  \int_{B_{R}}V(\rho(x))\dfrac{\left(  \varphi^{\prime
}(\rho(x))\right)  ^{2}}{\varphi^{2}(\rho(x))}\left\vert u\right\vert
^{2}dV_{\mathbb{H}}\right)  ^{1/2}\\
&  \geq\dfrac{1}{2}\int_{B_{R}}\left(  W(\rho(x))+V(\rho(x))\dfrac{\left(
\varphi^{\prime}(\rho(x))\right)  ^{2}}{\varphi^{2}(\rho(x))}\right)
\left\vert u\right\vert ^{2}dV_{\mathbb{H}}.
\end{align*}

\end{theorem}

In certain cases, deriving Hardy-type and Caffarelli-Kohn-Nirenberg
inequalities from these identities is relatively straightforward. Thus, Hardy
inequalities can be viewed as non-optimal scale invariant
Caffarelli-Kohn-Nirenberg inequalities in this context (see Subsection
\ref{subsct4.2}).

Finally, as a special case of the Caffarelli-Kohn-Nirenberg identities
discussed in Subsection \ref{subsct4.2}, when $a=-1$ and $b=0$, we obtain the
following improved Heisenberg uncertainty principle on $\mathbb{H}^{N}$:%
\begin{align*}
&  \left(  \int_{\mathbb{H}^{N}}|\nabla_{\mathbb{H}}u|^{2}dV_{\mathbb{H}%
}\right)  ^{1/2}\left(  \int_{\mathbb{H}^{N}}\rho^{2}(x)|u|^{2}dV_{\mathbb{H}%
}\right)  ^{1/2}\\
&  \geq\dfrac{1}{2}\int_{\mathbb{H}^{N}}\left(  N+(N-1)\dfrac{\rho(x)\cosh
\rho(x)-\sinh\rho(x)}{\sinh\rho(x)}\right)  \;|u|^{2}dV_{\mathbb{H}}.
\end{align*}
See also \cite{K18}. Moreover, as detailed in Subsection \ref{subsct4.3}, we
will establish the above improved Heisenberg uncertainty principle with exact
remainder and therefore obtain the following set of its optimizers:%
\[
E:=\left\{  ce^{-\alpha\rho^{2}(x)}:c\in\mathbb{R},\alpha>0\right\}  .
\]
To be exact, we will show that

\begin{theorem}
\label{M2}For $u\in C_{0}^{\infty}(\mathbb{H}^{N})\setminus\left\{  0\right\}
:$%
\begin{align*}
&  \left(  \int_{\mathbb{H}^{N}}|\nabla_{\mathbb{H}}u|^{2}dV_{\mathbb{H}%
}\right)  ^{1/2}\left(  \int_{\mathbb{H}^{N}}\rho^{2}(x)|u|^{2}dV_{\mathbb{H}%
}\right)  ^{1/2}\\
&  -\dfrac{1}{2}\int_{\mathbb{H}^{N}}\left(  N+(N-1)\dfrac{\rho(x)\cosh
\rho(x)-\sinh\rho(x)}{\sinh\rho(x)}\right)  \;|u|^{2}dV_{\mathbb{H}}\\
&  =\dfrac{\lambda^{2}}{2}\int_{\mathbb{H}^{N}}e^{-\frac{\rho^{2}(x)}%
{\lambda^{2}}}\left\vert \nabla_{\mathbb{H}}\left(  ue^{\frac{\rho^{2}%
(x)}{2\lambda^{2}}}\right)  \right\vert ^{2}dV_{\mathbb{H}}%
\end{align*}
with $\lambda=\left(  \dfrac{\int_{\mathbb{H}^{N}}\rho^{2}(x)|u|^{2}%
dV_{\mathbb{H}}}{\int_{\mathbb{H}^{N}}|\nabla_{\mathbb{H}}u|^{2}%
dV_{\mathbb{H}}}\right)  ^{1/4}.$
\end{theorem}

The above improved Heisenberg uncertainty principle and its manifold of
optimizers have inspired us to pose the following stability question for the
Heisenberg uncertainty principle on $\mathbb{H}^{N}$:

\noindent\textbf{{Question 1:}} \textit{Does there exist }$d\left(
u,E\right)  $\textit{, a suitable distance function from }$u$\textit{ to }%
$E$\textit{, such that }%
\begin{align}
\text{Heisenberg deficit}  &  :=\left(  \int_{\mathbb{H}^{N}}|\nabla
_{\mathbb{H}}u|^{2}dV_{\mathbb{H}}\right)  ^{1/2}\left(  \int_{\mathbb{H}^{N}%
}\rho^{2}(x)|u|^{2}dV_{\mathbb{H}}\right)  ^{1/2}\nonumber\\
&  -\dfrac{1}{2}\int_{\mathbb{H}^{N}}\left(  N+(N-1)\dfrac{\rho(x)\cosh
\rho(x)-\sinh\rho(x)}{\sinh\rho(x)}\right)  \;|u|^{2}dV_{\mathbb{H}%
}\nonumber\\
&  \gtrsim\;d^{2}\left(  u,E\right)  , \label{introdeficit}%
\end{align}
\textit{for }$u\in W^{1,2}(H^{N})\cap\{u:\int_{\mathbb{H}^{N}}\rho
^{2}(x)|u|^{2}dV_{\mathbb{H}}<\infty\}$\textit{?}

Here $E$ is the set of all optimizers of the Heisenberg uncertainty principle.
It is noteworthy that the above Heisenberg deficit represents an improvement
over the standard Heisenberg deficit defined by
\[
\left(  \int_{\mathbb{H}^{N}}|\nabla_{\mathbb{H}}u|^{2}dV_{\mathbb{H}}\right)
^{1/2}\left(  \int_{\mathbb{H}^{N}}\rho^{2}(x)|u|^{2}dV_{\mathbb{H}}\right)
^{1/2}-\dfrac{N}{2}\int_{\mathbb{H}^{N}}|u|^{2}dV_{\mathbb{H}},
\]
due to the positivity of the term $(\rho(x)\;\cosh\rho(x)\;-\;\sinh\rho(x))$.

Question 1 naturally arises from the work in \cite{CFNL, AJLL23}, where the
authors addressed the problem in Euclidean space $\mathbb{R}^{N}$. They
established both the optimal stability constant and its attainability,
achieving significant results. More clearly, in $\mathbb{R}^{N}$, with the
Heisenberg deficit defined as
\[
\delta_{2,\mathbb{R}^{N}}:=\left(  \int_{\mathbb{R}^{N}}|\nabla u|^{2}%
dx\right)  \left(  \int_{\mathbb{R}^{N}}|x|^{2}|u|^{2}dx\right)  -\dfrac
{N^{2}}{4}\left(  \int_{\mathbb{R}^{N}}|u|^{2}dx\right)  ^{2},
\]
by applying the concentration-compactness arguments, McCurdy and Venkatraman
proved in \cite{MV21} that there exist universal constants $C_{1}>0$ and
$C_{2}(N)>0$ such that
\[
\delta_{2,\mathbb{R}^{N}}(u)\geq C_{1}\left(  \int_{\mathbb{R}^{N}}%
|u|^{2}dx\right)  d_{1}^{2}(u,G)+C_{2}(N)d_{1}^{4}(u,G),
\]
for all $u\in\{u\in W^{1,2}(\mathbb{R}^{N}):\Vert xu\Vert_{2}<\infty\}$. Here,
$G:=\{ce^{-\alpha|x|^{2}}:c\in\mathbb{R},\alpha>0\}$ is the set of the
Gaussian functions on $\mathbb{R}^{N}$ and $d_{1}(u,G):=\inf\left\{  \left(
\int_{\mathbb{R}^{N}}\left\vert u-ce^{-\alpha|x|^{2}}\right\vert
^{2}dx\right)  ^{1/2}:c\in\mathbb{R},\alpha>0\right\}  $ is the $L^{2}$
distance to the set of the optimizers. Subsequently, in \cite{F21}, Fathi
provided a concise and constructive proof of the estimate, explicitly
identifying the constants $C_{1}$ and $C_{2}(N)$. Following this, Cazacu,
Flynn, Lam and Lu obtained in \cite{CFNL} a sharp version of the estimate with
optimal stability constants and demonstrated that these constants can be
achieved by nontrivial functions. In particular, the following family of
Poincar\'{e} inequality with Gaussian type measures plays the key role in the
approach in \cite{CFNL}:

\begin{theorema}
[\cite{CFNL}, Lemma~3.2]Let $\lambda>0,$ for all smooth function $u,$ there
holds
\begin{equation}
\dfrac{\lambda^{2}}{2}\int_{\mathbb{R}^{N}}\left\vert \nabla u\right\vert
^{2}\;e^{-\frac{|x|^{2}}{\lambda^{2}}}\;dx\geq\inf_{c\in\mathbb{R}}%
\int_{\mathbb{R}^{N}}\left\vert u-c\right\vert ^{2}\;e^{-\frac{|x|^{2}%
}{\lambda^{2}}}\;dx. \label{GTP}%
\end{equation}

\end{theorema}

When $\lambda=1$, (\ref{GTP}) becomes

\begin{theorema}
For all smooth function $u,$ there holds
\begin{equation}
\int_{\mathbb{R}^{N}}\left\vert \nabla u\right\vert ^{2}(2\pi)^{-\frac{N}{2}%
}\,e^{-\frac{|x|^{2}}{2}}\;dx\geq\int_{\mathbb{R}^{N}}\left\vert
u\;-\;(2\pi)^{-\frac{N}{2}}\int_{\mathbb{R}^{N}}u\;e^{-\frac{|x|^{2}}{2}%
}\;dx\right\vert ^{2}(2\pi)^{-\frac{N}{2}}e^{-\frac{|x|^{2}}{2}}\;dx.
\label{CGP}%
\end{equation}

\end{theorema}

We note that (\ref{CGP}) is the classical Poincar\'{e} inequality for Gaussian
measure that has been very well studied in the literature. In particular, we
refer the interested reader to the book \cite{BGL14} where the authors used
the notion of Curvature-Dimension condition to study the Poincar\'{e}
inequality and many other functional, isoperimetric and transportation
inequalities, and to \cite{WB} where Beckner proved a generalized Poincar\'{e}
inequality that interpolates between the classical Poincar\'{e} and the
logarithmic Sobolev inequalities. It is also worth mentioning that in
\cite{Mil09}, E. Milman proved deep results that Cheeger's isoperimetric
inequality, Poincar\'{e}'s inequality and exponential concentration of
Lipschitz functions are all quantitatively equivalent under certain convexity
assumptions. See also \cite{Mil10, Mil091}, for related results.

In the hyperbolic space setting, a version of (\ref{CGP}) has been studied.
For instance, in \cite{BMM22}, the authors established, among others, that
\begin{equation}
\int_{\mathbb{H}^{N}}\left\vert \nabla_{\mathbb{H}}u\right\vert ^{2}%
e^{-\rho^{2}(x)}dV_{\mathbb{H}}\geq K\inf_{c\in\mathbb{R}}\int_{\mathbb{H}%
^{N}}|u-c|^{2}e^{-\rho^{2}(x)}dV_{\mathbb{H}}, \label{CGPH}%
\end{equation}
for some universal constant $K>0$. See also our Corollary
\ref{weightedPoinIneq} in which we can derive (\ref{CGPH}) as a simple
consequence of our main results. We also refer the interested reader to
\cite{MRS11, RS11}, for instance, for related results.

Motivated by the stability results in \cite{CFNL} and in order to answer
Question 1, it is natural to raise the following question about the weighted
Poincar\'{e} inequality with Gaussian type measures (\ref{GTP}) on
$\mathbb{H}^{N}$:

\noindent\textbf{{Question 2:}} \textit{Does there exist a universal constant
}$C$,\textit{ independent of }$\lambda$\textit{ and }$u$,\textit{ such that
for all smooth function }$u$\textit{ and }$\lambda>0$\textit{, there holds }%
\begin{equation}
\dfrac{\lambda^{2}}{2}\int_{\mathbb{H}^{N}}\left\vert \nabla_{\mathbb{H}%
}u\right\vert ^{2}\;e^{-\frac{\rho^{2}(x)}{\lambda^{2}}}\;dV_{\mathbb{H}}\geq
C\inf_{c\in\mathbb{R}}\int_{\mathbb{H}^{N}}\left\vert u-c\right\vert
^{2}\;e^{-\frac{\rho^{2}(x)}{\lambda^{2}}}\;dV_{\mathbb{H}}\text{?}
\label{GTPH}%
\end{equation}

To our knowledge, an inequality like (\ref{GTPH}) has not been established for
hyperbolic space. The main challenging lies in the fact that the scaling
arguments are totally missing in the hyperbolic space. Thus, \eqref{CGPH} does
not lead to \eqref{GTPH} on hyperbolic spaces. Moreover, the fact that
$\mathbb{H}^{N}$ has constant negative sectional curvature $-1$ creates extra
difficulty. Indeed, in the Euclidean case, (\ref{GTP}) can be derived from
(\ref{CGP}) by exploiting simple scaling/dilation arguments, which are not
available in $\mathbb{H}^{N}$, as explained earlier. Therefore, exploring a
family of weighted Poincar\'{e} inequalities in hyperbolic space that depend
on $\lambda$ is of significant interest.

\medskip

The second principal aim of this article is to try to answer the above
Question 2 and establish some versions of the weighted Poincar\'{e} type
inequalities with Gaussian type measures on $\mathbb{H}^{N}$. More precisely,
we would like to investigate the validity of the weighted Poincar\'{e}
inequalities of the form:%
\begin{equation}
\dfrac{\lambda^{2}}{2}\int_{\mathbb{H}^{N}}\left\vert \nabla_{\mathbb{H}%
}u\right\vert ^{2}\;M_{\lambda}(\rho\left(  x\right)  )\;dV_{\mathbb{H}}\geq
K\inf_{c}\int_{\mathbb{H}^{N}}|u-c|^{2}\;N_{\lambda}(\rho\left(  x\right)
)\;dV_{\mathbb{H}}. \label{MLP}%
\end{equation}
Here $\lambda>0$, $M_{\lambda}$ and $N_{\lambda}$ are suitable nonnegative
potentials, $K>0$ is a constant independent of $\lambda$, and $u$ belongs to
an appropriate function space. In other words, we aim to develop a
scale-dependent weighted Poincar\'{e} inequality for hyperbolic space.

In this paper, we will present three admissible pairs $\left(  M_{\lambda
},N_{\lambda}\right)  $ for (\ref{MLP}) and derive three different versions of
the weighted Poincar\'{e} type inequalities. More precisely, in Subsection
\ref{subsct3.1}, by applying the Curvature-Dimension condition on the
Euclidean space, we prove that for all $\lambda>0$, $M_{\lambda}%
=e^{-\frac{\rho^{2}(x)}{\lambda^{2}}}$ and $N_{\lambda}=\left(  1-\tanh
^{2}\left(  \dfrac{\rho(x)}{2}\right)  \right)  ^{N}\cdot e^{-\frac{\rho
^{2}(x)}{\lambda^{2}}}$ satisfy the weighted Poincar\'{e} inequality
(\ref{MLP}). Indeed, we will establish the following result:

\begin{theorem}
\label{1stineq}{(First weighted Poincar\'{e} inequality)} Let $N\geq2.$ For
all $u\in C_{0}^{\infty}(\mathbb{H}^{N})$, there holds
\begin{equation}
\dfrac{\lambda^{2}}{2}\int_{\mathbb{H}^{N}}\left\vert \nabla_{\mathbb{H}%
}u\right\vert ^{2}e^{-\frac{\rho^{2}(x)}{\lambda^{2}}}dV_{\mathbb{H}}\geq
K\inf_{c}\int_{\mathbb{H}^{N}}|u-c|^{2}\left(  1-\tanh^{2}\left(  \dfrac
{\rho(x)}{2}\right)  \right)  ^{N}e^{-\frac{\rho^{2}(x)}{\lambda^{2}}%
}dV_{\mathbb{H}}, \label{fixedL^2Pineq}%
\end{equation}
for all $\lambda>0$ and $K$ is a constant independent of $\lambda$.
\end{theorem}

The major drawback of Theorem \ref{1stineq} is the presence of the weight
function $\left(  1-\tanh^{2}\left(  \frac{\rho(x)}{2}\right)  \right)  ^{N}$.
In Subsection \ref{subsct3.2}, by combining the Curvature-Dimension condition
on Riemannian manifolds and some delicate arguments, we are able to remove
that term and show in the below Theorem~\ref{2ndineq} that when $\lambda$ is
bounded from above, then the pair $M_{\lambda}=N_{\lambda}=e^{-\frac{\rho
^{2}(x)}{\lambda^{2}}}$ are admissible for (\ref{MLP}).

\begin{theorem}
\label{2ndineq}{(Second weighted Poincar\'{e} inequality)} Let $N\geq2.$ For
all $u\in C_{0}^{\infty}(\mathbb{H}^{N})$ and for any positive $\beta>0$ there
holds
\begin{equation}
\dfrac{\lambda^{2}}{2}\int_{\mathbb{H}^{N}}\left\vert \nabla_{\mathbb{H}%
}u\right\vert ^{2}e^{-\frac{\rho^{2}(x)}{\lambda^{2}}}dV_{\mathbb{H}}\geq
K\left(  \beta\right)  \inf_{c}\int_{\mathbb{H}^{N}}|u-c|^{2}e^{-\frac
{\rho^{2}(x)}{\lambda^{2}}}dV_{\mathbb{H}}, \label{2ndcaseineq}%
\end{equation}
for all $\lambda\in(0,\beta]$, and $K\left(  \beta\right)  $ is independent of
$\lambda$.\medskip
\end{theorem}

Finally, in Subsection \ref{subsct3.3}, when $\lambda$ is large enough, we
establish the weighted Poincar\'{e} inequality (\ref{MLP}) with the potentials
$M_{\lambda}=N_{\lambda}=e^{-\left(  \frac{\rho(x)}{\sinh\rho(x)}\right)
^{\frac{(N-1)}{2}}\frac{\rho^{2}(x)}{\lambda^{2}}}$ in Theorem~\ref{3rdineq}.

\begin{theorem}
{(Third weighted Poincar\'{e} inequality)}\label{3rdineq}Let $N\geq2$ and let
$\beta>0$ be large enough. For all $u\in C_{0}^{\infty}(\mathbb{H}^{N})$,
there holds
\begin{equation}
\dfrac{\lambda^{2}}{2}\int_{\mathbb{H}^{N}}|\nabla_{\mathbb{H}}u|^{2}%
e^{-\left(  \frac{\rho(x)}{\sinh\rho(x)}\right)  ^{\frac{(N-1)}{2}}\frac
{\rho^{2}(x)}{\lambda^{2}}}dV_{\mathbb{{H}}}\geq K\left(  \beta\right)
\inf_{c}\int_{\mathbb{H}^{N}}|u-c|^{2}e^{-\left(  \frac{\rho(x)}{\sinh\rho
(x)}\right)  ^{\frac{(N-1)}{2}}\frac{\rho^{2}(x)}{\lambda^{2}}}dV_{\mathbb{{H}%
}} \label{3rdcaseIneq}%
\end{equation}
for all $\lambda\geq\beta$ and $K\left(  \beta\right)  >0$ is independent of
$\lambda$.
\end{theorem}

The third central purpose of this article is to use the weighted Poincar\'{e}
type inequalities established in Section \ref{sct3} and the Heisenberg
Uncertainty Principle identity in Subsection \ref{subsct4.3} to investigate
the Question 1. More precisely, we will study the stability results for the
Heisenberg Uncertainty Principle on the hyperbolic space $\mathbb{H}^{N}$. In
our initial effort to achieve an $L^{2}$-stability result for the Heisenberg
uncertainty principle on the hyperbolic space $\mathbb{H}^{N}$, by taking
advantage of the weighted \textquotedblleft scale-dependent" Poincar\'{e}
inequalities, we establish the following stability result:

\begin{theorem}
\label{Cor 3.1} Let $N\geq2.$ For all $u\in W^{1,2}(\mathbb{H}^{N}%
)\cap\{u:\int_{\mathbb{H}^{N}}\rho^{2}(x)|u|^{2}dV_{\mathbb{H}}<\infty\}$
there holds
\begin{align*}
\delta_{1,\mathbb{H}}(u)  &  \geq K\inf_{c\in\mathbb{R}\text{,~}\lambda>0}%
\int_{\mathbb{H}^{N}}\left\vert u-ce^{-\frac{\rho^{2}(x)}{2\lambda^{2}}%
}\right\vert ^{2}\left(  1-\tanh^{2}\left(  \dfrac{\rho(x)}{2}\right)
\right)  ^{N}dV_{\mathbb{H}}\\
&  \geq K\inf_{c\in\mathbb{R}\text{,~}\lambda>0}\left\Vert u-ce^{-\frac
{\rho^{2}(x)}{2\lambda^{2}}}\right\Vert _{L^{2}(\mathbb{H}^{N},M)}^{2}\\
&  \geq Kd_{1}^{2}(u,E,M)
\end{align*}
where $K$ is a universal positive constant, $M=\left(  1-\tanh^{2}\left(
\dfrac{\rho(x)}{2}\right)  \right)  ^{N}$. Consequently,
\[
\delta_{2,\mathbb{H}}(u)\geq K\left(  \int_{\mathbb{H}^{N}}\left[
N+(N-1)\dfrac{\rho(x)\cosh\rho(x)-\sinh\rho(x)}{\sinh\rho(x)}\right]
|u|^{2}dV_{\mathbb{H}}\right)  d_{1}^{2}(u,E,M)+K^{2}d_{1}^{4}(u,E,M).
\]

\end{theorem}

Here
\[
d_{1}(u,E,M):=\inf\left\{  \left(  \int_{\mathbb{H}^{N}}\left\vert
u-ce^{-\alpha\rho^{2}(x)}\right\vert ^{2}MdV_{\mathbb{H}}\right)  ^{1/2}%
:c\in\mathbb{R},\alpha>0\right\}  ,
\]
and the Heisenberg deficits on $\mathbb{H}^{N}$ are defined as follows:
\begin{align*}
\delta_{1,\mathbb{H}}(u):  &  =\left(  \int_{\mathbb{H}^{N}}|\nabla
_{\mathbb{H}}u|^{2}dV_{\mathbb{H}}\right)  ^{1/2}\left(  \int_{\mathbb{H}^{N}%
}\rho^{2}(x)|u|^{2}dV_{\mathbb{H}}\right)  ^{1/2}\\
&  -\dfrac{1}{2}\int_{\mathbb{H}^{N}}\left(  N+(N-1)\dfrac{\rho(x)\cosh
\rho(x)-\ \sinh\rho(x)}{\sinh\rho(x)}\right)  \;|u|^{2}dV_{\mathbb{H}},
\end{align*}
and%
\begin{align*}
\delta_{2,\mathbb{H}}(u):  &  =\left(  \int_{\mathbb{H}^{N}}|\nabla
_{\mathbb{H}}u|^{2}dV_{\mathbb{H}}\right)  \left(  \int_{\mathbb{H}^{N}}%
\rho^{2}(x)|u|^{2}dV_{\mathbb{H}}\right) \\
&  -\dfrac{1}{4}\left(  \int_{\mathbb{H}^{N}}\left[  N+(N-1)\dfrac
{\rho(x)\cosh\rho(x)-\sinh\rho(x)}{\sinh\rho(x)}\right]  |u|^{2}%
dV_{\mathbb{H}}\right)  ^{2}.
\end{align*}
Though Theorem \ref{Cor 3.1} provides some stability results for the
Heisenberg Uncertainty Principle on the hyperbolic space $\mathbb{H}^{N}$, the
main shortcoming of the above result is the existence of the weight function
$\left(  1-\tanh^{2}\left(  \frac{\rho(x)}{2}\right)  \right)  ^{N}.$ In order
to investigate the stability of the Heisenberg Uncertainty Principle on
$\mathbb{H}^{N}$ without this term, we first introduce a function space as
follows: given any $\beta>0$, we define the set
\[
\mathcal{A}_{\beta}:=\left\{  u\in W^{1,2}(\mathbb{H}^{N})\;\text{such
that}\;\int_{\mathbb{H}^{N}}\rho^{2}(x)|u|^{2}dV_{\mathbb{H}}<\infty
\;\text{and}\;\dfrac{\int_{\mathbb{H}^{N}}\rho^{2}(x)|u)|^{2}dV_{\mathbb{H}}%
}{\int_{\mathbb{H}^{N}}|\nabla_{\mathbb{H}}u|^{2}dV_{\mathbb{H}}}\leq\beta
^{4}\;\right\}  .
\]
Then, we will show that

\begin{theorem}
\label{Cor 3.2} Let $N\geq2.$ Given a positive number $\beta$, for any
$u\in\mathcal{A}_{\beta}$,
\begin{align*}
\delta_{1,\mathbb{H}}(u)  &  \geq K\inf_{\lambda\leq\beta,c\in\mathbb{R}}%
\int_{\mathbb{H}^{N}}\left\vert u-ce^{-\frac{\rho^{2}(x)}{2\lambda^{2}}%
}\right\vert ^{2}dV_{\mathbb{H}}\\
&  \geq K\inf_{\lambda\leq\beta,c\in\mathbb{R}}\left\Vert u-ce^{-\frac
{\rho^{2}(x)}{2\lambda^{2}}}\right\Vert _{L^{2}(\mathbb{H}^{N})}^{2},
\end{align*}
where $K$ is a universal constant depending only on $\beta$. Consequently,
\begin{align*}
\delta_{2,\mathbb{H}}(u)  &  \geq K\left(  \int_{\mathbb{H}^{N}}\left[
N+(N-1)\dfrac{\rho(x)\cosh\rho(x)-\sinh\rho(x)}{\sinh\rho(x)}\right]
|u|^{2}dV_{\mathbb{H}}\right)  \inf_{\lambda\leq\beta,c\in\mathbb{R}%
}\left\Vert u-ce^{-\frac{\rho^{2}(x)}{2\lambda^{2}}}\right\Vert _{L^{2}%
(\mathbb{H}^{N})}^{2}\\
&  +K^{2}\inf_{\lambda\leq\beta,c\in\mathbb{R}}\left\Vert u-ce^{-\frac
{\rho^{2}(x)}{2\lambda^{2}}}\right\Vert _{L^{2}(\mathbb{H}^{N})}^{4}.
\end{align*}

\end{theorem}

Nevertheless, the above stability result relies on a strong assumption
regarding $\lambda$. To address this limitation, we construct a new family of
weight functions $M_{\lambda}$ such that the Poincar\'{e} inequality of type
\eqref{MLP} holds as $\lambda\rightarrow\infty$. Beside, we introduce a new
Heisenberg deficit $\Tilde{\delta}_{1,\mathbb{H}}$ for this case (see
Corollary \ref{Cor 3.3} for details). Combining this result with the earlier
stability findings, we derive the following uniform stability result for the
Heisenberg uncertainty principle in hyperbolic space:

\begin{theorem}
\label{finalstability} For some $\beta>0$ and for any $u$ such that $u\in
W^{1,2}(\mathbb{H}^{N})$ and $\int_{\mathbb{H}^{N}}\rho^{2}(x)|u|^{2}%
dV_{\mathbb{H}}<\infty$, we~have
\begin{align*}
&  \delta_{1,\mathbb{H}}(u)\mathbf{1}_{\mathcal{A}_{\beta}}(u)+\Tilde{\delta
}_{1,\mathbb{H}}(u)\mathbf{1}_{\overline{\mathcal{A}_{\beta}}}(u)\\
&  \geq K\inf_{c\in\mathbb{R}\text{,~}\lambda>0}\int_{\mathbb{H}^{N}%
}\left\vert u-ce^{-\frac{\rho^{2}(x)}{2\lambda^{2}}}\mathbf{1}_{\mathcal{A}%
_{\beta}}(u)-ce^{-\left(  \frac{\rho(x)}{\sinh\rho(x)}\right)  ^{\frac
{(N-1)}{2}}\frac{\rho^{2}(x)}{2\lambda^{2}}}\mathbf{1}_{\overline
{\mathcal{A}_{\beta}}}(u)\right\vert ^{2}dV_{\mathbb{H}}\\
&  =K\inf_{c\in\mathbb{R}\text{,~}\lambda>0}\left\Vert u-\left(
ce^{-\frac{\rho^{2}(x)}{2\lambda^{2}}}\mathbf{1}_{\mathcal{A}_{\beta}%
}(u)+ce^{-\left(  \frac{\rho(x)}{\sinh\rho(x)}\right)  ^{\frac{(N-1)}{2}}%
\frac{\rho^{2}(x)}{2\lambda^{2}}}\mathbf{1}_{\overline{\mathcal{A}_{\beta}}%
}(u)\right)  \right\Vert _{L^{2}(\mathbb{H}^{N})}^{2},
\end{align*}
for all $\lambda\leq\beta$, if $u\in\mathcal{A}_{\beta}$, or for all
$\lambda>\beta$ if $u\in\overline{\mathcal{A}_{\beta}}$, where $\overline
{\mathcal{A}_{\beta}}$ is the complement of the space $\mathcal{A}_{\beta}$.
Here $K=K(\beta)$ is a positive constant independent of $\lambda$ and $u$.
\end{theorem}

\medskip

The last primary intention of our article is to investigate the logarithmic
Sobolev inequality with Gaussian measure. We note that the Poincar\'{e}
inequalities and the logarithmic Sobolev inequalities are among the most
studied functional inequalities for semigroups. Moreover, the logarithmic
Sobolev inequalities encompass much more information than the Poincar\'{e}
inequalities. In the Euclidean setting, the logarithmic Sobolev inequalities
have been well-studied and well-understood. There is a vast literature on the
logarithmic Sobolev inequalities in different settings and their applications
in several fields of mathematics and physics. Here we refer the interested
reader to the seminal papers of L. Gross \cite{Gro75, Gro75b} and the
excellent book of Bakry-Gentil-Ledoux \cite{BGL14} and references therein.

In a very recent paper, the authors in \cite{GJK25} proved the following
inequality: Let $N\geq2$. There exist constants $C_{1}$ and $C_{2}$ such that
for every $\epsilon>0$ we have%
\begin{equation}
\int_{\mathbb{H}^{N}}\left\vert u\right\vert ^{2}\log\left\vert u\right\vert
^{2}dV_{\mathbb{H}}\leq\frac{\epsilon}{\pi}\int_{\mathbb{H}^{N}}\left\vert
\nabla_{\mathbb{H}}u\right\vert ^{2}dV_{\mathbb{H}}+\int_{\mathbb{H}^{N}%
}\left\vert u\right\vert ^{2}dV_{\mathbb{H}}\left[  \log\left(  \int
_{\mathbb{H}^{N}}\left\vert u\right\vert ^{2}dV_{\mathbb{H}}\right)
+C_{1}-C_{2}\log\epsilon\right]  . \label{wLS}%
\end{equation}
(\ref{wLS}) in the setting of the Euclidean space can be found in \cite{LL}.
Moreover, it was shown in \cite{LL} that in the Euclidean setting, (\ref{wLS})
is equivalent to the logarithmic Sobolev inequality with Gaussian measure via
a simple change of variable:%
\[
\int_{\mathbb{R}^{N}}\left\vert u\right\vert ^{2}\log\left(  \frac{\left\vert
u\right\vert ^{2}}{\int_{\mathbb{R}^{N}}\left\vert u\right\vert ^{2}%
e^{-\pi\left\vert x\right\vert ^{2}}dx}\right)  e^{-\pi\left\vert x\right\vert
^{2}}dx\lesssim\int_{\mathbb{R}^{N}}\left\vert \nabla_{\mathbb{H}}u\right\vert
^{2}e^{-\pi\left\vert x\right\vert ^{2}}dx.
\]
Nevertheless, due to the negativity of the sectional curvature on the
hyperbolic space $\mathbb{H}^{N}$, (\ref{wLS}) is significantly weaker than
the following logarithmic Sobolev inequality with Gaussian measure, if it
existed:%
\begin{equation}
\int_{\mathbb{H}^{N}}\left\vert u\right\vert ^{2}\log\left(  \frac{\left\vert
u\right\vert ^{2}}{\int_{\mathbb{H}^{N}}\left\vert u\right\vert ^{2}%
e^{-\pi\rho^{2}(x)}dV_{\mathbb{H}}}\right)  e^{-\pi\rho^{2}(x)}dV_{\mathbb{H}%
}\lesssim\int_{\mathbb{H}^{N}}\left\vert \nabla_{\mathbb{H}}u\right\vert
^{2}e^{-\pi\rho^{2}(x)}dV_{\mathbb{H}}. \label{LS}%
\end{equation}
In particular, (\ref{LS}), if it existed, implies (\ref{wLS}) but the reverse
direction does not hold true.

To the best of our knowledge, in the setting of the hyperbolic space, a
logarithmic Sobolev inequality like (\ref{LS}) is still missing in the
literature. Therefore, our final goal of this paper is to investigate the
logarithmic Sobolev inequalities with Gaussian measures. More precisely, let
$\beta>0$,\ $G=\int_{\mathbb{H}^{N}}e^{-\beta\rho^{2}(x)}dV_{\mathbb{H}}$ be
the normalization constant and $d\mu_{\mathbb{H}}:=\frac{e^{-\beta\rho^{2}%
(x)}}{G}dV_{\mathbb{H}}$ be a probability measure. Then, by using the
$U$-bounds approach (see \cite{HZ10}, for instance), we have the following
logarithmic Sobolev inequality that seems new in the literature:

\begin{theorem}
\label{ThmlogS}There exists a universal constant $LS\left(  \beta\right)  >0$
such that for all $u$ such that $\int_{\mathbb{H}^{N}}\left\vert
\nabla_{\mathbb{H}}u\right\vert ^{2}d\mu_{\mathbb{H}}+\int_{\mathbb{H}^{N}%
}\left\vert u\right\vert ^{2}d\mu_{\mathbb{H}}<\infty:$
\[
\int_{\mathbb{H}^{N}}\left\vert u\right\vert ^{2}\log\left(  \frac{\left\vert
u\right\vert ^{2}}{\int_{\mathbb{H}^{N}}\left\vert u\right\vert ^{2}%
d\mu_{\mathbb{H}}}\right)  d\mu_{\mathbb{H}}\leq LS\left(  \beta\right)
\int_{\mathbb{H}^{N}}\left\vert \nabla_{\mathbb{H}}u\right\vert ^{2}%
d\mu_{\mathbb{H}}.
\]

\end{theorem}

\medskip

We conclude this introduction with an outline of the article. Our paper is
organized as follows: In Section \ref{sct4}, we set up an abstract identity
and use it to establish $L^{2}$-Hardy type and $L^{2}$%
-Caffarelli-Kohn-Nirenberg type identities on $\mathbb{H}^{N}$. This section
also develops some related inequalities and an Heisenberg uncertainty
principle identity, which is a subclass of Caffarelli-Kohn-Nirenberg. In the
next section, we construct weighted Poincar\'{e} inequalities with weight
functions depending on a scaling parameter. Additionally, a version of the
weighted Poincar\'{e} inequality independent of the scaling parameter is given
as a direct corollary. Section \ref{sct5} presents several $L^{2}$-stability
results for the Heisenberg uncertainty principle, both scale non-invariant and
scale invariant, with various Heisenberg deficits. Finally, in Section
\ref{sct6}, we prove several $U$-bound estimates and use them to establish the
logarithmic Sobolev inequality with Gaussian measures.


\section{An abstract identity on the hyperbolic space $\mathbb{H}^{N}$-Proofs
of Theorems \ref{Thm 1.1}, \ref{M1} and \ref{M2}}

\label{sct4}In this section, we would like to construct an abstract functional
identity involving a vector field on the hyperbolic space that we will use
several times later to derive some stability results. More precisely, we will
provide a proof for Theorem \ref{Thm 1.1} and establish the following
identity:
\begin{align}
&  \lambda^{p}\int_{\mathbb{H}^{N}}A|\nabla_{\mathbb{H}}u|^{p}dV_{\mathbb{H}%
}+\dfrac{p-1}{\lambda^{p/(p-1)}}\int_{\mathbb{H}^{N}}A|uX|^{p}dV_{\mathbb{H}%
}\nonumber\\
&  =-\int_{\mathbb{H}^{N}}\operatorname{div}\left(  A|X|^{p-2}X\right)
|u|^{p}dV_{\mathbb{H}}+\int_{\mathbb{H}^{N}}A\mathcal{R}_{p}\left(  \dfrac
{uX}{\lambda^{1/(p-1)}},\lambda\nabla_{\mathbb{H}}u\right)  dV_{\mathbb{H}}.
\label{gnridt}%
\end{align}
Here $\mathcal{R}_{p}\left(  X,Y\right)  :=|Y|^{p}+(p-1)|X|^{p}-p|X|^{p-2}%
\left\langle X,Y\right\rangle $.

Throughout the section, unless we emphasise the use of a specific model of
$\mathbb{H}^{N}$, we will work with the Poincar\'{e} ball model of the
hyperbolic space $\mathbb{H}^{N}$. Also, $\left\langle \cdot,\cdot
\right\rangle _{\mathbb{R}^{N}}$ denotes the standard inner product on
$\mathbb{R}^{N}$.

\begin{proof}
[Proof of Theorem \ref{Thm 1.1}]The proof follows from a few direct
computations. Indeed, with $X$ be a vector field on $\mathbb{H}^{N}$, $A\in
C^{1}(\mathbb{H}^{N}\setminus\left\{  0\right\}  )$ and $u\in C_{0}%
^{1}(\mathbb{H}^{N}\setminus\left\{  0\right\}  )$, we get, from the
divergence theorem
\[
-\int_{\mathbb{H}^{N}}\operatorname{div}(A|X|^{p-2}X)|u|^{p}dV_{\mathbb{H}%
}=\int_{\mathbb{H}^{N}}\left\langle \nabla_{\mathbb{H}}|u|^{p},A|X|^{p-2}%
X\right\rangle dV_{\mathbb{H}}.
\]
Since
\[
\nabla_{\mathbb{H}}|u|^{p}=\dfrac{(1-|x|^{2})^{2}}{4}p|u|^{p-2}u\left(
\sum_{i=1}^{N}\dfrac{\partial u}{\partial x_{i}}\partial_{x_{i}}\right)  .
\]
Letting $X=\sum_{i=1}^{N}X_{i}\partial_{x_{i}}$, we have
\begin{align*}
\left\langle \nabla_{\mathbb{H}}|u|^{p},A|X|^{p-2}X\right\rangle  &
=\sum_{i=1}^{N}p|u|^{p-2}u\dfrac{\partial u}{\partial x_{i}}\cdot
A|X|^{p-2}X_{i}\\
&  =p|u|^{p-2}uA|X|^{p-2}\sum_{i=1}^{N}\dfrac{\partial u}{\partial x_{i}}\cdot
X_{i}\\
&  =p|u|^{p-2}uA|X|^{p-2}\left\langle \nabla_{\mathbb{H}}u,X\right\rangle .
\end{align*}
Hence,
\begin{align*}
-\int_{\mathbb{H}^{N}}\operatorname{div}(A|X|^{p-2}X)|u|^{p}dV_{\mathbb{H}}
&  =\int_{\mathbb{H}^{N}}\left\langle \nabla_{\mathbb{H}}|u|^{p}%
,A|X|^{p-2}X\right\rangle dV_{\mathbb{H}}\\
&  =p\int_{\mathbb{H}^{N}}\dfrac{1}{\lambda^{(p-2)/(p-1)}}A|uX|^{p-2}%
\left\langle \dfrac{uX}{\lambda^{1/(p-1)}},\lambda\nabla_{\mathbb{H}%
}u\right\rangle dV_{\mathbb{H}}\\
&  =\lambda^{p}\int_{\mathbb{H}^{N}}A|\nabla_{\mathbb{H}}u|^{p}dV_{\mathbb{H}%
}+\dfrac{p-1}{\lambda^{p/(p-1)}}\int_{\mathbb{H}^{N}}A|uX|^{p}dV_{\mathbb{H}%
}\\
&  -\int_{\mathbb{H}^{N}}A\mathcal{R}_{p}\left(  \dfrac{uX}{\lambda^{1/(p-1)}%
},\lambda\nabla_{\mathbb{H}}u\right)  dV_{\mathbb{H}},
\end{align*}
where $\mathcal{R}_{p}\left(  X,Y\right)  :=|Y|^{p}+(p-1)|X|^{p}%
-p|X|^{p-2}\left\langle X,Y\right\rangle $, as desired.
\end{proof}

\begin{remark}
From \cite{CKLL24, DLL22}, for instance, it is straightforward to see that
$\mathcal{R}_{p}\left(  X,Y\right)  \geq0$, and the equality holds iff $X=Y$.
Indeed, we have
\begin{align*}
\mathcal{R}_{p}\left(  X,Y\right)   &  =|Y|^{p}+(p-1)|X|^{p}-p|X|^{p-2}%
\left\langle X,Y\right\rangle \\
&  =\left\langle Y,Y\right\rangle ^{p/2}+(p-1)\left\langle X,X\right\rangle
^{p/2}-p\left\langle X,X\right\rangle ^{(p-2)/2}\left\langle X,Y\right\rangle
\\
&  =\left(  \dfrac{2}{1-|x|^{2}}\right)  ^{p}\cdot\left(  \left\langle
Y,Y\right\rangle _{\mathbb{R}^{N}}^{p/2}+(p-1)\left\langle X,X\right\rangle
_{\mathbb{R}^{N}}^{p/2}-p\left\langle X,X\right\rangle _{\mathbb{R}^{N}%
}^{(p-2)/2}\left\langle X,Y\right\rangle _{\mathbb{R}^{N}}\right) \\
&  =\left(  \dfrac{2}{1-|x|^{2}}\right)  ^{p}\cdot\left(  |Y|_{\mathbb{R}^{N}%
}^{p}+(p-1)|X|_{\mathbb{R}^{N}}^{p}-p|X|_{\mathbb{R}^{N}}^{p-2}\left\langle
X,Y\right\rangle _{\mathbb{R}^{N}}\right)  \geq0.
\end{align*}

\end{remark}

\subsection{The first application}

\label{subsct4.1} As a direct consequence, the above theorem is a
generalization of \cite[Theorem $3.2$]{FLLM23}:

\begin{theorem}
\label{Cor 1.1}{(\cite[Theorem $3.2$]{FLLM23})} Let $0<R\leq\infty$, $V\geq0$
and $W$ are $C^{1}$-functions on $(0,R)$. Assume that $(r^{N-1}V,r^{N-1}W)$ be
a Bessel pair on $(0,R)$, i.e., there exists $\varphi>0$ on $(0,R)$
satisfying
\[
(r^{N-1}V\varphi^{\prime})^{\prime}+r^{N-1}W\varphi=0.
\]
Then, we have the following identity for $u\in C_{0}^{\infty}(B_{R}%
\setminus\left\{  0\right\}  ):$
\begin{align*}
&  \int_{B_{R}}V(\rho(x))|\nabla_{\mathbb{H}}u|^{2}dV_{\mathbb{H}}-\int
_{B_{R}}W(\rho(x))|u|^{2}dV_{\mathbb{H}}\\
&  =\int_{B_{R}}V(\rho(x))|\varphi^{2}(\rho(x))|\left\vert \nabla_{\mathbb{H}%
}\left(  \dfrac{u}{\varphi(\rho(x))}\right)  \right\vert ^{2}dV_{\mathbb{H}}\\
&  -(N-1)\int_{B_{R}}V(\rho(x))\dfrac{\varphi^{\prime}(\rho(x))}{\varphi
(\rho(x))}\dfrac{\rho(x)\cosh{\rho(x)}-\sinh{\rho(x)}}{\rho(x)\sinh{\rho(x)}%
}|u|^{2}dV_{\mathbb{H}},
\end{align*}
where $B_{R}:=\{x\in\mathbb{H}^{N}:\rho(x)<R\}$.
\end{theorem}

\begin{proof}
By utilizing $\lambda=1$, $p=2$, $A=V(\rho(x))$, and $X=\dfrac{\varphi
^{\prime}(\rho(x))}{\varphi(\rho(x))}\nabla_{\mathbb{H}}(\rho(x))$, we have
\begin{align}
\operatorname{div}(AX)  &  =\operatorname{div}\left(  V(\rho(x))\dfrac
{\varphi^{\prime}(\rho(x))}{\varphi(\rho(x))}\nabla_{\mathbb{H}}%
(\rho(x))\right) \nonumber\label{eq 1.7}\\
&  =\nabla_{\mathbb{H}}\left(  V(\rho(x))\dfrac{\varphi^{\prime}(\rho
(x))}{\varphi(\rho(x))}\right)  \cdot\nabla_{\mathbb{H}}(\rho(x))+V(\rho
(x))\dfrac{\varphi^{\prime}(\rho(x))}{\varphi(\rho(x))}\text{div}\left(
\nabla_{\mathbb{H}}(\rho(x))\right) \nonumber\\
&  =V^{\prime}(\rho(x))\dfrac{\varphi^{\prime}(\rho(x))}{\varphi(\rho
(x))}+V(\rho(x))\dfrac{\varphi^{\prime\prime}(\rho(x))\varphi(\rho(x))-\left(
\varphi^{\prime}(\rho(x))\right)  ^{2}}{\varphi^{2}(\rho(x))}\nonumber\\
&  +(N-1)V(\rho(x))\dfrac{\varphi^{\prime}(\rho(x))}{\varphi(\rho(x))}%
\dfrac{\cosh(\rho(x))}{\sinh(\rho(x))}\nonumber\\
&  =\dfrac{1}{\varphi(\rho(x))}\left(  V^{\prime}(\rho(x))\varphi^{\prime
}(\rho(x))+V(\rho(x))\varphi^{\prime\prime}(\rho(x))+(N-1)V(\rho
(x))\dfrac{\varphi^{\prime}(\rho(x))}{\rho(x)}\right) \nonumber\\
&  -V(\rho(x))\dfrac{\left(  \varphi^{\prime}(\rho(x))\right)  ^{2}}%
{\varphi^{2}(\rho(x))}\nonumber\\
&  +(N-1)V(\rho(x))\dfrac{\varphi^{\prime}(\rho(x))}{\varphi(\rho(x))}\left(
\dfrac{\cosh(\rho(x))}{\sinh(\rho(x))}-\dfrac{1}{\rho(x)}\right) \nonumber\\
&  =-W(\rho(x))-V(\rho(x))\dfrac{\left(  \varphi^{\prime}(\rho(x))\right)
^{2}}{\varphi^{2}(\rho(x))}+(N-1)V(\rho(x))\dfrac{\varphi^{\prime}(\rho
(x))}{\varphi(\rho(x))}\left(  \dfrac{\cosh(\rho(x))}{\sinh(\rho(x))}%
-\dfrac{1}{\rho(x)}\right)  .
\end{align}
Therefore, from Theorem~\ref{Thm 1.1}, we have
\begin{align*}
&  \int_{B_{R}}V(\rho(x))|\nabla_{\mathbb{H}}u|^{2}dV_{\mathbb{H}}+\int
_{B_{R}}V(\rho(x))\dfrac{\left(  \varphi^{\prime}(\rho(x))\right)  ^{2}%
}{\varphi^{2}(\rho(x))}\left\vert u\right\vert ^{2}dV_{\mathbb{H}}\\
&  =\int_{B_{R}}W(\rho(x))|u|^{2}dV_{\mathbb{H}}+\int_{B_{R}}V(\rho
(x))\dfrac{\left(  \varphi^{\prime}(\rho(x))\right)  ^{2}}{\varphi^{2}%
(\rho(x))}\left\vert u\right\vert ^{2}dV_{\mathbb{H}}\\
&  -(N-1)\int_{B_{R}}V(\rho(x))\dfrac{\varphi^{\prime}(\rho(x))}{\varphi
(\rho(x))}\left(  \dfrac{\cosh(\rho(x))}{\sinh(\rho(x))}-\dfrac{1}{\rho
(x)}\right)  |u|^{2}dV_{\mathbb{H}}\\
&  +\int_{B_{R}}V(\rho(x))\left\vert u\dfrac{\varphi^{\prime}(\rho
(x))}{\varphi(\rho(x))}\nabla_{\mathbb{H}}(\rho(x))-\nabla_{\mathbb{H}%
}u\right\vert ^{2}dV_{\mathbb{H}}.
\end{align*}
Hence,
\begin{align*}
&  \int_{B_{R}}V(\rho(X))|\nabla_{\mathbb{H}}u|^{2}dV_{\mathbb{H}}-\int
_{B_{R}}W(\rho(x))|u|^{2}dV_{\mathbb{H}}\\
&  =\int_{B_{R}}V(\rho(x))|\varphi^{2}(\rho(x))|\left\vert \nabla_{\mathbb{H}%
}\left(  \dfrac{u}{\varphi(\rho(x))}\right)  \right\vert ^{2}dV_{\mathbb{H}}\\
&  -(N-1)\int_{B_{R}}V(\rho(x))\dfrac{\varphi^{\prime}(\rho(x))}{\varphi
(\rho(x))}\dfrac{\rho(x)\cosh{\rho(x)}-\sinh{\rho(x)}}{\rho(x)\sinh{\rho(x)}%
}|u|^{2}dV_{\mathbb{H}},
\end{align*}
which is exactly Theorem \ref{Cor 1.1}.
\end{proof}

By choosing $\left(  r^{N-1},\left(  \dfrac{N-2}{2}\right)  ^{2}%
r^{N-3}\right)  $ as a Bessel pair on $(0,\infty)$ with its corresponding
solution $\varphi=r^{-\frac{N-2}{2}}$, we are able to get the Hardy inequality
on the hyperbolic space with its exact remainders, i.e.

\begin{corollary}
Let $N \geq3$. For $u \in C^{\infty}_{0}(\mathbb{H}^{N})$, we have
\begin{align*}
&  \int_{\mathbb{H}^{N}}|\nabla_{\mathbb{H}}u|^{2}dV_{\mathbb{H}}-\left(
\dfrac{N-2}{2}\right)  ^{2}\int_{\mathbb{H}^{N}}\dfrac{|u|^{2}}{\rho^{2}%
(x)}dV_{\mathbb{H}}\\
&  =\int_{\mathbb{H}^{N}}\dfrac{1}{\rho^{N-2}(x)}\left|  \nabla_{\mathbb{H}%
}\left(  u\rho^{(N-2)/2}(x)\right)  \right|  ^{2}dV_{\mathbb{H}}\\
&  +\dfrac{(N-1)(N-2)}{2}\int_{\mathbb{H}^{N}}\dfrac{\rho(x)\cosh{\rho
(x)}-\sinh{\rho(x)}}{\rho^{2}(x)\sinh{\rho(x)}}|u|^{2}dV_{\mathbb{H}}.
\end{align*}
Consequently, it follows the Hardy inequality
\[
\int_{\mathbb{H}^{N}}|\nabla_{\mathbb{H}}u|^{2}dV_{\mathbb{H}} \geq\left(
\dfrac{N-2}{2}\right)  ^{2}\int_{\mathbb{H}^{N}}\dfrac{|u|^{2}}{\rho^{2}%
(x)}dV_{\mathbb{H}}.
\]

\end{corollary}

Furthermore, instead of utilizing $\lambda=1$ in \eqref{gnridt}, we can
optimize the parameter $\lambda$ to get the Caffarelli-Kohn-Nirenberg identity
with Bessel pairs that has been stated in Theorem \ref{M1}.

\begin{proof}
[Proof of Theorem \ref{M1}]By using the same arguments as in the proof of
Theorem \ref{Cor 1.1}, with $p=2$, $A=V(\rho(x))$, and $X=\dfrac
{\varphi^{\prime}(\rho(x))}{\varphi(\rho(x))}\nabla_{\mathbb{H}}(\rho(x))$, we
get
\begin{align*}
&  \lambda^{2}\int_{B_{R}}V(\rho(x))|\nabla_{\mathbb{H}}u|^{2}dV_{\mathbb{H}%
}+\dfrac{1}{\lambda^{2}}\int_{B_{R}}V(\rho(x))\dfrac{\left(  \varphi^{\prime
}(\rho(x))\right)  ^{2}}{\varphi^{2}(\rho(x))}\left\vert u\right\vert
^{2}dV_{\mathbb{H}}\\
&  =\int_{B_{R}}W(\rho(x))|u|^{2}dV_{\mathbb{H}}+\int_{B_{R}}V(\rho
(x))\dfrac{\left(  \varphi^{\prime}(\rho(x))\right)  ^{2}}{\varphi^{2}%
(\rho(x))}\left\vert u\right\vert ^{2}dV_{\mathbb{H}}\\
&  -(N-1)\int_{B_{R}}V(\rho(x))\dfrac{\varphi^{\prime}(\rho(x))}{\varphi
(\rho(x))}\left(  \dfrac{\cosh(\rho(x))}{\sinh(\rho(x))}-\dfrac{1}{\rho
(x)}\right)  |u|^{2}dV_{\mathbb{H}}\\
&  +\int_{B_{R}}V(\rho(x))\left\vert \dfrac{1}{\lambda}u\dfrac{\varphi
^{\prime}(\rho(x))}{\varphi(\rho(x))}\nabla_{\mathbb{H}}(\rho(x))-\lambda
\nabla_{\mathbb{H}}u\right\vert ^{2}dV_{\mathbb{H}}.
\end{align*}
By choosing $\lambda=\left(  \dfrac{\int_{B_{R}}V(\rho(x))\frac{\left(
\varphi^{\prime}(\rho(x))\right)  ^{2}}{\varphi^{2}(\rho(x))}\left\vert
u\right\vert ^{2}dV_{\mathbb{H}}}{\int_{B_{R}}V(\rho(x))|\nabla_{\mathbb{H}%
}u|^{2}dV_{\mathbb{H}}}\right)  ^{1/4}$, it implies
\begin{align*}
&  \left(  \int_{B_{R}}V(\rho(x))|\nabla_{\mathbb{H}}u|^{2}dV_{\mathbb{H}%
}\right)  ^{1/2}\left(  \int_{B_{R}}V(\rho(x))\dfrac{\left(  \varphi^{\prime
}(\rho(x))\right)  ^{2}}{\varphi^{2}(\rho(x))}\left\vert u\right\vert
^{2}dV_{\mathbb{H}}\right)  ^{1/2}\\
&  =\dfrac{1}{2}\int_{B_{R}}\left(  W(\rho(x))+V(\rho(x))\dfrac{\left(
\varphi^{\prime}(\rho(x))\right)  ^{2}}{\varphi^{2}(\rho(x))}\right)
\left\vert u\right\vert ^{2}dV_{\mathbb{H}}\\
&  -\dfrac{N-1}{2}\int_{B_{R}}V(\rho(x))\dfrac{\varphi^{\prime}(\rho
(x))}{\varphi(\rho(x))}\left(  \dfrac{\rho(x)\cosh\rho(x)-\sinh\rho(x)}%
{\rho(x)\sinh\rho(x)}\right)  |u|^{2}dV_{\mathbb{H}}\\
&  +\dfrac{1}{2}\int_{B_{R}}V(\rho(x))\left\vert \dfrac{1}{\lambda}%
u\dfrac{\varphi^{\prime}(\rho(x))}{\varphi(\rho(x))}\nabla_{\mathbb{H}}%
(\rho(x))-\lambda\nabla_{\mathbb{H}}u\right\vert ^{2}dV_{\mathbb{H}}.
\end{align*}
Therefore, since $\frac{\varphi^{\prime}}{\varphi}$ is negative and $V$ is
nonnegative, we have the following family of Caffarelli-Kohn-Nirenberg
inequalities with Bessel pairs as a direct consequence of the above identity
\begin{align*}
&  \left(  \int_{B_{R}}V(\rho(x))|\nabla_{\mathbb{H}}u|^{2}dV_{\mathbb{H}%
}\right)  ^{1/2}\left(  \int_{B_{R}}V(\rho(x))\dfrac{\left(  \varphi^{\prime
}(\rho(x))\right)  ^{2}}{\varphi^{2}(\rho(x))}\left\vert u\right\vert
^{2}dV_{\mathbb{H}}\right)  ^{1/2}\\
&  \geq\dfrac{1}{2}\int_{B_{R}}\left(  W(\rho(x))+V(\rho(x))\dfrac{\left(
\varphi^{\prime}(\rho(x))\right)  ^{2}}{\varphi^{2}(\rho(x))}\right)
\left\vert u\right\vert ^{2}dV_{\mathbb{H}},
\end{align*}
as desired.
\end{proof}

\subsection{The second application: The Caffarelli-Kohn-Nirenberg identities
on the hyperbolic space}

\label{subsct4.2} As another application of Theorem \ref{Thm 1.1}, we can get
some Caffarelli-Kohn-Nirenberg identities and inequalities on the hyperbolic
space. In particular, we prove

\begin{theorem}
\label{thmckn} Let $N\geq2.$

\begin{itemize}
\item For $b+1-a>0$ and $b \leq\dfrac{N-2}{2}$ and $u \in C_{0}^{\infty
}(\mathbb{H}^{N}\setminus\{0\})$ there holds%

\begin{equation}
\left(  \int_{\mathbb{H}^{N}}\dfrac{|\nabla_{\mathbb{H}}u|^{2}}{\rho^{2b}%
(x)}dV_{\mathbb{H}}\right)  \left(  \int_{\mathbb{H}^{N}}\dfrac{|u|^{2}}%
{\rho^{2a}(x)}dV_{\mathbb{H}}\right)  \geq\dfrac{(N-1-a-b)^{2}}{4}\left(
\int_{\mathbb{H}^{N}}\dfrac{|u|^{2}}{\rho^{a+b+1}(x)}dV_{\mathbb{H}}\right)
^{2}. \label{CKN1}%
\end{equation}

\medskip

\item For $b+1-a<0$ and $b \geq\dfrac{N-2}{2}$ and $u \in C_{0}^{\infty
}(\mathbb{H}^{N} \setminus\{0\})$ there holds%

\begin{align}
&  \left(  \int_{\mathbb{H}^{N}}\dfrac{|\nabla_{\mathbb{H}}u|^{2}}{\rho
^{2b}(x)}dV_{\mathbb{H}}\right)  ^{1/2}\left(  \int_{\mathbb{H}^{N}}%
\dfrac{|u|^{2}}{\rho^{2a}(x)}dV_{\mathbb{H}}\right)  ^{1/2}%
\nonumber\label{L2CKNidtc1}\\
&  -\dfrac{1}{2}\int_{\mathbb{H}^{N}}\left(  a+b-N+1-(N-1)\left(  \rho
(x)\coth{\rho(x)}-1\right)  \right)  \dfrac{|u|^{2}}{\rho^{a+b+1}%
(x)}dV_{\mathbb{H}}\nonumber\\
&  =\dfrac{\lambda^{b-a+1}}{2}\int_{\mathbb{H}^{N}}\dfrac{1}{\rho^{2b}%
(x)}\left\vert \nabla_{\mathbb{H}}\left(  u.e^{-\frac{\rho^{b-a+1}%
(x)}{(b-a+1)\lambda^{b-a+1}}}\right)  \right\vert ^{2}e^{\frac{2\rho
^{b-a+1}(x)}{(b-a+1)\lambda^{b-a+1}}}dV_{\mathbb{H}}.
\end{align}

\medskip

\item For $b+1-a<0$ and $b \geq\dfrac{N-2}{2}$ and $u \in C_{0}^{\infty
}(\mathbb{H}^{N} \setminus\{0\})$ there holds%

\begin{align}
&  \left(  \int_{\mathbb{H}^{N}}\dfrac{\left\vert \nabla_{\mathbb{H}%
}u\right\vert ^{2}}{\rho^{2b}(x)}dV_{\mathbb{H}}\right)  ^{1/2}\left(
\int_{\mathbb{H}^{N}}\dfrac{\left\vert u\right\vert ^{2}}{\rho^{2a}%
(x)}dV_{\mathbb{H}}\right)  ^{1/2}\nonumber\label{L2CKNidtc3}\\
&  -\dfrac{1}{2}\int_{\mathbb{H}^{N}}\left[  N-3b+a-3+(N-1)(1-\rho
(x)\coth(\rho(x)))\right]  \dfrac{\left\vert u\right\vert ^{2}}{\rho
^{a+b+1}(x)}dV_{\mathbb{H}}\nonumber\\
&  =\dfrac{\lambda^{b-a+1}}{2}(N-2b-2)(N-1)\int_{\mathbb{H}^{N}}(\rho
(x)\coth(\rho(x))-1)\dfrac{\left\vert u\right\vert ^{2}}{\rho^{2b+2}%
(x)}dV_{\mathbb{H}}\nonumber\\
&  +\dfrac{\lambda^{b-a+1}}{2}\int_{\mathbb{H}^{N}}\rho^{4+2b-2N}(x)\left\vert
\nabla_{\mathbb{H}}\left(  u\rho^{N-2b-2}(x)e^{-\frac{\rho^{b-a+1}%
(x)}{(b-a+1)\lambda^{b-a+1}}}\right)  \right\vert ^{2}e^{\frac{2\rho
^{b-a+1}(x)}{(b-a+1)\lambda^{b-a+1}}}dV_{\mathbb{H}}.
\end{align}

\medskip

\item For $b+1-a>0$ and $b \geq\dfrac{N-2}{2}$ and $u \in C_{0}^{\infty
}(\mathbb{H}^{N} \setminus\{0\})$ there holds%

\begin{align}
&  \left(  \int_{\mathbb{H}^{N}}\dfrac{\left\vert \nabla_{\mathbb{H}%
}u\right\vert ^{2}}{\rho^{2b}(x)}dV_{\mathbb{H}}\right)  ^{1/2}\left(
\int_{\mathbb{H}^{N}}\dfrac{\left\vert u\right\vert ^{2}}{\rho^{2a}%
(x)}dV_{\mathbb{H}}\right)  ^{1/2}\nonumber\label{L2CKNidtc4}\\
&  -\dfrac{1}{2}\int_{\mathbb{H}^{N}}\left[  3b-a-N+3-(N-1)(1-\rho
(x)\coth(\rho(x)))\right]  \dfrac{\left\vert u\right\vert ^{2}}{\rho
^{a+b+1}(x)}dV_{\mathbb{H}}\nonumber\\
&  =\dfrac{\lambda^{b-a+1}}{2}(N-2b-2)(N-1)\int_{\mathbb{H}^{N}}(\rho
(x)\coth(\rho(x))-1)\dfrac{\left\vert u\right\vert ^{2}}{\rho^{2b+2}%
(x)}dV_{\mathbb{H}}\nonumber\\
&  +\dfrac{\lambda^{b-a+1}}{2}\int_{\mathbb{H}^{N}}\rho^{4+2b-2N}(x)\left\vert
\nabla_{\mathbb{H}}\left(  u\rho^{N-2b-2}(x)e^{\frac{\rho^{b-a+1}%
(x)}{(b-a+1)\lambda^{b-a+1}}}\right)  \right\vert ^{2}e^{-\frac{2\rho
^{b-a+1}(x)}{(b-a+1)\lambda^{b-a+1}}}dV_{\mathbb{H}}.
\end{align}

\end{itemize}
\end{theorem}

\begin{proof}
[Proof of Theorem \ref{thmckn}]We begin with the proof of \eqref{CKN1}. By
choosing $\lambda=1$, $p=2$, $A=\frac{1}{\rho^{2b}(x)}$, and $X=-\dfrac
{1}{\lambda^{b-a+1}}\rho^{b-a}(x)\nabla_{\mathbb{H}}(\rho(x))$ in
Theorem~\ref{Thm 1.1} we have
\begin{align*}
-\lambda^{b-a+1}\operatorname{div}\left(  AX\right)   &  =\operatorname{div}%
\left(  \frac{1}{\rho^{2b}(x)}\rho^{b-a}(x)\nabla_{\mathbb{H}}(\rho(x))\right)
\\
&  =\operatorname{div}\left(  \rho^{-b-a}(x)\nabla_{\mathbb{H}}(\rho
(x))\right) \\
&  =\dfrac{1}{\sqrt{\det(g)}}\sum_{i=1}^{N}\dfrac{\partial}{\partial x_{i}%
}\left(  \rho^{-b-a}(x)\dfrac{(1-|x|^{2})x_{i}}{2|x|}\sqrt{\det(g)}\right) \\
&  =\dfrac{(1-|x|^{2})^{N}}{2^{N}}\sum_{i=1}^{N}\dfrac{\partial}{\partial
x_{i}}\left(  \rho^{-b-a}(x)\dfrac{(1-|x|^{2})x_{i}}{2|x|}\dfrac{2^{N}%
}{(1-|x|^{2})^{N}}\right) \\
&  =(-b-a)\rho^{b-a-1}(x)+(N-1)\rho^{-b-a}(x)|x|+(N-1)\dfrac{1-|x|^{2}}%
{2|x|}\rho^{-b-a}(x)\\
&  =(-b-a)\rho^{-b-a-1}(x)+(N-1)\rho^{-b-a}(x)\coth{\rho(x)}\\
&  =\rho^{-b-a-1}(x)\left(  -b-a+(N-1)\rho(x)\coth{\rho(x)}\right)  .
\end{align*}
From \eqref{gnridt}, we get%

\begin{align*}
\int_{\mathbb{H}^{N}}\dfrac{|\nabla_{\mathbb{H}}u|^{2}}{\rho^{2b}%
(x)}dV_{\mathbb{H}}  &  +\dfrac{1}{\lambda^{2b-2a+2}}\int_{\mathbb{H}^{N}%
}\dfrac{|u|^{2}}{\rho^{2a}(x)}dV_{\mathbb{H}}\\
&  -\dfrac{1}{\lambda^{b-a+1}}\int_{\mathbb{H}^{N}}\left(
N-1-a-b+(N-1)\left(  \rho(x)\coth{\rho(x)}-1\right)  \right)  \dfrac{|u|^{2}%
}{\rho^{a+b+1}(x)}dV_{\mathbb{H}}\\
&  =\int_{\mathbb{H}^{N}}\dfrac{1}{\rho^{2b}(x)}\left\vert \dfrac{u\rho
^{b-a}(x)\nabla_{\mathbb{H}}\rho(x)}{\lambda^{b-a+1}}+\nabla_{\mathbb{H}%
}u\right\vert ^{2}dV_{\mathbb{H}}\\
&  =\int_{\mathbb{H}^{N}}\dfrac{1}{\rho^{2b}(x)}\left\vert \nabla_{\mathbb{H}%
}\left(  u.e^{\frac{\rho^{b-a+1}(x)}{(b-a+1)\lambda^{b-a+1}}}\right)
\right\vert ^{2}e^{\frac{-2\rho^{b-a+1}(x)}{(b-a+1)\lambda^{b-a+1}}%
}dV_{\mathbb{H}},
\end{align*}
which implies that
\begin{align*}
\lambda^{b-a+1}\int_{\mathbb{H}^{N}}\dfrac{|\nabla_{\mathbb{H}}u|^{2}}%
{\rho^{2b}(x)}dV_{\mathbb{H}}  &  +\dfrac{1}{\lambda^{b-a+1}}\int
_{\mathbb{H}^{N}}\dfrac{|u|^{2}}{\rho^{2a}(x)}dV_{\mathbb{H}}\\
&  -\int_{\mathbb{H}^{N}}\left(  N-1-a-b+(N-1)\left(  \rho(x)\coth{\rho
(x)}-1\right)  \right)  \dfrac{|u|^{2}}{\rho^{a+b+1}(x)}dV_{\mathbb{H}}\\
&  =\lambda^{b-a+1}\int_{\mathbb{H}^{N}}\dfrac{1}{\rho^{2b}(x)}\left\vert
\nabla_{\mathbb{H}}\left(  u.e^{\frac{\rho^{b-a+1}(x)}{(b-a+1)\lambda^{b-a+1}%
}}\right)  \right\vert ^{2}e^{\frac{-2\rho^{b-a+1}(x)}{(b-a+1)\lambda^{b-a+1}%
}}dV_{\mathbb{H}}%
\end{align*}
By choosing $\lambda=\left(  \dfrac{\int_{\mathbb{H}^{N}}|u|^{2}/\rho
^{2a}(x)dV_{\mathbb{H}}}{\int_{\mathbb{H}^{N}}|\nabla_{\mathbb{H}}u|^{2}%
/\rho^{2b}(x)dV_{\mathbb{H}}}\right)  ^{1/(2b+2-2a)}$, we have the following
identity:
\begin{align*}
&  \left(  \int_{\mathbb{H}^{N}}\dfrac{|\nabla_{\mathbb{H}}u|^{2}}{\rho
^{2b}(x)}dV_{\mathbb{H}}\right)  ^{1/2}\left(  \int_{\mathbb{H}^{N}}%
\dfrac{|u|^{2}}{\rho^{2a}(x)}dV_{\mathbb{H}}\right)  ^{1/2}\\
&  -\dfrac{1}{2}\int_{\mathbb{H}^{N}}\left(  N-1-a-b+(N-1)\left(  \rho
(x)\coth{\rho(x)}-1\right)  \right)  \dfrac{|u|^{2}}{\rho^{a+b+1}%
(x)}dV_{\mathbb{H}}\\
&  =\dfrac{\lambda^{b-a+1}}{2}\int_{\mathbb{H}^{N}}\dfrac{1}{\rho^{2b}%
(x)}\left\vert \nabla_{\mathbb{H}}\left(  u.e^{\frac{\rho^{b-a+1}%
(x)}{(b-a+1)\lambda^{b-a+1}}}\right)  \right\vert ^{2}e^{\frac{-2\rho
^{b-a+1}(x)}{(b-a+1)\lambda^{b-a+1}}}dV_{\mathbb{H}}.
\end{align*}
Therefore, we have
\[
\left(  \int_{\mathbb{H}^{N}}\dfrac{|\nabla_{\mathbb{H}}u|^{2}}{\rho^{2b}%
(x)}dV_{\mathbb{H}}\right)  \left(  \int_{\mathbb{H}^{N}}\dfrac{|u|^{2}}%
{\rho^{2a}(x)}dV_{\mathbb{H}}\right)  \geq\dfrac{(N-1-a-b)^{2}}{4}\left(
\int_{\mathbb{H}^{N}}\dfrac{|u|^{2}}{\rho^{a+b+1}(x)}dV_{\mathbb{H}}\right)
^{2}.
\]

\medskip

To prove \eqref{L2CKNidtc1}, choose $\lambda=1,p=2$, $A=\frac{1}{\rho^{2b}%
(x)}$, and $X=\dfrac{1}{\lambda^{b-a+1}}\rho^{b-a}(x)\nabla_{\mathbb{H}}%
(\rho(x))$ in Theorem~\ref{Thm 1.1} and by a similar computation as above we
conclude the result.

\medskip

Next, we will present a proof of \eqref{L2CKNidtc3}. By choosing
$\lambda=1,p=2$, $A=\frac{1}{\rho^{2b}(x)}$, and $X=\dfrac{1}{\lambda^{b-a+1}%
}\rho^{b-a}(x)\nabla_{\mathbb{H}}(\rho(x))-(N-2b-2)\dfrac{\nabla_{\mathbb{H}%
}(\rho(x))}{\rho(x)}$ in Theorem~\ref{Thm 1.1}, and further computing each
terms separately
\begin{align*}
\operatorname{div}\left(  AX\right)   &  =\dfrac{1}{\lambda^{b+1-a}%
}\operatorname{div}\left(  \rho^{-b-a}(x)\nabla_{\mathbb{H}}(\rho(x))\right)
-(N-2b-2)\operatorname{div}\left(  \rho^{-2b-1}(x)\nabla_{\mathbb{H}}%
(\rho(x))\right) \\
&  =\dfrac{1}{\lambda^{b+1-a}}\left[  \nabla_{\mathbb{H}}(\rho^{-b-a}%
(x))\cdot\nabla_{\mathbb{H}}(\rho(x))+\rho^{-b-a}(x)\operatorname{div}%
(\nabla_{\mathbb{H}}(\rho(x)))\right] \\
&  -(N-2b-2)\left[  \nabla_{\mathbb{H}}(\rho^{-2b-1}(x))\cdot\nabla
_{\mathbb{H}}(\rho(x))+\rho^{-2b-1}(x)\operatorname{div}(\nabla_{\mathbb{H}%
}(\rho(x)))\right] \\
&  =\dfrac{1}{\lambda^{b+1-a}}\left[  (-b-a)\rho^{-b-a-1}(x)+(N-1)\rho
^{-b-a}(x)\coth{(\rho(x))}\right] \\
&  -(N-2b-2)\left[  (-2b-1)\rho^{-2b-2}(x)+(N-1)\rho^{-2b-1}(x)\coth
{(\rho(x))}\right] \\
&  =\dfrac{1}{\lambda^{b+1-a}}\rho^{-b-a-1}(x)\left[  (-b-a)+(N-1)\rho
(x)\coth{(\rho(x))}\right] \\
&  -(N-2b-2)\rho^{-2b-2}(x)\left[  (-2b-1)+(N-1)\rho(x)\coth{(\rho
(x))}\right]  .
\end{align*}
Moreover,
\[
|X|^{2}=\dfrac{\rho^{2b-2a}(x)}{\lambda^{2b+2-2a}}+(N-2b-2)^{2}\dfrac{1}%
{\rho^{2}(x)}-\dfrac{2(N-2b-2)}{\lambda^{b+1-a}}\rho^{b-a-1}(x),
\]
and
\[
\left\vert uX-\nabla_{\mathbb{H}}u\right\vert ^{2}=\left\vert u\dfrac
{\varphi^{\prime}(\rho(x))}{\varphi(\rho(x))}\nabla_{\mathbb{H^{N}}}%
(\rho(x))-\nabla_{\mathbb{H}}u\right\vert ^{2}=\varphi^{2}(\rho(x))\left\vert
\nabla\left(  \dfrac{u}{\varphi(\rho(x))}\right)  \right\vert ^{2},
\]
where $\varphi(r)=r^{2b+2-N}e^{\frac{r^{b+1-a}}{(b+1-a)\lambda^{b+1-a}}}$.
Therefore, from \eqref{gnridt}, we obtain
\begin{align*}
&  \int_{\mathbb{H}^{N}}\dfrac{\left\vert \nabla_{\mathbb{H}}u\right\vert
^{2}}{\rho^{2b}(x)}dV_{\mathbb{H}}+\dfrac{1}{\lambda^{2b-2a+2}}\int
_{\mathbb{H}^{N}}\dfrac{\left\vert u\right\vert ^{2}}{\rho^{2a}(x)}%
dV_{\mathbb{H}}+(N-2b-2)^{2}\int_{\mathbb{H}^{N}}\dfrac{\left\vert
u\right\vert ^{2}}{\rho^{2b+2}(x)}dV_{\mathbb{H}}\\
&  =\dfrac{1}{\lambda^{b-a+1}}\int_{\mathbb{H}^{N}}\left[
N-3b+a-3+(N-1)(1-\rho(x)\coth(\rho(x)))\right]  \dfrac{\left\vert u\right\vert
^{2}}{\rho^{a+b+1}(x)}dV_{\mathbb{H}}\\
&  +(N-2b-2)\int_{\mathbb{H}^{N}}\left[  N-2b-2+(N-1)(\rho(x)\coth
(\rho(x))-1)\right]  \dfrac{\left\vert u\right\vert ^{2}}{\rho^{2b+2}%
(x)}dV_{\mathbb{H}}\\
&  +\int_{\mathbb{H}^{N}}\rho^{4+2b-2N}(x)\left\vert \nabla_{\mathbb{H}%
}\left(  u\rho^{N-2b-2}(x)e^{-\frac{\rho^{b-a+1}(x)}{(b-a+1)\lambda^{b-a+1}}%
}\right)  \right\vert ^{2}e^{\frac{2\rho^{b-a+1}(x)}{(b-a+1)\lambda^{b-a+1}}%
}dV_{\mathbb{H}}.
\end{align*}
It implies that
\begin{align*}
&  \lambda^{b-a+1}\int_{\mathbb{H}^{N}}\dfrac{\left\vert \nabla_{\mathbb{H}%
}u\right\vert ^{2}}{\rho^{2b}(x)}dV_{\mathbb{H}}+\dfrac{1}{\lambda^{b-a+1}%
}\int_{\mathbb{H}^{N}}\dfrac{\left\vert u\right\vert ^{2}}{\rho^{2a}%
(x)}dV_{\mathbb{H}}\\
&  -\int_{\mathbb{H}^{N}}\left[  N-3b+a-3+(N-1)(1-\rho(x)\coth(\rho
(x)))\right]  \dfrac{\left\vert u\right\vert ^{2}}{\rho^{a+b+1}(x)}%
dV_{\mathbb{H}}\\
&  =\lambda^{b-a+1}(N-2b-2)(N-1)\int_{\mathbb{H}^{N}}(\rho(x)\coth
(\rho(x))-1)\dfrac{\left\vert u\right\vert ^{2}}{\rho^{2b+2}(x)}%
dV_{\mathbb{H}}\\
&  +\lambda^{b-a+1}\int_{\mathbb{H}^{N}}\rho^{4+2b-2N}(x)\left\vert
\nabla_{\mathbb{H}}\left(  u\rho^{N-2b-2}(x)e^{-\frac{\rho^{b-a+1}%
(x)}{(b-a+1)\lambda^{b-a+1}}}\right)  \right\vert ^{2}e^{\frac{2\rho
^{b-a+1}(x)}{(b-a+1)\lambda^{b-a+1}}}dV_{\mathbb{H}}.
\end{align*}
By choosing $\lambda=\left(  \dfrac{\int_{\mathbb{H}^{N}}|u|^{2}/\rho
^{2a}(x)dV_{\mathbb{H}}}{\int_{\mathbb{H}^{N}}|\nabla_{\mathbb{H}}u|^{2}%
/\rho^{2b}(x)dV_{\mathbb{H}}}\right)  ^{1/(2b+2-2a)},$ we prove the required identity.

\medskip

Finally, we shall prove \eqref{L2CKNidtc4}. By utilizing $\lambda=1,p=2$,
$A=\frac{1}{\rho^{2b}(x)}$, and
\[
X=-\dfrac{1}{\lambda^{b-a+1}}\rho^{b-a}(x)\nabla_{\mathbb{H}}(\rho
(x))-(N-2b-2)\dfrac{\nabla_{\mathbb{H}}(\rho(x))}{\rho(x)}%
\]
in Theorem~\ref{Thm 1.1} we prove the required identity.
\end{proof}

\medskip

\subsection{The third application: Heisenberg Uncertainty Principle Identity}

\label{subsct4.3} In this subsection, we will take advantage of Theorem
\ref{Thm 1.1} again to get an identity that would allow us to acquire some
$L^{2}$-stability results of the Heisenberg uncertainty principle on the
hyperbolic space.

\begin{proof}
[Proof of Theorem \ref{M2}]By applying Theorem \ref{Thm 1.1} with $p=2,A=1$
and $X=-\rho(x)\nabla_{\mathbb{H}}(\rho(x))$, \eqref{gnridt} turns into
\begin{align*}
&  \lambda^{2}\int_{\mathbb{H}^{N}}|\nabla_{\mathbb{H}}u|^{2}dV_{\mathbb{H}%
}+\dfrac{1}{\lambda^{2}}\int_{\mathbb{H}^{N}}\rho^{2}(x)|u|^{2}|\nabla
_{\mathbb{H}}(\rho(x))|^{2}dV_{\mathbb{H}}\\
&  =\int_{\mathbb{H}^{N}}\text{div}(\rho(x)\nabla_{\mathbb{H}}(\rho
(x)))|u|^{2}dV_{\mathbb{H}}+\int_{\mathbb{H}^{N}}\mathcal{R}_{2}\left(
\dfrac{-u\rho(x)\nabla_{\mathbb{H}}(\rho(x))}{\lambda},\lambda\ \nabla
_{\mathbb{H}}u\right)  dV_{\mathbb{H}}.
\end{align*}
This is a specific instance of the first case from the previous section, where
we select $a=-1$ and $b=0$. Now, we will present a detailed computation for
this case. We have
\[
d\rho(x)=\sum_{i=1}^{N}\partial_{i}\left(  \ln\dfrac{1+|x|}{1-|x|}\right)
dx_{i}.
\]
Then
\[
\nabla_{\mathbb{H}}(\rho(x))=\sum_{i=1}^{N}\dfrac{(1-|x|^{2})x_{i}}%
{2|x|}\dfrac{\partial}{\partial{x_{i}}},
\]
which implies that
\begin{equation}
|\nabla_{\mathbb{H}}(\rho(x))|^{2}=\dfrac{4}{(1-|x|^{2})^{2}}\sum_{i=1}%
^{N}\dfrac{(1-|x|^{2})^{2}}{4}\dfrac{x_{i}^{2}}{|x|^{2}}=1. \label{eq L2.1}%
\end{equation}
A straightforward computation gives
\begin{equation}
\operatorname{div}\left(  \rho(x)\nabla_{\mathbb{H}}(\rho(x))\right)
\;=\;N+(N-1)\dfrac{\rho(x)\cosh\rho(x)-\sinh\rho(x)}{\sinh\rho(x)}.
\label{eq L2.2}%
\end{equation}
Moreover,
\begin{align}
\mathcal{R}_{2}\left(  \dfrac{-u\rho(x)\nabla_{\mathbb{H}}(\rho(x))}{\lambda
},\lambda\nabla_{\mathbb{H}}u\right)   &  =\left\vert \dfrac{u\rho
(x)\nabla_{\mathbb{H}}(\rho(x))}{\lambda}+\lambda\nabla_{\mathbb{H}%
}u\right\vert ^{2}\nonumber\label{eq L2.3}\\
&  =\lambda^{2}e^{-\frac{\rho^{2}(x)}{\lambda^{2}}}\left\vert \dfrac
{u\rho(x)\nabla_{\mathbb{H}}(\rho(x))}{\lambda^{2}}e^{\frac{\rho^{2}%
(x)}{2\lambda^{2}}}+\nabla_{\mathbb{H}}ue^{\frac{\rho^{2}(x)}{2\lambda^{2}}%
}\right\vert ^{2}\nonumber\\
&  =\lambda^{2}e^{-\frac{\rho^{2}(x)}{\lambda^{2}}}\left\vert \nabla
_{\mathbb{H}}\left(  ue^{\frac{\rho^{2}(x)}{2\lambda^{2}}}\right)  \right\vert
^{2}.
\end{align}
Therefore, from \eqref{eq L2.1}, \eqref{eq L2.2} and \eqref{eq L2.3}, we are
able to get
\begin{align}
&  \lambda^{2}\int_{\mathbb{H}^{N}}|\nabla_{\mathbb{H}}u|^{2}dV_{\mathbb{H}%
}+\lambda^{-2}\int_{\mathbb{H}^{N}}\rho^{2}(x)|u|^{2}dV_{\mathbb{H}%
}\nonumber\label{L2idt}\\
&  -\int_{\mathbb{H}^{N}}\left(  N+(N-1)\dfrac{\rho(x)\cosh\rho(x)-\sinh
\rho(x)}{\sinh\rho(x)}\right)  \;|u|^{2}dV_{\mathbb{H}}\nonumber\\
&  =\lambda^{2}\int_{\mathbb{H}^{N}}e^{-\frac{\rho^{2}(x)}{\lambda^{2}}%
}\left\vert \nabla_{\mathbb{H}}\left(  ue^{\frac{\rho^{2}(x)}{2\lambda^{2}}%
}\right)  \right\vert ^{2}dV_{\mathbb{H}}.
\end{align}
By choosing $\lambda=\left(  \dfrac{\int_{\mathbb{H}^{N}}\rho^{2}%
(x)|u|^{2}dV_{\mathbb{H}}}{\int_{\mathbb{H}^{N}}|\nabla_{\mathbb{H}}%
u|^{2}dV_{\mathbb{H}}}\right)  ^{1/4},$ we acquire
\begin{equation}
\delta_{1,\mathbb{H}}(u)=\dfrac{\lambda^{2}}{2}\int_{\mathbb{H}^{N}}%
e^{-\frac{\rho^{2}(x)}{\lambda^{2}}}\left\vert \nabla_{\mathbb{H}}\left(
ue^{\frac{\rho^{2}(x)}{2\lambda^{2}}}\right)  \right\vert ^{2}dV_{\mathbb{H}},
\label{DeficitIdt}%
\end{equation}
where the Heisenberg deficit is defined as follows
\begin{align}
\delta_{1,\mathbb{H}}(u):  &  =\left(  \int_{\mathbb{H}^{N}}|\nabla
_{\mathbb{H}}u|^{2}dV_{\mathbb{H}}\right)  ^{1/2}\left(  \int_{\mathbb{H}^{N}%
}\rho^{2}(x)|u|^{2}dV_{\mathbb{H}}\right)  ^{1/2}\nonumber\label{eq 1.2}\\
&  -\dfrac{1}{2}\int_{\mathbb{H}^{N}}\left(  N+(N-1)\dfrac{\rho(x)\cosh
\rho(x)-\sinh\rho(x)}{\sinh\rho(x)}\right)  \;|u|^{2}dV_{\mathbb{H}}.
\end{align}

\end{proof}

Following the ideas in \cite{AJLL23}, starting from \eqref{DeficitIdt}, to get
a result about the $L^{2}$-stability for the Heisenberg Uncertainty Principle,
it would be ideal to construct a Poincar\'{e} type inequality on the
hyperbolic spaces having the following form
\begin{equation}
\dfrac{\lambda^{2}}{2}\int_{\mathbb{H}^{N}}\left\vert \nabla_{\mathbb{H}%
}u\right\vert ^{2}e^{-\frac{\rho^{2}(x)}{\lambda^{2}}}dV_{\mathbb{H}}\geq
K\inf_{c}\int_{\mathbb{H}^{N}}|u-c|^{2}e^{-\frac{\rho^{2}(x)}{\lambda^{2}}%
}dV_{\mathbb{H}}, \label{L^2Pineq1}%
\end{equation}
for all $\lambda>0$ and $K$ is a positive constant independent of $\lambda$.
Unfortunately, it is challenging to get exactly such a Poincar\'{e} type
inequality. This difficulty may arise from two key factors. Firstly, the
factor relating to the scaling parameter should have the form of
$O(\lambda^{2})$. Secondly, the constant must be independent of $\lambda$ for
all $\lambda>0$. Therefore, with some modifications in the inequality, we can
get some corresponding $L^{2}$-stability results. In other words, in the next
section, we are considering to construct some weighted Poicar\'{e}
inequalities of the form
\begin{equation}
K\int_{\mathbb{H}^{N}}\left\vert \nabla_{\mathbb{H}}u\right\vert
^{2}M_{\lambda}dV_{\mathbb{H}}\geq\inf_{c}\int_{\mathbb{H}^{N}}|u-c|^{2}%
N_{\lambda}dV_{\mathbb{H}}, \label{L^2Pgnrine}%
\end{equation}
where $K=O(\lambda^{2})$, and $M_{\lambda}$, $N_{\lambda}$ are potentials
defined on $\mathbb{H}^{N}$.


\section{Scale-Dependent Poincar\'{e} inequalities with Gaussian type measures
on the hyperbolic space-Proofs of Theorems \ref{1stineq}, \ref{2ndineq} and
\ref{3rdineq}}

\label{sct3}

\subsection{The first weighted Poincar\'{e} inequality}

\label{subsct3.1} For this case, we are working on \eqref{L^2Pgnrine} with
$M_{\lambda}=e^{-\frac{\rho^{2}(x)}{\lambda^{2}}}$ and $N_{\lambda}=\left(
1-\tanh^{2}\left(  \dfrac{\rho(x)}{2}\right)  \right)  ^{N}\cdot
e^{-\frac{\rho^{2}(x)}{\lambda^{2}}}$. The following lemma will be useful in
the proof of Theorem \ref{1stineq}:

\begin{lemma}
\label{lemma1stineq} Let $N\geq2$ and $B(0,1)$ be the unit ball centered at
the origin on $\mathbb{R}^{N}.$ For all $u\in C_{0}^{\infty}(B(0,1))$ there
holds
\[
\int_{B(0,1)}\left\vert \nabla u\right\vert ^{2}e^{-\frac{\rho^{2}(x)}%
{\lambda^{2}}}dx\geq\dfrac{8}{\lambda^{2}}\inf_{c}\int_{B(0,1)}|u-c|^{2}%
e^{-\frac{\rho^{2}(x)}{\lambda^{2}}}dx,
\]
for all $\lambda>0$. Moreover, the constant $8$ is sharp.
\end{lemma}

\begin{proof}
From \cite[Corollary 4.8.2]{BGL14}, setting $W=\frac{\rho^{2}(x)}{\lambda^{2}%
}$, it remains to verify that there exists a constant $C>0$ such that
\[
\text{Hess}(W)\geq\frac{C}{\lambda^{2}}\text{Id}.
\]
By some simple computations, we obtain, for all $1\leq i\leq N$,
\begin{align*}
\dfrac{\partial}{\partial x_{i}}\left(  \ln^{2}\left(  \dfrac{1+|x|}%
{1-|x|}\right)  \right)   &  =2\ln\dfrac{1+|x|}{1-|x|}.\dfrac{\partial
}{\partial x_{i}}\left(  \ln\dfrac{1+|x|}{1-|x|}\right) \\
&  =4\ln\dfrac{1+|x|}{1-|x|}\dfrac{x_{i}}{|x|(1-|x|^{2})},
\end{align*}
which implies that
\[
\dfrac{\partial^{2}}{\partial x_{i}\partial x_{j}}\left(  \ln^{2}\left(
\dfrac{1+|x|}{1-|x|}\right)  \right)  =\dfrac{8x_{i}x_{j}}{|x|^{2}%
(1-|x|^{2})^{2}}+4\ln\dfrac{1+|x|}{1-|x|}\dfrac{3x_{i}x_{j}|x|^{2}-x_{i}x_{j}%
}{|x|^{3}(1-|x|^{2})^{2}}:=H_{ij},
\]
if $1\leq i\neq j\leq N$, and
\[
\dfrac{\partial^{2}}{\partial x_{i}^{2}}\left(  \ln^{2}\left(  \dfrac
{1+|x|}{1-|x|}\right)  \right)  =\dfrac{8x_{i}^{2}}{|x|^{2}(1-|x|^{2})^{2}%
}+4\ln\dfrac{1+|x|}{1-|x|}\dfrac{3x_{i}^{2}|x|^{2}-x_{i}^{2}+|x|^{2}-|x|^{4}%
}{|x|^{3}(1-|x|^{2})^{2}}:=H_{ii},
\]
for all $1\leq i\leq N$. Let $H=[H_{ij}]$. Therefore, we would like to find
$C>0$ such that
\[
\text{Hess}(W)=\dfrac{1}{\lambda^{2}}H\geq\dfrac{C}{\lambda^{2}}\text{Id},
\]
which is equivalent to
\[
H-C\text{Id}\geq0.
\]
Next, we will calculate the determinant of the matrix $H-C\text{Id}$. Indeed,
we can represent entries of $(H-C\text{Id})$ in the form of
\[
\left[  \left(  4\ln\left(  \dfrac{1+|x|}{1-|x|}\right)  \dfrac{1}%
{|x|-|x|^{3}}-C\right)  \delta_{ij}+\left(  \dfrac{8}{|x|^{2}(1-|x|^{2})^{2}%
}+4\ln\dfrac{1+|x|}{1-|x|}\dfrac{3|x|^{2}-1}{|x|^{3}(1-|x|^{2})^{2}}\right)
x_{i}x_{j}\right]  .
\]
For convenience, we denote
\[
P:=4\ln\left(  \dfrac{1+|x|}{1-|x|}\right)  \dfrac{1}{|x|-|x|^{3}}-C,
\]
and
\[
Q:=\dfrac{8}{|x|^{2}(1-|x|^{2})^{2}}+4\ln\dfrac{1+|x|}{1-|x|}\dfrac
{3|x|^{2}-1}{|x|^{3}(1-|x|^{2})^{2}}.
\]
Then, we can write the $i^{th}$-row of the matrix $H-C\text{Id}$ as a vector
in $\mathbb{R}^{N}$, i.e
\[
\lbrack H-C\text{Id}]_{i}=(0,0,\ldots,\underset{i^{th}}{\underbrace{P}}%
,\ldots,0)+Qx_{i}(x_{1},x_{2},\ldots,x_{N}).
\]
As a result, we have
\[
\text{det}(H-C\text{Id})=P^{N}+P^{N-1}Q\sum_{i=1}^{N}x_{i}^{2}=P^{N}%
+P^{N-1}Q|x|^{2}=P^{N-1}(P+|x|^{2}Q).
\]
Firstly, we would like to determine the best $C$ such that $H-C\text{Id}>0$.
Using Sylvester's criterion for positive definite matrices that states that
the determinants of all leading principal minors are positive, we have
$P^{k}+P^{k-1}Q\sum_{i=1}^{k}x_{i}^{2}>0$ for all $k=\overline{1,N}$, which
implies that $P^{k-1}(P+|x|^{2}Q)>0$ for all $k=\overline{1,N}$. Taking $k=1$
and $k=2$, we have $P$ and $P+|x|^{2}Q$ are positive or
\begin{align*}
C  &  <\min_{x\in B(0,1)\setminus\{0\}}\left\{  4\ln\left(  \dfrac
{1+|x|}{1-|x|}\right)  \dfrac{1}{|x|-|x|^{3}},\dfrac{8}{(1-|x|^{2})^{2}}%
+8\ln\dfrac{1+|x|}{1-|x|}\dfrac{|x|}{(1-|x|^{2})^{2}}\right\} \\
&  <\min_{x\in B(0,1)\setminus\{0\}}\left\{  4\ln\left(  \dfrac{1+|x|}%
{1-|x|}\right)  \dfrac{1}{|x|-|x|^{3}}\right\}
\end{align*}
Thus, $C\leq8$. We claim that $C=8$ is the best constant. Indeed, for any
$C>8$, if $N$ is odd, we can choose $x\in B(0,1)\setminus\{0\}$, such that $P$
and $P+|x|^{2}Q$ are negative. Then $P^{N-1}(P+|x|^{2}Q)<0$, a contradiction.
Otherwise, if $N$ is even, we take $x\in B(0,1)\setminus\{0\}$ satisfying
$P<0<P+|x|^{2}Q$, which follows us $P^{N-1}(P+|x|^{2}Q)<0$, also a
contradiction. Therefore, $C=8$ is the best constant so that $P^{k-1}%
(P+|x|^{2}Q)>0$ for all $k=\overline{1,N}$. Obviously, when $C=8$,
$P^{k}+P^{k-1}Q\sum_{i=1}^{k}x_{i}^{2}>0$ for all $k=\overline{1,N}$, since
$Q$ is positive.\newline Next, we will determine $C$ such that $H-C\text{Id}%
\geq0$. From above, $C=8$ satisfies this requirement. We claim that any $C>8$
will not satisfy the condition. Indeed, using the same arguments as above, for
any $C>8$, we can choose $x\in B(0,1)\setminus\{0\}$ such that $\text{det}%
(H-C\text{Id})=P^{N-1}(P+|x|^{2}Q)<0$, which contradicts to the version of
Sylvester's criterion for positive semidefinite matrices. Therefore, the best
constant is $C=8$.
\end{proof}

Now, we are ready to give the proof of Theorem \ref{1stineq}.

\begin{proof}
[Proof of Theorem \ref{1stineq}]Rewriting \eqref{fixedL^2Pineq} on the
Euclidean space, it is equivalent to prove
\begin{equation}
\dfrac{\lambda^{2}}{2}\int_{B(0,1)}\left\vert \nabla u\right\vert ^{2}\left(
\dfrac{2}{1-|x|^{2}}\right)  ^{N-2}e^{-\frac{\rho^{2}(x)}{\lambda^{2}}}dx\geq
K\inf_{c}\int_{B(0,1)}|u-c|^{2}e^{-\frac{\rho^{2}(x)}{\lambda^{2}}}dx,
\label{L2PineqRn}%
\end{equation}
where $\nabla$ denotes the Euclidean gradient, and $K$ is independent of
$\lambda$, since
\[
dV_{\mathbb{H}}=\dfrac{2^{N}}{(1-|x|^{2})^{N}}dx\;\;\text{and}\;\;\left\vert
\nabla_{\mathbb{H}}(u)\right\vert ^{2}=\left(  \dfrac{1-|x|^{2}}{2}\right)
^{2}\left\vert \nabla u\right\vert ^{2}.
\]
By a rough estimate, we get
\[
LHS_{\eqref{L2PineqRn}}\geq\dfrac{\lambda^{2}}{2^{3-N}}\int_{B(0,1)}\left\vert
\nabla u\right\vert ^{2}e^{-\frac{\rho^{2}(x)}{\lambda^{2}}}dx.
\]
As a result, \eqref{L2PineqRn} follows directly from Lemma \ref{lemma1stineq}.
Hence, the proof is complete.
\end{proof}

\subsection{The second weighted Poincar\'{e} inequality}

\label{subsct3.2}Return to the estimate \eqref{L^2Pineq1}, it seems quite
challenging to require the existence of a universal constant $K$ for all
$\lambda>0$. In this case, we will restrict ourselves to work with $\lambda$
bounded above by a finite number. In other words, we will prove the inequality
\eqref{L^2Pgnrine} with $M_{\lambda}=N_{\lambda}=e^{-\frac{\rho^{2}%
(x)}{\lambda^{2}}}$, and $K=O(\lambda^{2})$ for all $\lambda$ bounded above by
a given constant.

Before giving the proof of Theorem \ref{2ndineq}, we recall a theorem in
\cite{BGL14} stating a Poincar\'{e} inequality for general Riemannian
manifolds, i.e

\begin{lemma}
\label{lem 2.2} Let $(M,g)$ be a Riemannian manifold with the measure $\mu$
having density $e^{-W}$ such that
\[
\text{Ric}_{g}+\text{Hess}W\geq\rho g,
\]
with $\rho>0$, then we have
\[
\int_{M}|\nabla_{g}u|^{2}e^{-W}d\mu\geq\rho\inf_{c}\int|u-c|^{2}e^{-W}d\mu.
\]

\end{lemma}

Now, we are in the position to give the proof of Theorem \ref{2ndineq}.

\begin{proof}
[Proof of Theorem \ref{2ndineq}]To get our desired estimate, we will split the
interval $(0,\beta]$ into two subintervals:\newline\textbf{Case 1:}
$\mathbf{0<\lambda<\alpha:=\sqrt{\dfrac{2-\epsilon_{0}}{N-1}}\;\text{for
some}\;0<\epsilon_{0}<2.}$ Let us consider the hyperbolic space $\mathbb{H}%
^{N}$ with the Poincar\'{e} ball model. Then, the Hessian matrix of
$\mathbb{H}^{N}$ is given by
\[
(D^{2}u)_{ij}=\dfrac{\partial^{2}u}{\partial x_{i}\partial x_{j}}-\sum
_{k=1}^{N}\Gamma_{ij}^{k}(x)\dfrac{\partial u}{\partial x_{k}},
\]
where
\[
\Gamma_{ij}^{k}(x)=\dfrac{2(\delta_{jk}x_{i}+\delta_{ki}x_{j}-\delta_{ij}%
x_{k})}{1-|x|^{2}}.
\]
Then, for $1\leq i=j\leq N$
\[
(D^{2}u)_{ii}=\dfrac{\partial^{2}u}{\partial x_{i}^{2}}-\dfrac{4}{1-|x|^{2}%
}x_{i}\dfrac{\partial u}{\partial x_{i}}+\dfrac{2}{1-|x|^{2}}\sum_{k=1}%
^{N}x_{k}\dfrac{\partial u}{\partial x_{k}},
\]
and, for $1\leq i\neq j\leq N$,
\[
(D^{2}u)_{ij}=\dfrac{\partial^{2}u}{\partial x_{i}\partial x_{j}}-\dfrac
{2}{1-|x|^{2}}x_{i}\dfrac{\partial u}{\partial x_{j}}-\dfrac{2}{1-|x|^{2}%
}x_{j}\dfrac{\partial u}{\partial x_{i}}.
\]
By taking $u=\ln\left(  \dfrac{1+|x|}{1-|x|}\right)  ^{2}$, and making use of
some computations in the previous section, it is straightforward to get
\begin{align*}
(D^{2}u)_{ii}  &  =\dfrac{8x_{i}^{2}}{|x|^{2}(1-|x|^{2})^{2}}-4\ln\left(
\dfrac{1+|x|}{1-|x|}\right)  \dfrac{x_{i}^{2}}{|x|(1-|x|^{2})^{2}}-4\ln\left(
\dfrac{1+|x|}{1-|x|}\right)  \dfrac{x_{i}^{2}}{|x|^{3}(1-|x|^{2})^{2}}\\
&  +4\ln\left(  \dfrac{1+|x|}{1-|x|}\right)  \dfrac{1}{|x|(1-|x|^{2})}%
+8\ln\left(  \dfrac{1+|x|}{1-|x|}\right)  \dfrac{|x|}{(1-|x|^{2})^{2}},
\end{align*}
and
\[
(D^{2}u)_{ij}=\dfrac{8x_{i}x_{j}}{|x|^{2}(1-|x|^{2})^{2}}-4\ln\left(
\dfrac{1+|x|}{1-|x|}\right)  \dfrac{x_{i}x_{j}}{|x|(1-|x|^{2})^{2}}%
-4\ln\left(  \dfrac{1+|x|}{1-|x|}\right)  \dfrac{x_{i}x_{j}}{|x|^{3}%
(1-|x|^{2})^{2}}.
\]
Then, in order to utilize Lemma \ref{lem 2.2}, we need to find $\rho=\frac
{A}{\lambda^{2}}>0$ such that
\[
-(N-1)g_{ij}+\dfrac{D^{2}u}{\lambda^{2}}\geq\rho g_{ij},
\]
Firstly, we want to find the best constant $C$ such that $D^{2}u-Cg_{ij}\geq
0$. By using analogous arguments as in the proof of Lemma \ref{lemma1stineq},
we will calculate the determinant of the matrix $D^{2}u-Cg_{ij}$. Indeed, we
can represent entries of $(D^{2}u-Cg_{ij})$ in the form of
\[
\left[  P\delta_{ij}+Qx_{i}x_{j}\right]  ,
\]
where
\[
P:=4\ln\left(  \dfrac{1+|x|}{1-|x|}\right)  \dfrac{1}{|x|(1-|x|^{2})}%
+8\ln\left(  \dfrac{1+|x|}{1-|x|}\right)  \dfrac{|x|}{(1-|x|^{2})^{2}}%
-C\dfrac{4}{(1-|x|^{2})^{2}},
\]
and
\[
Q:=\dfrac{8}{|x|^{2}(1-|x|^{2})^{2}}-4\ln\dfrac{1+|x|}{1-|x|}\dfrac
{1}{|x|(1-|x|^{2})^{2}}-4\ln\dfrac{1+|x|}{1-|x|}\dfrac{1}{|x|^{3}%
(1-|x|^{2})^{2}}.
\]
Then, we can write the $i^{th}$-row of the matrix $D^{2}u-Cg_{ij}$ as a vector
in $\mathbb{R}^{N}$, i.e
\[
\lbrack D^{2}u-Cg_{ij}]_{i}=(0,0,\ldots,\underset{i^{th}}{\underbrace{P}%
},\ldots,0)+Qx_{i}(x_{1},x_{2},\ldots,x_{N}).
\]
As a result, we have
\[
\text{det}(D^{2}u-Cg_{ij})=P^{N}+P^{N-1}Q\sum_{i=1}^{N}x_{i}^{2}=P^{N}%
+P^{N-1}Q|x|^{2}=P^{N-1}(P+|x|^{2}Q).
\]
Similarly, we have $C=2$ as the best constant so that $D^{2}u-Cg_{ij}\geq0$.
Therefore,
\[
\dfrac{D^{2}u}{\lambda^{2}}\geq\dfrac{2}{\lambda^{2}}g_{ij},
\]
which follows us
\[
-(N-1)g_{ij}+\dfrac{D^{2}u}{\lambda^{2}}\geq\left(  \dfrac{2}{\lambda^{2}%
}-N+1\right)  g_{ij}\geq\dfrac{A}{\lambda^{2}}g_{ij},
\]
where $A>0$. Then, by choosing $0<A=\epsilon_{0}<2$, we have
\begin{equation}
\dfrac{\lambda^{2}}{2}\int_{\mathbb{H}^{N}}\left\vert \nabla_{\mathbb{H}%
}u\right\vert ^{2}e^{-\frac{\rho^{2}(x)}{\lambda^{2}}}dV_{\mathbb{H}}%
\geq\dfrac{\epsilon_{0}}{2}\inf_{c}\int_{\mathbb{H}^{N}}|u-c|^{2}%
e^{-\frac{\rho^{2}(x)}{\lambda^{2}}}dV_{\mathbb{H}},
\end{equation}
for all $0<\lambda\leq\sqrt{\dfrac{2-\epsilon_{0}}{N-1}}=\alpha$.

\textbf{Case 2:} $\mathbf{\alpha\leq\lambda\leq\beta}$. Here, we assumed that
$\beta> \alpha$. Otherwise, we have nothing to prove. To deal with the case,
we use some results from \cite{BMM22}. Indeed, let $M=e^{-\frac{r^{2}}%
{\lambda^{2}}}$, where $r=\rho(x)$, and $v=-\log(M)=\frac{r^{2}}{\lambda^{2}}%
$. Then, $\dfrac{\partial v}{\partial r}=\dfrac{2}{\lambda^{2}}r$ and
$\dfrac{\partial^{2} v}{\partial r^{2}}=\dfrac{2}{\lambda^{2}}$, which allows
us
\[
\Delta_{\mathbb{H}}v=\dfrac{\partial^{2} v}{\partial r^{2}}+(N-1)\coth r
\dfrac{\partial v}{\partial r}=\dfrac{2}{\lambda^{2}}+\dfrac{2(N-1)}%
{\lambda^{2}}r\coth r,
\]
and
\[
a|\nabla_{\mathbb{H}}v|^{2}-\Delta_{\mathbb{H}}v=\dfrac{4a}{\lambda^{4}}%
r^{2}-\dfrac{2}{\lambda^{2}}-\dfrac{2(N-1)}{\lambda^{2}}r \coth r.
\]
For a fixed $a\in(0, 1)$, since $\lambda\in[\alpha, \beta]$, there exists a
universal $R_{0}>0$ such that
\[
\dfrac{4a}{\lambda^{4}}r^{2}-\dfrac{2}{\lambda^{2}}-\dfrac{2(N-1)}{\lambda
^{2}}r \coth r \geq C >0,
\]
for all $r>R_{0}$, $\lambda\in[\alpha, \beta]$. Hence, by using [Corollary
$4.2$. \cite{BMM22}], there exists a function $W \geq1$ such that
\[
-\Delta_{M, \mathbb{H}}W (x)\geq\theta W(x)-b\mathbf{1}_{B_{R_{0}}}(x),
\]
for all $x \in\mathbb{H}^{N}$, $\lambda\in[\alpha, \beta]$, where $\theta>0$
and $b\geq0$ depend on only $\alpha$ and $\beta$, and
\[
-\Delta_{M, \mathbb{H}}:=-\Delta_{\mathbb{H}}-\nabla_{\mathbb{H}}(\log
M)\cdot\nabla_{\mathbb{H}}.
\]
As a result, we get
\[
\int_{\mathbb{H}^{N}} |u|^{2}M dV_{\mathbb{H}} \leq\int_{\mathbb{H}^{N}}%
\dfrac{-\Delta_{M, \mathbb{H}}W}{\theta W}|u|^{2}MdV_{\mathbb{H}}%
+b\int_{B_{R_{0}}} \dfrac{|u|^{2}}{\theta W}MdV_{\mathbb{H}}.
\]
For the first term in the RHS, we have
\begin{align*}
\int_{\mathbb{H}^{N}}\dfrac{-\Delta_{M, \mathbb{H}}W}{\theta W}|u|^{2}%
MdV_{\mathbb{H}}  &  =-\int_{\mathbb{H}^{N}}\dfrac{\Delta_{\mathbb{H}}%
W}{\theta W}|u|^{2}MdV_{\mathbb{H}}-\int_{\mathbb{H}^{N}}\dfrac{\nabla
_{\mathbb{H}}(\log M)\cdot\nabla_{\mathbb{H}}W}{\theta W}|u|^{2}%
MdV_{\mathbb{H}}\\
&  =\int_{\mathbb{H}^{N}} \nabla_{\mathbb{H}}W\cdot\nabla_{\mathbb{H}}\left(
\dfrac{|u|^{2}M}{\theta W}\right)  dV_{\mathbb{H}}-\int_{\mathbb{H}^{N}}%
\dfrac{\nabla_{\mathbb{H}}(\log M)\cdot\nabla_{\mathbb{H}}W}{\theta W}%
|u|^{2}MdV_{\mathbb{H}}\\
&  =\int_{\mathbb{H}^{N}} \nabla_{\mathbb{H}}W\cdot\nabla_{\mathbb{H}}\left(
\dfrac{|u|^{2}}{\theta W}\right)  MdV_{\mathbb{H}}\\
&  =\dfrac{2}{\theta}\int_{\mathbb{H}^{N}}\dfrac{u}{W}\nabla_{\mathbb{H}}u
\cdot\nabla_{\mathbb{H}}W M dV_{\mathbb{H}}-\dfrac{1}{\theta}\int
_{\mathbb{H}^{N}}\dfrac{|u|^{2}|\nabla_{\mathbb{H}}W|^{2}}{W^{2}}
MdV_{\mathbb{H}}\\
&  =\dfrac{1}{\theta}\left[  \int_{\mathbb{H}^{N}}|\nabla_{\mathbb{H}}%
u|^{2}MdV_{\mathbb{H}}-\int_{\mathbb{H}^{N}}\left|  \nabla_{\mathbb{H}%
}u-\dfrac{u}{W}\nabla_{\mathbb{H}}W\right|  ^{2}MdV_{\mathbb{H}}\right] \\
&  \leq\dfrac{1}{\theta} \int_{\mathbb{H}^{N}}|\nabla_{\mathbb{H}}%
u|^{2}MdV_{\mathbb{H}}.
\end{align*}
Regarding the second term, exploiting $M=e^{-\frac{r^{2}}{\lambda^{2}}}\leq1,$
we get,
\begin{align*}
b\int_{B_{R_{0}}} \dfrac{|u|^{2}}{\theta W}MdV_{\mathbb{H}}  &  \leq C(b,
\theta)\int_{B_{R_{0}}}|u|^{2}MdV_{\mathbb{H}} \leq C(b, \theta)
\int_{B_{R_{0}}}|u|^{2}dV_{\mathbb{H}}\\
&  \leq C(b, \theta, R_{0}) \int_{B_{R_{0}}}|\nabla_{\mathbb{H}}%
u|^{2}dV_{\mathbb{H}}\;(\text{here, we assume that $\int_{B_{R_{0}}}u=0$})\\
&  \leq C(b, \theta, R_{0}) e^{\frac{R_{0}^{2}}{\alpha^{2}}}\int_{B_{R_{0}}%
}|\nabla_{\mathbb{H}}u|^{2}MdV_{\mathbb{H}} \leq C(b, \theta, R_{0}, \alpha)
\int_{\mathbb{H}^{N}}|\nabla_{\mathbb{H}}u|^{2}MdV_{\mathbb{H}}.
\end{align*}
Therefore,
\begin{align*}
\int_{\mathbb{H}^{N}}|u|^{2}MdV_{\mathbb{H}}  &  \leq\left(  \dfrac{1}{\theta
}+C(b, \theta, R_{0}, \alpha)\right)  \int_{\mathbb{H}^{N}}|\nabla
_{\mathbb{H}}u|^{2}MdV_{\mathbb{H}} \leq C(\alpha, \beta)\int_{\mathbb{H}^{N}%
}|\nabla_{\mathbb{H}}u|^{2}MdV_{\mathbb{H}}\\
&  \leq\dfrac{ C(\alpha, \beta)}{\alpha^{2}}\alpha^{2}\int_{\mathbb{H}^{N}%
}|\nabla_{\mathbb{H}}u|^{2}MdV_{\mathbb{H}} \leq C(\alpha, \beta)\lambda^{2}
\int_{\mathbb{H}^{N}}|\nabla_{\mathbb{H}}u|^{2}MdV_{\mathbb{H}},
\end{align*}
for all $\lambda\in[\alpha, \beta]$. It follows that \eqref{2ndcaseineq} holds
for all $\lambda\in[\alpha, \beta]$, where $K$ depends on $\alpha$ and $\beta$.

To sum up, by combining two cases, we can choose a constant $K$, depending on
$\beta$ only, for \eqref{2ndcaseineq} with $\lambda\in(0, \beta]$, which
completes the proof.
\end{proof}

\subsection{The third weighted Poincar\'{e} inequality}

\label{subsct3.3} In the previous section, we have proved \eqref{L^2Pineq1}
with $\lambda$ close to the origin. However, the same method does not work for
all $\lambda>0$. We want to deal with the case when $\lambda$ is large enough
by considering a different couple of $(M_{\lambda},N_{\lambda})$ in
\eqref{L^2Pgnrine}. Indeed, we consider
\[
M_{\lambda}=N_{\lambda}=e^{-\left(  \frac{r}{\sinh r}\right)  ^{\frac
{(N-1)}{2}}\frac{r^{2}}{\lambda^{2}}}%
\]
with $r:=\rho(x)$ for short.

We begin with the following lemma:

\begin{lemma}
Given a positive density function $M$, we have
\[
\int_{\mathbb{H}^{N}}|\nabla u|^{2}MdV_{\mathbb{H}}\geq\int_{\mathbb{H}^{N}%
}\left(  \dfrac{(N-1)^{2}}{4}+\dfrac{1}{4}|\nabla_{\mathbb{H}}v|^{2}-\dfrac
{1}{2}\Delta_{\mathbb{H}}v\right)  |u|^{2}MdV_{\mathbb{{H}}},
\]
where $v:=-\log M$.
\end{lemma}

\begin{proof}
Since $M>0$, we are able to set $g=u\sqrt{M}$. Then,
\[
\nabla_{\mathbb{H}}u=\dfrac{\nabla_{\mathbb{H}}g}{\sqrt{M}}+\dfrac{1}{2}%
\dfrac{g}{\sqrt{M}}\nabla_{\mathbb{H}}v,
\]
which in turn implies,
\[
|\nabla_{\mathbb{H}}u|^{2}=\dfrac{|\nabla_{\mathbb{H}}g|^{2}}{M}+\dfrac{1}%
{4}\dfrac{|g|^{2}}{M}|\nabla_{\mathbb{H}}v|^{2}+\dfrac{g\nabla_{\mathbb{H}%
}g\nabla_{\mathbb{H}}v}{M}.
\]
As a result,
\begin{align*}
\int_{\mathbb{H}^{N}}\left\vert \nabla_{\mathbb{H}}u\right\vert ^{2}%
MdV_{\mathbb{H}}  &  =\int_{\mathbb{H}^{N}}\left(  |\nabla_{\mathbb{H}}%
g|^{2}+\dfrac{1}{4}|g|^{2}|\nabla_{\mathbb{H}}v|^{2}+\dfrac{1}{2}%
\nabla_{\mathbb{H}}(g^{2})\nabla_{\mathbb{H}}v\right)  dV_{\mathbb{H}}\\
&  =\int_{\mathbb{H}^{N}}|\nabla_{\mathbb{H}}g|^{2}dV_{\mathbb{H}}%
+\int_{\mathbb{H}^{N}}\left(  \dfrac{1}{4}|\nabla_{\mathbb{H}}v|^{2}-\dfrac
{1}{2}\Delta_{\mathbb{H}}v\right)  |g|^{2}dV_{\mathbb{H}}\\
&  \geq\dfrac{(N-1)^{2}}{4}\int_{\mathbb{H}^{N}}|g|^{2}dV_{\mathbb{H}}%
+\int_{\mathbb{H}^{N}}\left(  \dfrac{1}{4}|\nabla_{\mathbb{H}}v|^{2}-\dfrac
{1}{2}\Delta_{\mathbb{H}}v\right)  |g|^{2}dV_{\mathbb{H}},
\end{align*}
which gives us our results.
\end{proof}

The above simple lemma would be extremely useful if we can find out a
potential $M$ such that
\[
\lambda^{2}\left(  \dfrac{(N-1)^{2}}{4}+\dfrac{1}{4}|\nabla_{\mathbb{H}}
v(x)|^{2}-\dfrac{1}{2}\Delta_{\mathbb{H}} v(x)\right)  \geq K,
\]
for all $\lambda$ large enough, where $v=-\log M$ the constant $K>0$ is
independent of $\lambda$ and $x \in\mathbb{H}^{N}$.

\begin{proof}
[Proof of Theorem \ref{3rdineq}]We will work with $M_{\lambda}=e^{-\left(
\frac{r}{\sinh r}\right)  ^{\frac{(N-1)}{2}}\frac{r^{2}}{\lambda^{2}}}$ with
$r=\rho(x)$. Indeed, as above, setting
\[
v(r)=-\log M_{\lambda}=\left(  \frac{r}{\sinh r}\right)  ^{\frac{(N-1)}{2}%
}\cdot\frac{r^{2}}{\lambda^{2}}.
\]
By computing directly, we get
\[
\dfrac{\partial v}{\partial r}=\dfrac{N+3}{2\lambda^{2}}\dfrac{r^{\frac
{(N+1)}{2}}}{(\sinh r)^{\frac{(N-1)}{2}}}-\dfrac{N-1}{2\lambda^{2}}%
\dfrac{r^{\frac{(N+3)}{2}}\cosh r}{(\sinh r)^{\frac{(N+1)}{2}}},
\]
and
\begin{align*}
\dfrac{\partial^{2}v}{\partial r^{2}}  &  =\dfrac{(N+1)(N+3)}{4\lambda^{2}%
}\left(  \dfrac{r}{\sinh r}\right)  ^{\frac{(N-1)}{2}}-\dfrac{(N-1)(N+3)}%
{2\lambda^{2}}\left(  \dfrac{r}{\sinh r}\right)  ^{\frac{(N+1)}{2}}\cosh r\\
&  -\dfrac{N-1}{2\lambda^{2}}\dfrac{r^{\frac{(N+3)}{2}}}{(\sinh r)^{\frac
{(N-1)}{2}}}+\dfrac{(N-1)(N+1)}{4\lambda^{2}}\left(  \dfrac{r}{\sinh
r}\right)  ^{\frac{(N+3)}{2}}(\cosh r)^{2}.
\end{align*}
Thus,
\begin{align*}
\dfrac{1}{4}|\nabla_{\mathbb{H}}v|^{2}-\dfrac{1}{2}\Delta_{\mathbb{H}}v  &
=\dfrac{1}{4}\left(  \dfrac{\partial v}{\partial r}\right)  ^{2}-\dfrac{1}%
{2}\left(  \dfrac{\partial^{2}v}{\partial r^{2}}+(N-1)\coth r\dfrac{\partial
v}{\partial r}\right) \\
&  =\dfrac{(N+3)^{2}}{16\lambda^{4}}\dfrac{r^{N+1}}{(\sinh r)^{N-1}}%
+\dfrac{(N-1)^{2}}{16\lambda^{4}}\dfrac{r^{N+3}(\cosh r)^{2}}{(\sinh r)^{N+1}%
}\\
&  -\dfrac{(N-1)(N+3)}{8\lambda^{4}}\dfrac{r^{N+2}\cosh r}{(\sinh r)^{N}%
}-\dfrac{(N+1)(N+3)}{8\lambda^{2}}\left(  \dfrac{r}{\sinh r}\right)
^{\frac{(N-1)}{2}}\\
&  +\dfrac{(N-3)(N-1)}{8\lambda^{2}}\left(  \dfrac{r}{\sinh r}\right)
^{\frac{(N+3)}{2}}\cosh^{2}r+\dfrac{N-1}{4\lambda^{2}}\dfrac{r^{\frac
{(N+3)}{2}}}{(\sinh r)^{\frac{(N-1)}{2}}}.
\end{align*}
Let us define
\[
V(r):=\lambda^{2}\left(  \dfrac{(N-1)^{2}}{4}+\dfrac{1}{4}|\nabla_{\mathbb{H}%
}v|^{2}-\dfrac{1}{2}\Delta_{\mathbb{H}}v\right)  .
\]
We can see that for $\lambda$ large enough, there exist positive constants
$K,\beta$ such that
\[
V(r)\geq K
\]
for all $r\geq0$, $\lambda\geq\beta$, $N\geq2$, where $K$ is independent of
$r,\lambda$. Therefore, with the measure $M_{\lambda}$ defined above, we have
already proved that
\begin{equation}
\lambda^{2}\int_{\mathbb{H}^{N}}|\nabla_{\mathbb{H}}u|^{2}M_{\lambda
}dV_{\mathbb{{H}}}\geq K\int_{\mathbb{H}^{N}}|u|^{2}M_{\lambda}dV_{\mathbb{{H}%
}}\geq K\inf_{c}\int_{\mathbb{H}^{N}}|u-c|^{2}M_{\lambda}dV_{\mathbb{{H}}}
\label{eq 2.12}%
\end{equation}
for all $\lambda\geq\beta$, where $K>0$ is independent of $\lambda$, as desired.
\end{proof}

Theorem \ref{2ndineq} and Theorem \ref{3rdineq} can be consolidated and
expressed as a single result, as follows:

\begin{corollary}
Given $\beta>0$ defined as in Theorem \ref{3rdineq}, for any $\alpha\geq\beta
$, we have
\begin{equation}
\dfrac{\lambda^{2}}{2}\int_{\mathbb{H}^{N}}\left\vert \nabla_{\mathbb{H}%
}u\right\vert ^{2}\left(  A_{\lambda}\mathbf{1}_{\{\lambda\leq\alpha
\}}+B_{\lambda}\mathbf{1}_{\{\lambda>\alpha\}}\right)  dV_{\mathbb{H}}\geq
K\inf_{c}\int_{\mathbb{H}^{N}}|u-c|^{2}\left(  A_{\lambda}\mathbf{1}%
_{\{\lambda\leq\alpha\}}+B_{\lambda}\mathbf{1}_{\{\lambda>\alpha\}}\right)
dV_{\mathbb{H}}, \label{comb2nd&3rd}%
\end{equation}
for all $\lambda>0$, where $K=K(\alpha)>0$ is independent of $\lambda$,
$A_{\lambda}=e^{-\frac{\rho^{2}(x)}{\lambda^{2}}}$, and $B_{\lambda
}=e^{-\left(  \frac{\rho(x)}{\sinh\rho(x)}\right)  ^{\frac{(N-1)}{2}}%
\frac{\rho^{2}(x)}{\lambda^{2}}}$.
\end{corollary}

In particular, when $\alpha$ is large enough, it is straightforward to get the
following weighted Poincar\'{e} inequality, i.e.

\begin{corollary}
\label{weightedPoinIneq}Let $\beta>0$. Then we have
\[
\int_{\mathbb{H}^{N}}\left\vert \nabla_{\mathbb{H}}u\right\vert ^{2}%
e^{-\beta\rho^{2}(x)}dV_{\mathbb{H}}\geq K\left(  \beta\right)  \inf_{c}%
\int_{\mathbb{H}^{N}}|u-c|^{2}e^{-\beta\rho^{2}(x)}dV_{\mathbb{H}},
\]
where $K>0$ is a positive constant.
\end{corollary}


\section{$L^{2}$-Stability of the Heisenberg Uncertainty Principle-Proofs of
Theorems \ref{Cor 3.1}, \ref{Cor 3.2} and \ref{finalstability}}

\label{sct5} As applications of the above weighted Poincar\'{e} inequalities,
we can get the corresponding $L^{2}$-stability results on the hyperbolic
space. With $E:=\left\{  ce^{-\alpha\rho^{2}(x)}:c\in\mathbb{R},\alpha
>0\right\}  $ the set of the Gaussian functions, and $M$ a density on
$\mathbb{H}^{N}$, we define
\[
d_{1}(u,E,M):=\inf\left\{  \left(  \int_{\mathbb{H}^{N}}\left\vert
u-ce^{-\alpha\rho^{2}(x)}\right\vert ^{2}MdV_{\mathbb{H}}\right)  ^{1/2}%
:c\in\mathbb{R},\alpha>0\right\}  .
\]
Firstly, we can get a scale non-invariant version of the $L^{2}$-stability
result, i.e.

\begin{theorem}
\label{Cor 3.0} Let $N\geq2.$ For all $u\in W^{1,2}(\mathbb{H}^{N}%
)\cap\{u:\int_{\mathbb{H}^{N}}\rho^{2}(x)|u|^{2}dV_{\mathbb{H}}<\infty\}$
there holds
\begin{align*}
&  \int_{\mathbb{H}^{N}}|\nabla_{\mathbb{H}}u|^{2}dV_{\mathbb{H}}%
+\int_{\mathbb{H}^{N}}\rho^{2}(x)|u|^{2}dV_{\mathbb{H}}\\
&  -\int_{\mathbb{H}^{N}}\left[  N+(N-1)\dfrac{\rho(x)\cosh\rho(x)-\sinh
\rho(x)}{\sinh\rho(x)}\right]  |u|^{2}dV_{\mathbb{H}}\\
&  \geq K\inf_{c\in\mathbb{R}}\int_{\mathbb{H}^{N}}\left\vert u-ce^{-\frac
{\rho^{2}(x)}{2}}\right\vert ^{2}dV_{\mathbb{H}}\\
&  \geq Kd_{1}^{2}(u,E,1).
\end{align*}

\end{theorem}

\begin{proof}
[Proof of Theorem \ref{Cor 3.0}]By utilizing $(r^{N-1},r^{N-1}(N-r^{2}))$ as a
Bessel pair with $\varphi(r)=e^{-\frac{r^{2}}{2}}$ in Theorem~\ref{Cor 1.1},
we get
\begin{align*}
&  \int_{\mathbb{H}^{N}}|\nabla_{\mathbb{H}}u|^{2}dV_{\mathbb{H}}%
-\int_{\mathbb{H}^{N}}(N-\rho^{2}(x))|u|^{2}dV_{\mathbb{H}}\\
&  =\int_{\mathbb{H}^{N}}\left\vert \nabla_{\mathbb{H}}\left(  ue^{\rho
^{2}(x)/2}\right)  \right\vert ^{2}e^{-\rho^{2}(x)}dV_{\mathbb{H}}\\
&  +(N-1)\int_{\mathbb{H}^{N}}\dfrac{\rho(x)\cosh\rho(x)-\sinh\rho(x)}%
{\sinh\rho(x)}|u|^{2}dV_{\mathbb{H}}.
\end{align*}
It follows
\begin{align*}
&  \int_{\mathbb{H}^{N}}|\nabla_{\mathbb{H}}u|^{2}dV_{\mathbb{H}}%
+\int_{\mathbb{H}^{N}}\rho^{2}(x)|u|^{2}dV_{\mathbb{H}}\\
&  -\int_{\mathbb{H}^{N}}\left[  N+(N-1)\dfrac{\rho(x)\cosh\rho(x)-\sinh
\rho(x)}{\sinh\rho(x)}\right]  |u|^{2}dV_{\mathbb{H}}\\
&  =\int_{\mathbb{H}^{N}}\left\vert \nabla_{\mathbb{H}}\left(  ue^{\rho
^{2}(x)/2}\right)  \right\vert ^{2}e^{-\rho^{2}(x)}dV_{\mathbb{H}}\\
&  \geq K\inf_{c\in\mathbb{R}}\int_{\mathbb{H}^{N}}|ue^{\rho^{2}(x)/2}%
-c|^{2}e^{-\rho^{2}(x)}dV_{\mathbb{H}}\\
&  \geq K\inf_{c\in\mathbb{R}}\int_{\mathbb{H}^{N}}\left\vert u-ce^{-\frac
{\rho^{2}(x)}{2}}\right\vert ^{2}dV_{\mathbb{H}},
\end{align*}
as desired. The first inequality is from Corollary \ref{weightedPoinIneq}.
\end{proof}

Inspired from \cite[Theorem~1.4]{CFNL}, we are able to get another version of
the scale non-invariant Heisenberg Uncertainty Principle. Indeed,

\begin{corollary}
With $K$ the constant defined in Corollary \ref{weightedPoinIneq}, we have
\begin{align*}
&  \int_{\mathbb{H}^{N}}|\nabla_{\mathbb{H}}u|^{2}dV_{\mathbb{H}}%
+\int_{\mathbb{H}^{N}}\rho^{2}(x)|u|^{2}dV_{\mathbb{H}}dV_{\mathbb{H}}%
-\int_{\mathbb{H}^{N}}\left[  N+(N-1)\dfrac{\rho(x)\cosh\rho(x)-\sinh\rho
(x)}{\sinh\rho(x)}\right]  |u|^{2}dV_{\mathbb{H}}\\
&  \geq\dfrac{K}{K+1}\inf_{c}\int_{\mathbb{H}^{N}}\left\vert \nabla
_{\mathbb{H}}(u-ce^{-\frac{\rho^{2}(x)}{2}})\right\vert ^{2}dV_{\mathbb{H}%
}+\int_{\mathbb{H}^{N}}\left\vert u-ce^{-\frac{\rho^{2}(x)}{2}}\right\vert
^{2}dV_{\mathbb{H}}\\
&  -\dfrac{N-1}{8}\int_{\mathbb{H}^{N}}\left\vert u-ce^{-\frac{\rho^{2}(x)}%
{2}}\right\vert ^{2}\dfrac{\rho(x)}{\tanh\frac{\rho(x)}{2}}\left(  1-\tanh
^{2}\left(  \dfrac{\rho(x)}{2}\right)  \right)  ^{3}dV_{\mathbb{H}}\\
&  +\dfrac{1}{4}\int_{\mathbb{H}^{N}}\left\vert u-ce^{-\frac{\rho^{2}(x)}{2}%
}\right\vert ^{2}\rho^{2}(x)\left(  1-\tanh^{2}\left(  \dfrac{\rho(x)}%
{2}\right)  \right)  ^{2}dV_{\mathbb{H}}\\
&  -\frac{1}{4}\int_{\mathbb{H}^{N}}\left\vert u-ce^{-\frac{\rho^{2}(x)}{2}%
}\right\vert ^{2}\left(  1-\tanh^{2}\left(  \dfrac{\rho(x)}{2}\right)
\right)  ^{2}dV_{\mathbb{H}}\\
&  -\dfrac{N-3}{4}\int_{\mathbb{H}^{N}}\left\vert u-ce^{-\frac{\rho^{2}(x)}%
{2}}\right\vert ^{2}\rho(x)\tanh\left(  \frac{\rho(x)}{2}\right)  \left(
1-\tanh^{2}\left(  \dfrac{\rho(x)}{2}\right)  \right)  ^{2}dV_{\mathbb{H}}.
\end{align*}

\end{corollary}

\begin{proof}
Let $u=ge^{-r^{2}/2}$ and $r=\rho(x)$, then
\begin{align*}
\int_{\mathbb{H}^{N}}\left\vert \nabla_{\mathbb{H}}(u-ce^{-r^{2}%
/2})\right\vert ^{2}dV_{\mathbb{H}}  &  =\int_{\mathbb{H}^{N}}\left\vert
\nabla_{\mathbb{H}}((g-c)e^{-r^{2}/2})\right\vert ^{2}dV_{\mathbb{H}}\\
&  =\int_{\mathbb{H}^{N}}\left\vert \nabla_{\mathbb{H}}g\right\vert
^{2}e^{-r^{2}}dV_{\mathbb{H}}+\int_{\mathbb{H}^{N}}\left\vert g-c\right\vert
^{2}\left\vert \nabla_{\mathbb{H}}(e^{-r^{2}/2})\right\vert ^{2}%
dV_{\mathbb{H}}\\
&  +2\int_{\mathbb{H}^{N}}e^{-r^{2}/2}(g-c)\nabla_{\mathbb{H}}(g-c)\cdot
\nabla_{\mathbb{H}}(e^{-r^{2}/2})dV_{\mathbb{H}}.
\end{align*}
We have
\[
\nabla_{\mathbb{H}}(e^{-r^{2}/2})=\left(  \dfrac{1-|x|^{2}}{2}\right)
^{2}\nabla\left(  e^{-r^{2}/2}\right)  =-\left(  \dfrac{1-|x|^{2}}{2}\right)
re^{-r^{2}/2}\dfrac{x}{|x|},
\]
where $\nabla$ is the gradient operator on the Euclidean space. It follows
\[
\left\vert \nabla_{\mathbb{H}}(e^{-r^{2}/2})\right\vert ^{2}=\left(
\dfrac{1-|x|^{2}}{2}\right)  ^{2}r^{2}e^{-r^{2}}.
\]
Thus,
\begin{equation}
\int_{\mathbb{H}^{N}}\left\vert g-c\right\vert ^{2}\left\vert \nabla
_{\mathbb{H}}(e^{-r^{2}/2})\right\vert ^{2}dV_{\mathbb{H}}=\int_{\mathbb{H}%
^{N}}\left\vert g-c\right\vert ^{2}r^{2}e^{-r^{2}}\left(  \dfrac{1-|x|^{2}}%
{2}\right)  ^{2}dV_{\mathbb{H}}. \label{IIterm}%
\end{equation}
Moreover,
\begin{align*}
&  2\int_{\mathbb{H}^{N}}e^{-r^{2}/2}(g-c)\nabla_{\mathbb{H}}(g-c)\cdot
\nabla_{\mathbb{H}}(e^{-r^{2}/2})dV_{\mathbb{H}}\\
&  =-\int_{B(0,1)}e^{-r^{2}}\dfrac{r}{|x|}\left(  \dfrac{1-|x|^{2}}{2}\right)
^{3-N}\nabla((g-c)^{2})\cdot xdx\\
&  =\int_{B(0,1)}\left\vert g-c\right\vert ^{2}\operatorname{div}\left(
e^{-r^{2}}r\left(  \dfrac{1-|x|^{2}}{2}\right)  ^{3-N}\frac{x}{|x|}\right)
dx.
\end{align*}
We have
\begin{align*}
\operatorname{div}\left(  e^{-r^{2}}r\left(  \dfrac{1-|x|^{2}}{2}\right)
^{3-N}\frac{x}{|x|}\right)   &  =\sum_{i=1}^{N}\dfrac{\partial}{\partial
x_{i}}\left(  e^{-r^{2}}r\left(  \dfrac{1-|x|^{2}}{2}\right)  ^{3-N}%
\frac{x_{i}}{|x|}\right) \\
&  =Ne^{-r^{2}}\dfrac{r}{|x|}\left(  \dfrac{1-|x|^{2}}{2}\right)  ^{3-N}%
+\sum_{i=1}^{N}x_{i}\dfrac{\partial}{\partial x_{i}}\left(  e^{-r^{2}}%
\dfrac{r}{|x|}\left(  \dfrac{1-|x|^{2}}{2}\right)  ^{3-N}\right) \\
&  =(N-1)e^{-r^{2}}\dfrac{r}{|x|}\left(  \dfrac{1-|x|^{2}}{2}\right)
^{3-N}-2e^{-r^{2}}r^{2}\left(  \dfrac{1-|x|^{2}}{2}\right)  ^{2-N}\\
&  +e^{-r^{2}}\left(  \dfrac{1-|x|^{2}}{2}\right)  ^{2-N}+(N-3)e^{-r^{2}%
}r|x|\left(  \dfrac{1-|x|^{2}}{2}\right)  ^{2-N}.
\end{align*}
Therefore,
\begin{align}
&  2\int_{\mathbb{H}^{N}}e^{-r^{2}/2}(g-c)\nabla_{\mathbb{H}}(g-c)\cdot
\nabla_{\mathbb{H}}(e^{-r^{2}/2})dV_{\mathbb{H}}\nonumber\label{IIIterm}\\
&  =(N-1)\int_{\mathbb{H}^{N}}\left\vert g-c\right\vert ^{2}e^{-r^{2}}%
\dfrac{r}{|x|}\left(  \dfrac{1-|x|^{2}}{2}\right)  ^{3}dV_{\mathbb{H}}%
-2\int_{\mathbb{H}^{N}}\left\vert g-c\right\vert ^{2}e^{-r^{2}}r^{2}\left(
\dfrac{1-|x|^{2}}{2}\right)  ^{2}dV_{\mathbb{H}}\nonumber\\
&  +\int_{\mathbb{H}^{N}}\left\vert g-c\right\vert ^{2}e^{-r^{2}}\left(
\dfrac{1-|x|^{2}}{2}\right)  ^{2}dV_{\mathbb{H}}+(N-3)\int_{\mathbb{H}^{N}%
}\left\vert g-c\right\vert ^{2}e^{-r^{2}}r|x|\left(  \dfrac{1-|x|^{2}}%
{2}\right)  ^{2}dV_{\mathbb{H}}.
\end{align}
From \eqref{IIterm} and \eqref{IIIterm}, we get
\begin{align*}
\int_{\mathbb{H}^{N}}\left\vert \nabla_{\mathbb{H}}(u-ce^{-r^{2}%
/2})\right\vert ^{2}dV_{\mathbb{H}}  &  =\int_{\mathbb{H}^{N}}\left\vert
\nabla_{\mathbb{H}}((g-c)e^{-r^{2}/2})\right\vert ^{2}dV_{\mathbb{H}}\\
&  =\int_{\mathbb{H}^{N}}\left\vert \nabla_{\mathbb{H}}g\right\vert
^{2}e^{-r^{2}}dV_{\mathbb{H}}+(N-1)\int_{\mathbb{H}^{N}}\left\vert
g-c\right\vert ^{2}e^{-r^{2}}\dfrac{r}{|x|}\left(  \dfrac{1-|x|^{2}}%
{2}\right)  ^{3}dV_{\mathbb{H}}\\
&  -\int_{\mathbb{H}^{N}}\left\vert g-c\right\vert ^{2}e^{-r^{2}}r^{2}\left(
\dfrac{1-|x|^{2}}{2}\right)  ^{2}dV_{\mathbb{H}}+\int_{\mathbb{H}^{N}%
}\left\vert g-c\right\vert ^{2}e^{-r^{2}}\left(  \dfrac{1-|x|^{2}}{2}\right)
^{2}dV_{\mathbb{H}}\\
&  +(N-3)\int_{\mathbb{H}^{N}}\left\vert g-c\right\vert ^{2}e^{-r^{2}%
}r|x|\left(  \dfrac{1-|x|^{2}}{2}\right)  ^{2}dV_{\mathbb{H}}.
\end{align*}
Hence
\begin{align*}
&  \int_{\mathbb{H}^{N}}\left\vert \nabla_{\mathbb{H}}(u-ce^{-r^{2}%
/2})\right\vert ^{2}dV_{\mathbb{H}}+\int_{\mathbb{H}^{N}}\left\vert
u-ce^{-r^{2}/2}\right\vert ^{2}dV_{\mathbb{H}}\\
&  -\dfrac{N-1}{8}\int_{\mathbb{H}^{N}}\left\vert g-c\right\vert ^{2}%
e^{-r^{2}}\dfrac{r}{|x|}\left(  1-\tanh^{2}\left(  \dfrac{r}{2}\right)
\right)  ^{3}dV_{\mathbb{H}}+\dfrac{1}{4}\int_{\mathbb{H}^{N}}\left\vert
g-c\right\vert ^{2}e^{-r^{2}}r^{2}\left(  1-\tanh^{2}\left(  \dfrac{r}%
{2}\right)  \right)  ^{2}dV_{\mathbb{H}}\\
&  -\frac{1}{4}\int_{\mathbb{H}^{N}}\left\vert g-c\right\vert ^{2}e^{-r^{2}%
}\left(  1-\tanh^{2}\left(  \dfrac{r}{2}\right)  \right)  ^{2}dV_{\mathbb{H}%
}-\dfrac{N-3}{4}\int_{\mathbb{H}^{N}}\left\vert g-c\right\vert ^{2}e^{-r^{2}%
}r\tanh\left(  \frac{r}{2}\right)  \left(  1-\tanh^{2}\left(  \dfrac{r}%
{2}\right)  \right)  ^{2}dV_{\mathbb{H}}\\
&  =\int_{\mathbb{H}^{N}}\left\vert \nabla_{\mathbb{H}}g\right\vert
^{2}e^{-r^{2}}dV_{\mathbb{H}}+\int_{\mathbb{H}^{N}}\left\vert g-c\right\vert
^{2}e^{-r^{2}}dV_{\mathbb{H}}.
\end{align*}
Consequently, from Corollary \ref{weightedPoinIneq},
\begin{align*}
&  \dfrac{K}{K+1}\inf_{c}\left\{  \int_{\mathbb{H}^{N}}\left\vert
\nabla_{\mathbb{H}}(u-ce^{r^{2}/2})\right\vert ^{2}dV_{\mathbb{H}}%
+\int_{\mathbb{H}^{N}}\left\vert u-ce^{-r^{2}/2}\right\vert ^{2}%
dV_{\mathbb{H}}\right. \\
&  -\dfrac{N-1}{8}\int_{\mathbb{H}^{N}}\left\vert g-c\right\vert ^{2}%
e^{-r^{2}}\dfrac{r}{|x|}\left(  1-\tanh^{2}\left(  \dfrac{r}{2}\right)
\right)  ^{3}dV_{\mathbb{H}}+\dfrac{1}{4}\int_{\mathbb{H}^{N}}\left\vert
g-c\right\vert ^{2}e^{-r^{2}}r^{2}\left(  1-\tanh^{2}\left(  \dfrac{r}%
{2}\right)  \right)  ^{2}dV_{\mathbb{H}}\\
&  \left.  -\frac{1}{4}\int_{\mathbb{H}^{N}}\left\vert g-c\right\vert
^{2}e^{-r^{2}}\left(  1-\tanh^{2}\left(  \dfrac{r}{2}\right)  \right)
^{2}dV_{\mathbb{H}}-\dfrac{N-3}{4}\int_{\mathbb{H}^{N}}\left\vert
g-c\right\vert ^{2}e^{-r^{2}}r\tanh\left(  \frac{r}{2}\right)  \left(
1-\tanh^{2}\left(  \dfrac{r}{2}\right)  \right)  ^{2}dV_{\mathbb{H}}\right\}
\\
&  \leq\dfrac{K}{K+1}\left(  \int_{\mathbb{H}^{N}}\left\vert \nabla
_{\mathbb{H}}g\right\vert ^{2}e^{-r^{2}}dV_{\mathbb{H}}+\dfrac{1}{K}%
\int_{\mathbb{H}^{N}}\left\vert \nabla_{\mathbb{H}}g\right\vert ^{2}e^{-r^{2}%
}dV_{\mathbb{H}}\right) \\
&  \leq\int_{\mathbb{H}^{N}}\left\vert \nabla_{\mathbb{H}}g\right\vert
^{2}e^{-r^{2}}dV_{\mathbb{H}}=\int_{\mathbb{H}^{N}}\left\vert \nabla
_{\mathbb{H}}(ue^{r^{2}/2})\right\vert ^{2}e^{-r^{2}}dV_{\mathbb{H}},
\end{align*}
which implies our result.
\end{proof}

Next, some scale invariant $L^{2}$-stability results are given as direct
consequences of the weighted Poincar\'{e} inequalities. Recall that the
Heisenberg deficit is defined by
\begin{align*}
\delta_{1,\mathbb{H}}(u):  &  =\left(  \int_{\mathbb{H}^{N}}|\nabla
_{\mathbb{H}}u|^{2}dV_{\mathbb{H}}\right)  ^{1/2}\left(  \int_{\mathbb{H}^{N}%
}\rho^{2}(x)|u|^{2}dV_{\mathbb{H}}\right)  ^{1/2}\\
&  -\dfrac{1}{2}\int_{\mathbb{H}^{N}}\left[  N+(N-1)\dfrac{\rho(x)\cosh
\rho(x)-\ \sinh\rho(x)}{\sinh\rho(x)}\right]  |u|^{2}dV_{\mathbb{H}}.
\end{align*}
Besides, we also consider another Heisenberg deficit as follows
\begin{align*}
\delta_{2,\mathbb{H}}(u):  &  =\left(  \int_{\mathbb{H}^{N}}|\nabla
_{\mathbb{H}}u|^{2}dV_{\mathbb{H}}\right)  \left(  \int_{\mathbb{H}^{N}}%
\rho^{2}(x)|u|^{2}dV_{\mathbb{H}}\right) \\
&  -\dfrac{1}{4}\left(  \int_{\mathbb{H}^{N}}\left[  N+(N-1)\dfrac
{\rho(x)\cosh\rho(x)-\sinh\rho(x)}{\sinh\rho(x)}\right]  |u|^{2}%
dV_{\mathbb{H}}\right)  ^{2}.
\end{align*}

By taking the advantage of Theorem \ref{1stineq}, we will provide a proof for
Theorem \ref{Cor 3.1} and prove that
\begin{align*}
\delta_{1,\mathbb{H}}(u)  &  \geq K\inf_{c\in\mathbb{R}\text{,~}\lambda>0}%
\int_{\mathbb{H}^{N}}\left\vert u-ce^{-\frac{\rho^{2}(x)}{2\lambda^{2}}%
}\right\vert ^{2}\left(  1-\tanh^{2}\left(  \dfrac{\rho(x)}{2}\right)
\right)  ^{N}dV_{\mathbb{H}}\\
&  \geq K\inf_{c\in\mathbb{R}\text{,~}\lambda>0}\left\Vert u-ce^{-\frac
{\rho^{2}(x)}{2\lambda^{2}}}\right\Vert _{L^{2}(\mathbb{H}^{N},M)}^{2}\\
&  \geq Kd_{1}^{2}(u,E,M)
\end{align*}
and
\[
\delta_{2,\mathbb{H}}(u)\geq K\left(  \int_{\mathbb{H}^{N}}\left[
N+(N-1)\dfrac{\rho(x)\cosh\rho(x)-\sinh\rho(x)}{\sinh\rho(x)}\right]
|u|^{2}dV_{\mathbb{H}}\right)  d_{1}^{2}(u,E,M)+K^{2}d_{1}^{4}(u,E,M).
\]

\begin{proof}
[Proof of Theorem \ref{Cor 3.1}]From \eqref{DeficitIdt}, we have
\begin{align*}
\delta_{1,\mathbb{H}}(u)  &  =\dfrac{\lambda^{2}}{2}\int_{\mathbb{H}^{N}%
}e^{-\frac{\rho^{2}(x)}{\lambda^{2}}}\left\vert \nabla_{\mathbb{H}}\left(
ue^{\frac{\rho^{2}(x)}{2\lambda^{2}}}\right)  \right\vert ^{2}dV_{\mathbb{H}%
}\\
&  \geq K\inf_{c}\int_{\mathbb{H}^{N}}\left\vert ue^{\frac{\rho^{2}%
(x)}{2\lambda^{2}}}-c\right\vert ^{2}\left(  1-\tanh^{2}\left(  \dfrac
{\rho(x)}{2}\right)  \right)  ^{N}e^{-\frac{\rho^{2}(x)}{\lambda^{2}}%
}dV_{\mathbb{H}}\\
&  \geq K\inf_{c}\int_{\mathbb{H}^{N}}\left\vert u-ce^{-\frac{\rho^{2}%
(x)}{2\lambda^{2}}}\right\vert ^{2}\left(  1-\tanh^{2}\left(  \dfrac{\rho
(x)}{2}\right)  \right)  ^{N}dV_{\mathbb{H}},
\end{align*}
as desired, where the first inequality is from Theorem \ref{1stineq}.
\end{proof}

Next, given $\beta>0$, we define the set
\[
\mathcal{A}_{\beta}:=\left\{  u\in W^{1,2}(\mathbb{H}^{N})\;\text{such
that}\;\int_{\mathbb{H}^{N}}\rho^{2}(x)|u|^{2}dV_{\mathbb{H}}<\infty
\;\text{and}\;\dfrac{\int_{\mathbb{H}^{N}}\rho^{2}(x)|u|^{2}dV_{\mathbb{H}}%
}{\int_{\mathbb{H}^{N}}|\nabla_{\mathbb{H}}u|^{2}dV_{\mathbb{H}}}\leq\beta
^{4}\;\right\}  .
\]

Using this space of functions and Theorem \ref{2ndineq}, we have the stability
result stated in Theorem \ref{Cor 3.2}:

\begin{proof}
[Proof of Theorem \ref{Cor 3.2}]Similarly, using \eqref{DeficitIdt}, it
follows us
\begin{align*}
\delta_{1,\mathbb{H}}(u)  &  =\dfrac{\lambda^{2}}{2}\int_{\mathbb{H}^{N}%
}e^{-\frac{\rho^{2}(x)}{\lambda^{2}}}\left\vert \nabla_{\mathbb{H}}\left(
ue^{\frac{\rho^{2}(x)}{2\lambda^{2}}}\right)  \right\vert ^{2}dV_{\mathbb{H}%
}\\
&  \geq K\inf_{c}\int_{\mathbb{H}^{N}}\left\vert ue^{\frac{\rho^{2}%
(x)}{2\lambda^{2}}}-c\right\vert ^{2}e^{-\frac{\rho^{2}(x)}{\lambda^{2}}%
}dV_{\mathbb{H}}\\
&  \geq K\inf_{c}\int_{\mathbb{H}^{N}}\left\vert u-ce^{-\frac{\rho^{2}%
(x)}{2\lambda^{2}}}\right\vert ^{2}dV_{\mathbb{H}},
\end{align*}
as desired. Here, we notice that in the first equality, we used the fact that
\[
\lambda=\left(  \dfrac{\int_{\mathbb{H}^{N}}\rho^{2}(x)|u|^{2}dV_{\mathbb{H}}%
}{\int_{\mathbb{H}^{N}}|\nabla_{\mathbb{H}}u|^{2}dV_{\mathbb{H}}}\right)
^{1/4},
\]
which is the reason why we just care about $u\in\mathcal{A}_{\beta}$.
\end{proof}

Next, we point out another $L^{2}$-stability result for the case when the
function $u$ belongs to the complement of $\mathcal{A}_{\beta}$. More precisely,

\begin{corollary}
\label{Cor 3.3} With $\beta$ defined as in Theorem \ref{3rdineq}, for any
$u\in\overline{\mathcal{A}_{\beta}}$, the complement of $\mathcal{A}_{\beta}$,
we have
\begin{align*}
\Tilde{\delta}_{1,\mathbb{H}}(u)  &  \geq K\inf_{\lambda\geq\beta
,c\in\mathbb{R}}\int_{\mathbb{H}^{N}}\left\vert u-ce^{-\left(  \frac{\rho
(x)}{\sinh\rho(x)}\right)  ^{\frac{(N-1)}{2}}\frac{\rho^{2}(x)}{2\lambda^{2}}%
}\right\vert ^{2}dV_{\mathbb{H}}\\
&  =K\inf_{\lambda\geq\beta,c\in\mathbb{R}}\left\Vert u-ce^{-\left(
\frac{\rho(x)}{\sinh\rho(x)}\right)  ^{\frac{(N-1)}{2}}\frac{\rho^{2}%
(x)}{2\lambda^{2}}}\right\Vert _{L^{2}(\mathbb{H}^{N})}^{2},
\end{align*}
where $K$ is a universal constant depending only on $\beta$.
\end{corollary}

\begin{proof}
We will make use of Theorem \ref{3rdineq}. First, from \eqref{L2idt}, we have
\begin{align*}
&  \lambda^{2}\int_{\mathbb{H}^{N}}|\nabla_{\mathbb{H}}u|^{2}dV_{\mathbb{H}%
}+\lambda^{-2}\int_{\mathbb{H}^{N}}\rho^{2}(x)|u|^{2}dV_{\mathbb{H}}\\
&  -\int_{\mathbb{H}^{N}}\left(  N+(N-1)\dfrac{\rho(x)\cosh\rho(x)-\sinh
\rho(x)}{\sinh\rho(x)}\right)  \;|u|^{2}dV_{\mathbb{H}}\\
&  =\int_{\mathbb{H}^{N}}\left\vert \dfrac{u\rho(x)\nabla_{\mathbb{H}}%
(\rho(x))}{\lambda}+\lambda\nabla_{\mathbb{H}}u\right\vert ^{2}dV_{\mathbb{H}%
}.
\end{align*}
For $M_{\lambda}=e^{-\left(  \frac{r}{\sinh r}\right)  ^{\frac{(N-1)}{2}}%
\frac{r^{2}}{\lambda^{2}}}$ with $r=\rho(x)$, we have
\[
\nabla_{\mathbb{H}}M_{\lambda}=\dfrac{M_{\lambda}}{2\lambda^{2}}%
\dfrac{r^{\frac{(N+1)}{2}}}{(\sinh r)^{\frac{(N-1)}{2}}}\left(  (N-1)r\dfrac
{\cosh r}{\sinh r}-(N+3)\right)  \nabla_{\mathbb{H}}r,
\]
which in turn implies
\begin{align*}
|\nabla_{\mathbb{H}}M_{\lambda}|^{2}  &  =\dfrac{(N-1)^{2}}{4\lambda^{4}%
}M_{\lambda}^{2}\dfrac{r^{N+3}}{(\sinh r)^{N+1}}\cosh^{2}r+\dfrac{(N+3)^{2}%
}{4\lambda^{4}}M_{\lambda}^{2}\dfrac{r^{N+1}}{(\sinh r)^{N-1}}\\
&  -\dfrac{(N-1)(N+3)}{2\lambda^{4}}M_{\lambda}^{2}\dfrac{r^{N+2}}{(\sinh
r)^{N}}\cosh r.
\end{align*}
Therefore,
\begin{align*}
\lambda^{2}M_{\lambda}\left\vert \nabla_{\mathbb{H}}(uM_{\lambda}%
^{-1/2})\right\vert ^{2}  &  =\lambda^{2}|\nabla_{\mathbb{H}}u|^{2}%
+\dfrac{\lambda^{2}}{4}\dfrac{\left\vert u\right\vert ^{2}}{M_{\lambda}^{2}%
}|\nabla_{\mathbb{H}}M_{\lambda}|^{2}-\lambda^{2}\dfrac{u}{M_{\lambda}}%
\nabla_{\mathbb{H}}u\nabla_{\mathbb{H}}M_{\lambda}\\
&  +\left\vert \dfrac{ur\nabla_{\mathbb{H}}r}{\lambda}+\lambda\nabla
_{\mathbb{H}}u\right\vert ^{2}-\left\vert \dfrac{ur\nabla_{\mathbb{H}}%
r}{\lambda}+\lambda\nabla_{\mathbb{H}}u\right\vert ^{2},
\end{align*}
which allows us
\begin{align*}
&  \left\vert \dfrac{ur\nabla_{\mathbb{H}}r}{\lambda}+\lambda\nabla
_{\mathbb{H}}u\right\vert ^{2}=\lambda^{2}M_{\lambda}\left\vert \nabla
_{\mathbb{H}}(uM_{\lambda}^{-1/2})\right\vert ^{2}\\
&  -\dfrac{\left\vert u\right\vert ^{2}}{8\lambda^{2}}\left(  \dfrac
{(N-1)^{2}}{2}\dfrac{r^{N+3}}{(\sinh r)^{N+1}}\cosh^{2}r+\dfrac{(N+3)^{2}}%
{2}\dfrac{r^{N+1}}{(\sinh r)^{N-1}}-(N-1)(N+3)\dfrac{r^{N+2}}{(\sinh r)^{N}%
}\cosh r-8r^{2}\right) \\
&  +u\nabla_{\mathbb{H}}u\nabla_{\mathbb{H}}r\left(  \dfrac{N-1}{2}%
\dfrac{r^{\frac{(N+3)}{2}}}{(\sinh r)^{\frac{(N+1)}{2}}}\cosh r-\dfrac{N+3}%
{2}\dfrac{r^{\frac{(N+1)}{2}}}{(\sinh r)^{\frac{(N-1)}{2}}}+2r\right)  .
\end{align*}
Hence, from \eqref{DeficitIdt} and choosing $\lambda=\left(  \dfrac
{\int_{\mathbb{H}^{N}}\rho^{2}(x)|u|^{2}dV_{\mathbb{H}}}{\int_{\mathbb{H}^{N}%
}|\nabla_{\mathbb{H}}u|^{2}dV_{\mathbb{H}}}\right)  ^{1/4}$, we acquire
\[
\delta_{1,\mathbb{H}}(u)=\dfrac{1}{2}\int_{\mathbb{H}^{N}}\left\vert
\dfrac{ur\nabla_{\mathbb{H}}r}{\lambda}+\lambda\nabla_{\mathbb{H}}u\right\vert
^{2}dV_{\mathbb{H}},
\]
where
\begin{align*}
\delta_{1,\mathbb{H}}(u)  &  =\left(  \int_{\mathbb{H}^{N}}|\nabla
_{\mathbb{H}}u|^{2}dV_{\mathbb{H}}\right)  ^{1/2}\left(  \int_{\mathbb{H}^{N}%
}\rho^{2}(x)|u|^{2}dV_{\mathbb{H}}\right)  ^{1/2}\\
&  -\dfrac{1}{2}\int_{\mathbb{H}^{N}}\left(  N+(N-1)\dfrac{\rho(x)\cosh
\rho(x)-\ \sinh\rho(x)}{\sinh\rho(x)}\right)  |u|^{2}dV_{\mathbb{H}}.
\end{align*}
Therefore, by defining
\begin{align*}
\Tilde{\delta}_{1,\mathbb{H}}(u)  &  :=\delta_{1,\mathbb{H}}(u)+\dfrac
{1}{16\lambda^{2}}\int_{\mathbb{H}^{N}}\left\vert u\right\vert ^{2}\left(
\dfrac{(N-1)^{2}}{2}\dfrac{r^{N+3}}{(\sinh r)^{N+1}}\cosh^{2}r+\dfrac
{(N+3)^{2}}{2}\dfrac{r^{N+1}}{(\sinh r)^{N-1}}\right. \\
&  \left.  -(N-1)(N+3)\dfrac{r^{N+2}}{(\sinh r)^{N}}\cosh r-8r^{2}\right)
dV_{\mathbb{H}}\\
&  -\int_{\mathbb{H}^{N}}u\nabla_{\mathbb{H}}u\nabla_{\mathbb{H}}r\left(
\dfrac{N-1}{4}\dfrac{r^{\frac{(N+3)}{2}}}{(\sinh r)^{\frac{(N+1)}{2}}}\cosh
r-\dfrac{N+3}{4}\dfrac{r^{\frac{(N+1)}{2}}}{(\sinh r)^{\frac{(N-1)}{2}}%
}+r\right)  dV_{\mathbb{H}},
\end{align*}
we obtain
\begin{align*}
\Tilde{\delta}_{1,\mathbb{H}}(u)  &  =\dfrac{\lambda^{2}}{2}\int
_{\mathbb{H}^{N}}\left\vert \nabla_{\mathbb{H}}(uM_{\lambda}^{-1/2}%
)\right\vert ^{2}M_{\lambda}dV_{\mathbb{H}}\geq K\inf_{c}\int_{\mathbb{H}^{N}%
}|uM_{\lambda}^{-1/2}-c|^{2}M_{\lambda}dV_{\mathbb{H}}\\
&  \geq K\inf_{c}\int_{\mathbb{H}^{N}}\left\vert u-cM_{\lambda}^{1/2}%
\right\vert ^{2}dV_{\mathbb{H}}%
\end{align*}
for all $\lambda\geq\beta$, where $\beta$ is defined in Theorem \ref{3rdineq}.
\end{proof}

\begin{proof}
[Proof of Theorem \ref{finalstability}]Combining Corollary \ref{Cor 3.3} with
Theorem \ref{Cor 3.2} when $\beta$ is chosen as in Theorem \ref{3rdineq}, we
deduce the desired result.
\end{proof}

\section{Log-Sobolev inequality with Gaussian measure-Proof of Theorem
\ref{ThmlogS}}

\label{sct6}In this section, we will prove the log-Sobolev inequality. Let
$\beta>0$,\ $G=\int_{\mathbb{H}^{N}}e^{-\beta\rho^{2}(x)}dV_{\mathbb{H}}$ be
the normalization constant and $d\mu_{\mathbb{H}}:=\frac{e^{-\beta\rho^{2}%
(x)}}{G}dV_{\mathbb{H}}$ be a probability measure. We will modify the method
in \cite{HZ10} and first establish the following $U$-bound estimates

\begin{lemma}
\label{l5.1}Let $N\geq2$. There exist positive universal constants $C_{1},$
$D_{1}$ such that the following bound holds
\[
\int_{\mathbb{H}^{N}}|f|\rho\left(  x\right)  d\mu_{\mathbb{H}}\leq C_{1}%
\int_{\mathbb{H}^{N}}|\nabla_{\mathbb{H}}f|d\mu_{\mathbb{H}}+D_{1}%
\int_{\mathbb{H}^{N}}|f|d\mu_{\mathbb{H}}%
\]
for all $f$ such that $\int_{\mathbb{H}^{N}}|\nabla_{\mathbb{H}}%
f|d\mu_{\mathbb{H}}+\int_{\mathbb{H}^{N}}|f|d\mu_{\mathbb{H}}<\infty$.
\end{lemma}

\begin{proof}
Without loss of generality, we assume that $f\geq0$. First, we consider a
smooth function $f\geq0$ such that $f$ vanishes on the unit ball
$\{x\in\mathbb{H}^{N}:\rho(x)<1\}$. By using the Leibniz rule, we have
\begin{equation}
\int_{\mathbb{H}^{N}}(\nabla_{\mathbb{H}}f\cdot\nabla_{\mathbb{H}}%
\rho)e^{-\beta\rho^{2}(x)}dV_{\mathbb{H}}=\int_{\mathbb{H}^{N}}\nabla
_{\mathbb{H}}\rho\cdot\nabla_{\mathbb{H}}(fe^{-\beta\rho^{2}(x)}%
)dV_{\mathbb{H}}+2\beta\int_{\mathbb{H}^{N}}f\rho(x)|\nabla_{\mathbb{H}}%
\rho|^{2}e^{-\beta\rho^{2}(x)}dV_{\mathbb{H}}. \label{eq 0.1}%
\end{equation}
Since $\left\vert \nabla_{\mathbb{H}}\rho\right\vert =1$, we obtain%
\[
\int_{\mathbb{H}^{N}}(\nabla_{\mathbb{H}}f\cdot\nabla_{\mathbb{H}}%
\rho)e^{-\beta\rho^{2}(x)}dV_{\mathbb{H}}\leq\int_{\mathbb{H}^{N}}%
|\nabla_{\mathbb{H}}f||\nabla_{\mathbb{H}}\rho|e^{-\beta\rho^{2}%
(x)}dV_{\mathbb{H}}=\int_{\mathbb{H}^{N}}|\nabla_{\mathbb{H}}f|e^{-\beta
\rho^{2}(x)}dV_{\mathbb{H}}.
\]
Using the fact that
\[
\Delta_{\mathbb{H}}\rho=\frac{\partial^{2}}{\partial\rho^{2}}(\rho
)+(N-1)\coth\rho\frac{\partial}{\partial\rho}(\rho)=(N-1)\coth\rho
\leq(N-1)\coth(1)
\]
outside the unit ball, and that $f$ vanishes on the unit ball, we deduce that
\begin{align*}
\int_{\mathbb{H}^{N}}\nabla_{\mathbb{H}}\rho\cdot\nabla_{\mathbb{H}%
}(fe^{-\beta\rho^{2}(x)})dV_{\mathbb{H}}  &  =-\int_{\mathbb{H}^{N}}%
\Delta_{\mathbb{H}}\rho\cdot(fe^{-\beta\rho^{2}(x)})dV_{\mathbb{H}}\\
&  \geq-(N-1)\coth(1)\int_{\mathbb{H}^{N}}fe^{-\beta\rho^{2}(x)}%
dV_{\mathbb{H}}.
\end{align*}
Therefore, from \eqref{eq 0.1}, we get
\[
-(N-1)\coth(1)\int_{\mathbb{H}^{N}}fe^{-\beta\rho^{2}(x)}dV_{\mathbb{H}%
}+2\beta\int_{\mathbb{H}^{N}}f\rho(x)e^{-\beta\rho^{2}(x)}dV_{\mathbb{H}}%
\leq\int_{\mathbb{H}^{N}}|\nabla_{\mathbb{H}}f|e^{-\beta\rho^{2}%
(x)}dV_{\mathbb{H}},
\]
which implies that
\[
\int_{\mathbb{H}^{N}}f\rho(x)e^{-\beta\rho^{2}(x)}dV_{\mathbb{H}}\leq\dfrac
{1}{2\beta}\int_{\mathbb{H}^{N}}|\nabla_{\mathbb{H}}f|e^{-\beta\rho^{2}%
(x)}dV_{\mathbb{H}}+\dfrac{(N-1)\coth(1)}{2\beta}\int_{\mathbb{H}^{N}%
}fe^{-\beta\rho^{2}(x)}dV_{\mathbb{H}},
\]
as desired.\newline Next, we will work on $f$ not vanishing on the unit ball.
To deal with this case, we express
\[
f=\phi f+(1-\phi)f:=f_{0}+f_{1},
\]
where $\phi(x)=\min(1,\max(2-\rho(x),0)).$ Therefore, $|f_{1}|=|(1-\phi
)f|\leq|f|$ and
\begin{align*}
\int_{\mathbb{H}^{N}}|f|\rho(x)d\mu_{\mathbb{H}}  &  =\int_{\{x\in
\mathbb{H}^{N}:\rho(x)\leq2\}}|f|\rho(x)d\mu_{\mathbb{H}}+\int_{\{x\in
\mathbb{H}^{N}:\rho(x)>2\}}|f|\rho(x)d\mu_{\mathbb{H}}\\
&  \leq2\int_{\{x\in\mathbb{H}^{N}:\rho(x)\leq2\}}|f|d\mu_{\mathbb{H}}%
+\int_{\{x\in\mathbb{H}^{N}:\rho(x)>2\}}|f_{1}|\rho(x)d\mu_{\mathbb{H}}\\
&  \leq2\int_{\mathbb{H}^{N}}|f|d\mu_{\mathbb{H}}+\int_{\mathbb{H}^{N}}%
|f_{1}|\rho(x)d\mu_{\mathbb{H}}.
\end{align*}
Since $f_{1}$ vanishes on the unit ball, applying the previous case, we
acquire
\[
\int_{\mathbb{H}^{N}}|f|\rho(x)d\mu_{\mathbb{H}}\leq2\int_{\mathbb{H}^{N}%
}|f|d\mu_{\mathbb{H}}+C\int_{\mathbb{H}^{N}}|\nabla_{\mathbb{H}}f_{1}%
|d\mu_{\mathbb{H}}+D\int_{\mathbb{H}^{N}}|f_{1}|d\mu_{\mathbb{H}}.
\]
Moreover, since
\begin{align*}
|\nabla_{\mathbb{H}}f_{1}|=|\nabla_{\mathbb{H}}(f-\phi f)|=|(1-\phi
)\nabla_{\mathbb{H}}f-f\nabla_{\mathbb{H}}\phi|  &  \leq|\nabla_{\mathbb{H}%
}f|+|\nabla_{\mathbb{H}}\phi||f|\\
&  \leq|\nabla_{\mathbb{H}}f|+|\nabla_{\mathbb{H}}\rho||f|\\
&  =|\nabla_{\mathbb{H}}f|+|f|,
\end{align*}
we get
\[
\int_{\mathbb{H}^{N}}|f|\rho(x)d\mu_{\mathbb{H}}\leq(2+C+D)\int_{\mathbb{H}%
^{N}}|f|d\mu_{\mathbb{H}}+C\int_{\mathbb{H}^{N}}|\nabla_{\mathbb{H}}%
f|d\mu_{\mathbb{H}},
\]
which holds for any function $f$.
\end{proof}

\begin{lemma}
\label{l5.2}Let $N\geq2$. There exist positive universal constants $C_{2},$
$D_{2}$ such that the following bound holds%
\[
\int_{\mathbb{H}^{N}}\rho^{2}(x)\left\vert u\right\vert ^{2}d\mu_{\mathbb{H}%
}\leq C_{2}\int_{\mathbb{H}^{N}}\left\vert \nabla_{\mathbb{H}}u\right\vert
^{2}d\mu_{\mathbb{H}}+D_{2}\int_{\mathbb{H}^{N}}\left\vert u\right\vert
^{2}d\mu_{\mathbb{H}}%
\]
for all $u$ such that $\int_{\mathbb{H}^{N}}\left\vert \nabla_{\mathbb{H}%
}u\right\vert ^{2}d\mu_{\mathbb{H}}+\int_{\mathbb{H}^{N}}\left\vert
u\right\vert ^{2}d\mu_{\mathbb{H}}<\infty$.
\end{lemma}

\begin{proof}
Let us define $\rho_{1}(x)=\max(1,\rho(x))$. Then, we can rewrite $\rho
_{1}(x)=\dfrac{1+\rho(x)+|\rho(x)-1|}{2}$. By Lemma \ref{l5.1}, we can find
two positive universal constants $C$, $D>0$ such that
\[
\int_{\mathbb{H}^{N}}|f|\rho_{1}(x)d\mu_{\mathbb{H}}\leq C\int_{\mathbb{H}%
^{N}}|\nabla_{\mathbb{H}}f|d\mu_{\mathbb{H}}+D\int_{\mathbb{H}^{N}}%
|f|d\mu_{\mathbb{H}}.
\]
Now, let $f=|u|^{2}\rho_{1}$. We have
\begin{align*}
\int_{\mathbb{H}^{N}}|u|^{2}\rho_{1}^{2}(x)d\mu_{\mathbb{H}}  &
=\int_{\mathbb{H}^{N}}f\rho_{1}(x)d\mu_{\mathbb{H}}\\
&  \leq C\int_{\mathbb{H}^{N}}|\nabla_{\mathbb{H}}f|d\mu_{\mathbb{H}}%
+D\int_{\mathbb{H}^{N}}fd\mu_{\mathbb{H}}.
\end{align*}
Using the H\"older inequality and the Cauchy--Schwarz inequality, we obtain
\begin{align*}
\int_{\mathbb{H}^{N}}fd\mu_{\mathbb{H}}=\int_{\mathbb{H}^{N}}|u|^{2}\rho
_{1}d\mu_{\mathbb{H}}  &  =\int_{\mathbb{H}^{N}}|u||u|\rho_{1}d\mu
_{\mathbb{H}}\\
&  \leq\left(  \int_{\mathbb{H}^{N}}|u|^{2}d\mu_{\mathbb{H}}\right)
^{1/2}\left(  \int_{\mathbb{H}^{N}}|u|^{2}\rho_{1}^{2}d\mu_{\mathbb{H}%
}\right)  ^{1/2}\\
&  \leq\dfrac{\alpha^{2}}{2}\int_{\mathbb{H}^{N}}|u|^{2}d\mu_{\mathbb{H}%
}+\dfrac{1}{2\alpha^{2}}\int_{\mathbb{H}^{N}}|u|^{2}\rho_{1}^{2}%
d\mu_{\mathbb{H}}.
\end{align*}
Also, since $\left\vert \nabla\rho_{1}\right\vert \leq1$, we get
\begin{align*}
\int_{\mathbb{H}^{N}}|\nabla_{\mathbb{H}}f|d\mu_{\mathbb{H}}  &  \leq
2\int_{\mathbb{H}^{N}}|u||\nabla_{\mathbb{H}}u|\rho_{1}d\mu_{\mathbb{H}}%
+\int_{\mathbb{H}^{N}}|u|^{2}|\nabla_{\mathbb{H}}\rho_{1}|d\mu_{\mathbb{H}}\\
&  \leq2\int_{\mathbb{H}^{N}}|u||\nabla_{\mathbb{H}}u|\rho_{1}d\mu
_{\mathbb{H}}+\int_{\mathbb{H}^{N}}|u|^{2}d\mu_{\mathbb{H}}%
\end{align*}
and
\begin{align*}
2\int_{\mathbb{H}^{N}}|u||\nabla_{\mathbb{H}}u|\rho_{1}d\mu_{\mathbb{H}}  &
\leq2\left(  \int_{\mathbb{H}^{N}}|\nabla_{\mathbb{H}}u|^{2}d\mu_{\mathbb{H}%
}\right)  ^{1/2}\left(  \int_{\mathbb{H}^{N}}|u|^{2}\rho_{1}^{2}%
d\mu_{\mathbb{H}}\right)  ^{1/2}\\
&  \leq\gamma^{2}\int_{\mathbb{H}^{N}}|\nabla_{\mathbb{H}}u|^{2}%
d\mu_{\mathbb{H}}+\dfrac{1}{\gamma^{2}}\int_{\mathbb{H}^{N}}|u|^{2}\rho
_{1}^{2}d\mu_{\mathbb{H}}.
\end{align*}
Combining all inequalities above, we deduce that%
\begin{align*}
&  \int_{\mathbb{H}^{N}}\rho^{2}(x)\left\vert u\right\vert ^{2}d\mu
_{\mathbb{H}}\\
&  \leq\int_{\mathbb{H}^{N}}|u|^{2}\rho_{1}^{2}(x)d\mu_{\mathbb{H}}\\
&  \leq C\left(  \gamma^{2}\int_{\mathbb{H}^{N}}|\nabla_{\mathbb{H}}u|^{2}%
d\mu_{\mathbb{H}}+\dfrac{1}{\gamma^{2}}\int_{\mathbb{H}^{N}}|u|^{2}\rho
_{1}^{2}d\mu_{\mathbb{H}}+\int_{\mathbb{H}^{N}}|u|^{2}d\mu_{\mathbb{H}}\right)
\\
&  +D\left(  \dfrac{\alpha^{2}}{2}\int_{\mathbb{H}^{N}}|u|^{2}d\mu
_{\mathbb{H}}+\dfrac{1}{2\alpha^{2}}\int_{\mathbb{H}^{N}}|u|^{2}\rho_{1}%
^{2}d\mu_{\mathbb{H}}\right) \\
&  =C\gamma^{2}\int_{\mathbb{H}^{N}}|\nabla_{\mathbb{H}}u|^{2}d\mu
_{\mathbb{H}}+\left(  \dfrac{C}{\gamma^{2}}+\dfrac{D}{2\alpha^{2}}\right)
\int_{\mathbb{H}^{N}}|u|^{2}\rho_{1}^{2}d\mu_{\mathbb{H}}+\left(
C+\dfrac{D\alpha^{2}}{2}\right) \int_{\mathbb{H}^{N}}|u|^{2}d\mu_{\mathbb{H}}.
\end{align*}
Therefore, by choosing $\alpha,\gamma$ large enough, we can find two positive
universal constants $K,$ $L$ such that
\[
\int_{\mathbb{H}^{N}}\rho^{2}(x)\left\vert u\right\vert ^{2}d\mu_{\mathbb{H}%
}\leq C_{2}\int_{\mathbb{H}^{N}}\left\vert \nabla_{\mathbb{H}}u\right\vert
^{2}d\mu_{\mathbb{H}}+D_{2}\int_{\mathbb{H}^{N}}\left\vert u\right\vert
^{2}d\mu_{\mathbb{H}}.
\]

\end{proof}

We are now ready to provide a proof for Theorem \ref{ThmlogS}.

\begin{proof}
[Proof of Theorem \ref{ThmlogS}]Assume that $\int_{\mathbb{H}^{N}}\left\vert
u\right\vert ^{2}d\mu_{\mathbb{H}}=A^{2}>0$. Let $v=\frac{ue^{-\frac{\beta}%
{2}\rho^{2}(x)}}{A\sqrt{G}}$. Then $\int_{\mathbb{H}^{N}}\left\vert
v\right\vert ^{2}dV_{\mathbb{H}}=1$.

If $N=2$, then by Jensen's inequality and the Sobolev inequality $\left(
\int_{\mathbb{H}^{2}}\left\vert v\right\vert ^{4}dV_{\mathbb{H}}\right)
^{\frac{1}{2}}\leq S_{2}\int_{\mathbb{H}^{2}}\left\vert \nabla_{\mathbb{H}%
}v\right\vert ^{2}dV_{\mathbb{H}}$, we have
\begin{align*}
\int_{\mathbb{H}^{2}}\left\vert v\right\vert ^{2}\log\left\vert v\right\vert
^{2}dV_{\mathbb{H}}  &  \leq\log\left(  \int_{\mathbb{H}^{N}}\left\vert
v\right\vert ^{4}dV_{\mathbb{H}}\right) \\
&  \leq2\log\left(  S_{2}\int_{\mathbb{H}^{2}}\left\vert \nabla_{\mathbb{H}%
}v\right\vert ^{2}dV_{\mathbb{H}}\right) \\
&  \leq2\log S_{2}+2\int_{\mathbb{H}^{2}}\left\vert \nabla_{\mathbb{H}%
}v\right\vert ^{2}dV_{\mathbb{H}}.
\end{align*}

If $N\geq3$, by using Jensen's inequality, we obtain that%
\begin{align*}
\int_{\mathbb{H}^{N}}\left\vert v\right\vert ^{2}\log\left\vert v\right\vert
^{2}dV_{\mathbb{H}}  &  =\frac{N-2}{2}\int_{\mathbb{H}^{N}}\log\left(
\left\vert v\right\vert ^{\frac{4}{N-2}}\right)  \left\vert v\right\vert
^{2}dV_{\mathbb{H}}\\
&  \leq\frac{N-2}{2}\log\left(  \int_{\mathbb{H}^{N}}\left\vert v\right\vert
^{\frac{4}{N-2}}\left\vert v\right\vert ^{2}dV_{\mathbb{H}}\right) \\
&  =\frac{N}{2}\log\left(  \int_{\mathbb{H}^{N}}\left\vert v\right\vert
^{\frac{2N}{N-2}}dV_{\mathbb{H}}\right)  ^{\frac{N-2}{N}}.
\end{align*}
By applying the Poincar\'{e}-Sobolev inequalities on the hyperbolic space
\[
S_{N}\int_{\mathbb{H}^{N}}\left\vert \nabla_{\mathbb{H}}v\right\vert
^{2}dV_{\mathbb{H}}\geq\left(  \int_{\mathbb{H}^{N}}\left\vert v\right\vert
^{\frac{2N}{N-2}}dV_{\mathbb{H}}\right)  ^{\frac{N-2}{N}}%
\]
we get%
\[
\int_{\mathbb{H}^{N}}\left\vert v\right\vert ^{2}\log\left\vert v\right\vert
^{2}dV_{\mathbb{H}}\leq\frac{N}{2}\log\left(  S_{N}\int_{\mathbb{H}^{N}%
}\left\vert \nabla_{\mathbb{H}}v\right\vert ^{2}dV_{\mathbb{H}}\right)
\leq\frac{N}{2}\log S_{N}+\frac{N}{2}\int_{\mathbb{H}^{N}}\left\vert
\nabla_{\mathbb{H}}v\right\vert ^{2}dV_{\mathbb{H}}%
\]
In particular, for $N\geq2$, we have
\[
\int_{\mathbb{H}^{N}}\left\vert v\right\vert ^{2}\log\left\vert v\right\vert
^{2}dV_{\mathbb{H}}\leq C_{1}+C_{2}\int_{\mathbb{H}^{N}}\left\vert
\nabla_{\mathbb{H}}v\right\vert ^{2}dV_{\mathbb{H}}.
\]
Therefore%
\begin{align*}
&  \int_{\mathbb{H}^{N}}\left\vert u\right\vert ^{2}\log\left(  \frac
{\left\vert u\right\vert ^{2}}{\int_{\mathbb{H}^{N}}\left\vert u\right\vert
^{2}d\mu_{\mathbb{H}}}\right)  d\mu_{\mathbb{H}}\\
&  =\int_{\mathbb{H}^{N}}\left\vert u\right\vert ^{2}\log\left(
\frac{\left\vert u\right\vert ^{2}e^{-\beta\rho^{2}(x)}}{A^{2}G}\frac
{G}{e^{-\beta\rho^{2}(x)}}\right)  d\mu_{\mathbb{H}}\\
&  =\int_{\mathbb{H}^{N}}\left\vert u\right\vert ^{2}\log\left(
\frac{\left\vert u\right\vert ^{2}e^{-\beta\rho^{2}(x)}}{A^{2}G}\right)
d\mu_{\mathbb{H}}+\beta\int_{\mathbb{H}^{N}}\rho^{2}(x)\left\vert u\right\vert
^{2}d\mu_{\mathbb{H}}+\log G\int_{\mathbb{H}^{N}}\left\vert u\right\vert
^{2}d\mu_{\mathbb{H}}\\
&  =A^{2}\int_{\mathbb{H}^{N}}\left\vert v\right\vert ^{2}\log\left\vert
v\right\vert ^{2}dV_{\mathbb{H}}+\beta\int_{\mathbb{H}^{N}}\rho^{2}%
(x)\left\vert u\right\vert ^{2}d\mu_{\mathbb{H}}+\log G\int_{\mathbb{H}^{N}%
}\left\vert u\right\vert ^{2}d\mu_{\mathbb{H}}\\
&  \leq C_{1}A^{2}+C_{2}A^{2}\int_{\mathbb{H}^{N}}\left\vert \nabla
_{\mathbb{H}}v\right\vert ^{2}dV_{\mathbb{H}}+\beta\int_{\mathbb{H}^{N}}%
\rho^{2}(x)\left\vert u\right\vert ^{2}d\mu_{\mathbb{H}}+\log G\int
_{\mathbb{H}^{N}}\left\vert u\right\vert ^{2}d\mu_{\mathbb{H}}\\
&  =C_{1}\int_{\mathbb{H}^{N}}\left\vert u\right\vert ^{2}d\mu_{\mathbb{H}%
}+C_{2}\int_{\mathbb{H}^{N}}\left\vert \nabla_{\mathbb{H}}\left(
\frac{ue^{-\frac{\beta}{2}\rho^{2}(x)}}{\sqrt{G}}\right)  \right\vert
^{2}dV_{\mathbb{H}}+\beta\int_{\mathbb{H}^{N}}\rho^{2}(x)\left\vert
u\right\vert ^{2}d\mu_{\mathbb{H}}+\log G\int_{\mathbb{H}^{N}}\left\vert
u\right\vert ^{2}d\mu_{\mathbb{H}}.
\end{align*}
Note that%
\begin{align*}
\int_{\mathbb{H}^{N}}\left\vert \nabla_{\mathbb{H}}\left(  \frac
{ue^{-\frac{\beta}{2}\rho^{2}(x)}}{\sqrt{G}}\right)  \right\vert
^{2}dV_{\mathbb{H}}  &  =\int_{\mathbb{H}^{N}}\left\vert \nabla_{\mathbb{H}%
}u-\beta u\rho(x)\nabla_{\mathbb{H}}\rho(x)\right\vert ^{2}d\mu_{\mathbb{H}}\\
&  \leq2\int_{\mathbb{H}^{N}}\left\vert \nabla_{\mathbb{H}}u\right\vert
^{2}d\mu_{\mathbb{H}}+2\beta^{2}\int_{\mathbb{H}^{N}}\rho^{2}(x)\left\vert
u\right\vert ^{2}d\mu_{\mathbb{H}}.
\end{align*}
Also, by the $U$-bound estimate (Lemma \ref{l5.2}), we have
\[
\int_{\mathbb{H}^{N}}\rho^{2}(x)\left\vert u\right\vert ^{2}d\mu_{\mathbb{H}%
}\leq C_{2}\int_{\mathbb{H}^{N}}\left\vert \nabla_{\mathbb{H}}u\right\vert
^{2}d\mu_{\mathbb{H}}+D_{2}\int_{\mathbb{H}^{N}}\left\vert u\right\vert
^{2}d\mu_{\mathbb{H}}.
\]
Therefore%
\[
\int_{\mathbb{H}^{N}}\left\vert u\right\vert ^{2}\log\left(  \frac{\left\vert
u\right\vert ^{2}}{\int_{\mathbb{H}^{N}}\left\vert u\right\vert ^{2}%
d\mu_{\mathbb{H}}}\right)  d\mu_{\mathbb{H}}\leq C\int_{\mathbb{H}^{N}%
}\left\vert \nabla_{\mathbb{H}}u\right\vert ^{2}d\mu_{\mathbb{H}}%
+D\int_{\mathbb{H}^{N}}\left\vert u\right\vert ^{2}d\mu_{\mathbb{H}}%
\]
for some positive universal constants $C$ and $D$. In particular, we have%
\begin{align*}
&  \int_{\mathbb{H}^{N}}\left\vert u-\int_{\mathbb{H}^{N}}ud\mu_{\mathbb{H}%
}\right\vert ^{2}\log\left(  \frac{\left\vert u-\int_{\mathbb{H}^{N}}%
ud\mu_{\mathbb{H}}\right\vert ^{2}}{\int_{\mathbb{H}^{N}}\left\vert
u-\int_{\mathbb{H}^{N}}ud\mu_{\mathbb{H}}\right\vert ^{2}d\mu_{\mathbb{H}}%
}\right)  d\mu_{\mathbb{H}}\\
&  \leq C\int_{\mathbb{H}^{N}}\left\vert \nabla_{\mathbb{H}}\left(
u-\int_{\mathbb{H}^{N}}ud\mu_{\mathbb{H}}\right)  \right\vert ^{2}%
d\mu_{\mathbb{H}}+D\int_{\mathbb{H}^{N}}\left\vert u-\int_{\mathbb{H}^{N}%
}ud\mu_{\mathbb{H}}\right\vert ^{2}d\mu_{\mathbb{H}}\\
&  =C\int_{\mathbb{H}^{N}}\left\vert \nabla_{\mathbb{H}}u\right\vert ^{2}%
d\mu_{\mathbb{H}}+D\int_{\mathbb{H}^{N}}\left\vert u-\int_{\mathbb{H}^{N}%
}ud\mu_{\mathbb{H}}\right\vert ^{2}d\mu_{\mathbb{H}}.
\end{align*}
Finally, using the following estimate (see \cite{BZ05})%
\begin{align*}
&  \int_{\mathbb{H}^{N}}\left\vert u\right\vert ^{2}\log\left(  \frac
{\left\vert u\right\vert ^{2}}{\int_{\mathbb{H}^{N}}\left\vert u\right\vert
^{2}d\mu_{\mathbb{H}}}\right)  d\mu_{\mathbb{H}}\\
&  \leq\int_{\mathbb{H}^{N}}\left\vert u-\int_{\mathbb{H}^{N}}ud\mu
_{\mathbb{H}}\right\vert ^{2}\log\left(  \frac{\left\vert u-\int
_{\mathbb{H}^{N}}ud\mu_{\mathbb{H}}\right\vert ^{2}}{\int_{\mathbb{H}^{N}%
}\left\vert u-\int_{\mathbb{H}^{N}}ud\mu_{\mathbb{H}}\right\vert ^{2}%
d\mu_{\mathbb{H}}}\right)  d\mu_{\mathbb{H}}+2\int_{\mathbb{H}^{N}}\left\vert
u-\int_{\mathbb{H}^{N}}ud\mu_{\mathbb{H}}\right\vert ^{2}d\mu_{\mathbb{H}},
\end{align*}
and the Poincar\'{e} inequality (Corollary \ref{weightedPoinIneq})
\[
K\left(  \beta\right)  \int_{\mathbb{H}^{N}}\left\vert u-\int_{\mathbb{H}^{N}%
}ud\mu_{\mathbb{H}}\right\vert ^{2}d\mu_{\mathbb{H}}\leq\int_{\mathbb{H}^{N}%
}\left\vert \nabla_{\mathbb{H}}u\right\vert ^{2}d\mu_{\mathbb{H}},
\]
we obtain the log-Sobolev inequality%
\[
\int_{\mathbb{H}^{N}}\left\vert u\right\vert ^{2}\log\left(  \frac{\left\vert
u\right\vert ^{2}}{\int_{\mathbb{H}^{N}}\left\vert u\right\vert ^{2}%
d\mu_{\mathbb{H}}}\right)  d\mu_{\mathbb{H}}\leq\left(  C+\frac{D+2}{K\left(
\beta\right)  }\right)  \int_{\mathbb{H}^{N}}\left\vert \nabla_{\mathbb{H}%
}u\right\vert ^{2}d\mu_{\mathbb{H}}.
\]

\end{proof}

\noindent\textbf{Acknowledgement:} A. Do and G. Lu were partially supported by
grants from the Simons Foundation and a Simons Fellowship. D. Ganguly was
partially supported by the SERB MATRICS (MTR/2023/000331). N. Lam was
partially supported by an NSERC Discovery Grant.


\end{document}